\documentclass[a4paper]{amsart}

\usepackage{a4wide}
\usepackage{amsmath,amssymb,amsthm,xy,mathrsfs,caption} 
\usepackage[usenames]{color}

\usepackage{graphicx}
\newcommand{\vcentered}[1]{\begingroup
\setbox0=\hbox{#1}%
\parbox{\wd0}{\box0}\endgroup}

\usepackage{wrapfig}

\newcounter{kommentar}

\numberwithin{equation}{section}

\def\l{\lambda} \def\s{\sigma} \def\R{\mathbb{R}}
\def\I{\mathbb{I}}

\def\<{\langle} \def\>{\rangle}

\def\inv{^{-1}}

\renewcommand\ge{\geqslant}
\renewcommand\le{\leqslant}
\newcommand\ie{i.e.}

\newcommand\eg{e.g.}

\newcommand\forget[1]{}

\DeclareMathOperator{\End}{End}
\DeclareMathOperator{\Aut}{Aut}

\renewcommand{\Im}{\mathop{\rm Im}\nolimits}

\newcommand{\Ker}{\mathop{\rm Ker}\nolimits}

\renewcommand{\L}{\Lambda}

\newenvironment{smatrix}{\left(\begin{smallmatrix}}{\end{smallmatrix}\right)}
\newcommand\smatr[1]{\begin{smatrix}#1\end{smatrix}}

\newcommand\dm[1]{\mathop{\mathcal{D}^\mathrm{b}}(#1)}
\newcommand{\dml}{\dm{\L}}

\newcommand{\modC}{\mod\mathcal{C}}
\newcommand{\modCG}{\mod(\mathcal{C}/G)}

\DeclareMathOperator{\Supp}{\operatorname{Supp}}

\DeclareMathOperator{\rad}{\operatorname{rad}}

\DeclareMathOperator{\Ext}{Ext}
\renewcommand{\mod}{\mathop{\rm mod}\nolimits}
\DeclareMathOperator{\Mod}{Mod}
\DeclareMathOperator{\stmod}{\underline{mod}}
\DeclareMathOperator{\add}{add}
\DeclareMathOperator{\Hom}{Hom}

\DeclareMathOperator{\stend}{\underline{End}}
\newcommand{\gldim}{\mathop{{\rm gl.dim}}\nolimits}
\newcommand{\injdim}{\mathop{{\rm inj.dim}}\nolimits}
\newcommand{\projdim}{\mathop{{\rm proj.dim}}\nolimits}
\DeclareMathOperator{\proj}{proj}
\DeclareMathOperator{\inj}{inj}
\DeclareMathOperator{\ind}{ind}
\DeclareMathOperator{\lcm}{lcm}
\DeclareMathOperator{\RF}{RF}

\newtheorem{thm}{Theorem}[section]
\newtheorem{lma}[thm]{Lemma}
\newtheorem{prop}[thm]{Proposition}
\newtheorem{cor}[thm]{Corollary}

\newtheorem{question}[thm]{Question}
\newtheorem{conj}[thm]{Conjecture}
\theoremstyle{definition}
\newtheorem{dfn}[thm]{Definition}
\theoremstyle{remark}
\newtheorem{rmk}[thm]{Remark}
\theoremstyle{definition}
\newtheorem{ex}[thm]{Example}

\renewcommand{\theenumi}{\alph{enumi}}

\xyoption{all}

\newcommand{\U}{\mathcal{U}}
\newcommand{\C}{\mathcal{C}}
\newcommand{\V}{\mathcal{V}}
\newcommand{\TT}{\mathcal{T}}
\def\Z{\mathbb{Z}}

\title{$d$-representation-finite self-injective algebras}
\author{Erik Darp\"o}
\author{Osamu Iyama}
\address{Graduate School of Mathematics, Nagoya University, Furocho, Chikusaku, Nagoya, Japan}
\email[Darp\"o]{darpo@math.nagoya-u.ac.jp}
\email[Iyama]{iyama@math.nagoya-u.ac.jp}
\urladdr[Iyama]{http://www.math.nagoya-u.ac.jp/~iyama/}
\thanks{The first author was partially supported by JSPS Grant-in-Aid for Scientific
  Research (C) 18K03238. The second author was partially supported by JSPS Grant-in-Aid for Scientific Research (B) 16H03923, (C) 23540045 and (S) 15H05738, and both authors were partially supported by JSPS Grant-in-Aid for Scientific Research 11F01753.}

\thanks{{\em Mathematics Subject Classification 2010:} 16G10 (16D50, 16G70, 18A25, 18E30)}
\thanks{{\em Key words and phrases:} self-injective algebra, repetitive category, Galois
  covering, $d$-representation-finite algebra, $d$-cluster-tilting,
  Auslander--Reiten theory.} 

\begin{document}
\date{}

\begin{abstract}
In this paper, we initiate the study of higher-dimensional Auslander--Reiten theory
of self-injective algebras.
We give a systematic construction of (weakly) $d$-representation-finite
self-injective algebras as orbit algebras of the repetitive categories of algebras
of finite global dimension satisfying a certain finiteness condition for the Serre
functor. The condition holds, in particular, for all fractionally Calabi-Yau algebras of
global dimension at most $d$.
This generalizes Riedtmann's classical construction of rep\-re\-sen\-ta\-tion-finite
self-injective algebras.
Our method is based on an adaptation of Gabriel's covering theory for $k$-linear
categories to the setting of higher-dimensional Auslander--Reiten theory.

Applications include $n$-fold trivial extensions and (classical and higher) preprojective algebras, which
are shown to be $d$-representation-finite in many cases. We also get a complete
classification of all $d$-representation-finite self-injective Nakayama algebras for
arbitrary $d$. 
\end{abstract}
\maketitle 
\tableofcontents

\section{Introduction}
In representation theory of finite-dimensional algebras, one of the fundamental problems
is to classify re\-pre\-sen\-ta\-tion-finite algebras, that is, algebras having only
finitely many indecomposable modules up to isomorphism.
A famous result by Gabriel \cite{gabriel} characterizes
representation-finite hereditary algebras over algebraically closed fields as path
algebras of Dynkin quivers. 
Another well-known result is the classification of representation-finite
self-injective algebras, due chiefly to Riedtmann
\cite{BLR,riedtmann80,riedtmann83} (see also \cite{hw83,waschbusch,waschbusch2}).
Her result, in characteristic different from $2$,
can be summarized as follows (see, \eg, \cite[Theorem~3.5]{skowronski06}, or
\cite[Theorem~8.1]{sy08}\forget{ and the sentence following it}):

\begin{thm}\label{riedtmann}
\begin{enumerate}
\item Let $\L$ be a tilted algebra of Dynkin type and $\phi$ an admissible automorphism
  of the repetitive category $\widehat\L$ of $\L$.
  Then $\widehat\L/\phi$ is a representation-finite self-injective algebra.
\item Every representation-finite self-injective algebra over an algebraically
closed field of characteristic different from $2$ is obtained in this way.
\end{enumerate}
\end{thm}

\noindent
An autoequivalence $\phi$ of a Krull--Schmidt category $\C$ is said to be
\emph{admissible} if $\phi^i$ acts freely on isomorphism classes of indecomposable objects
in $\C$ for all $i>0$.
For $\C=\widehat{\L}$, this is equivalent to $\widehat\L/\phi$ being a finite-dimensional
algebra (see Lemma~\ref{admissibility}).

After Gabriel's and Riedtmann's seminal results, the theory of representation-finite
algebras has had two culmination points.
One is the paper \cite{bgrs85} by Bautista--Gabriel--Roiter--Salmeron showing that it is a
finite problem, concerning the combinatorics of ray categories, to determine whether or
not a given algebra over an algebraically closed field is representation finite 
(see also \cite[Chapters~13, 14]{gabriel-roiter}).
The other one is the characterization of Auslander--Reiten quivers of
representation-finite algebras by Igusa--Todorov \cite{it84} and Brenner \cite{br86}
(see also \cite{iyama05}).
In view of Auslander's bijection \cite{a71} between representation-finite algebras and
Auslander algebras (that is, algebras of global dimension at most two and dominant
dimension at least two), the latter result can be seen as a structure theorem for
Auslander algebras.

Seeking to generalize Auslander's result, one is naturally led to consider
\emph{$d$-Auslander-algebras}, i.e., algebras of global dimension at most $d+1$ and
dominant dimension at least $d+1$. Such algebras correspond bijectively to equivalence
classes of so-called \emph{$d$-cluster-tilting} modules \cite{iyama07b} (see
Definition~\ref{CTdef} below).
The question thus occurs if there are generalizations of the fundamental results about
representation-finite algebras to algebras possessing a $d$-cluster-tilting module, 
that is, to \emph{$d$-rep\-re\-sen\-ta\-tion-finite} algebras.
Significant progress in this and related directions has been made in recent years; see, for
example, 
\cite{ABa,asai18,egl19,gls06,gls07,guo13,hi11a,hi11b,HIO,HZ,io11,io13,IW,jasso15a,jasso16,JK,Jo,Miz,OT,pasquali17}.
Amongst all $d$-Auslander-algebras, the ones corresponding to $d$-cluster-tilting
modules of algebras that are self-injective are characterized by the property that the
class of projective-injective modules is closed under the Nakayama functor.
The present paper provides new tools to study this class of algebras,
  through the corresponding self-injective algebras and their $d$-cluster-tilting
  modules.

Our aim is to generalize Riedtmann's construction from the point
of view of higher-dimensional Auslander--Reiten theory.
The fundamental idea is to construct $d$-representation-finite
self-injective algebras as orbit algebras of the repetitive categories of certain algebras $\L$
of finite global dimension, called \emph{$\nu_d$-finite} algebras. 
This class of algebras includes $d$-representation-finite algebras of global dimension $d$
(which are a higher-dimensional analogue of representation-finite hereditary algebras)
and, more generally, (twisted) fractionally Calabi--Yau algebras. 
The bounded derived category $\dml$ of a $\nu_d$-finite algebra $\L$ has a
$d$-cluster-tilting subcategory $\U$ of a certain nice form --
which we call an \emph{orbital} $d$-cluster-tilting subcategory.

Taking the orbit algebra of $\widehat{\L}$ with respect to an automorphism $\phi$ that
preserves $\U$ gives a $d$-representation-finite
self-injective algebra $\widehat{\L}/\phi$ (see Theorem~\ref{basic corollary}).
Large classes of self-injective algebras, including many instances of $n$-fold trivial
extension algebras (Corollaries \ref{fracCY}, \ref{trivext}, \ref{trivrf}) and higher
preprojective algebras (Corollary \ref{preproj}), can be shown to be
$d$-representation-finite using this construction.
We apply these results to give examples amongst canonical algebras of tubular type and
tensor products of Dynkin quivers. Algebras of tame as well as wild representation type
occur in these constructions. 
It is known by \cite{eh08} that any module of a $d$-re\-pre\-sen\-ta\-tion-fi\-nite self-injective
algebra has complexity at most one, i.e., the dimensions of the terms in a minimal
projective resolution of such a module are bounded from above by a constant.
The results in this paper give an explicit construction of self-injective algebras with
this rather particular property.

In the formulation of Riedtmann's classification in terms of orbit algebras, Gabriel's
covering theory \cite{gabriel2} plays an important role. 
Let $k$ be a field. A central idea in this theory is to extend the notion of
representation-finiteness from finite-dimensional algebras to $k$-linear categories, including
repetitive categories, through the concept of local representation-finiteness.
Crucially, Gabriel proved that local representation-finiteness is preserved under taking
orbit categories. 
To develop an analogue of this theory, we introduce the notion of \emph{local
$d$-representation-finiteness} (see Definition~\ref{lrfdef}), 
which is a categorical correspondent to $d$-representation-finiteness in algebras.
One of our main results, Corollary~\ref{bijection}, is that local
$d$-representation-finiteness is preserved under taking orbit categories satisfying a
natural invariance property.
This, together with the observation that orbital $d$-cluster-tilting subcategories of
$\dml$ are locally bounded, provides the foundation for the results quoted in the
previous paragraph.
Our strategy can be summarized as follows.
{\small\[\begin{xy}
(30,32)*{\L},
(70,32)*{\widehat{\L}},
(109,32)*{\widehat{\L}/\phi},
(0,20)*\txt{Classical\\ setting},
(30,20)*+[F:<5pt>]{\txt{representation-finite\\ hereditary $k$-algebra\\ $\cap$\\
representation-finite\\ tilted $k$-algebra}}="1",
(70,20)*+[F:<5pt>]{\txt{locally\\ representation-finite\\ $k$-linear category}}="2",
(109,20)*+[F:<5pt>]{\txt{representation-finite\\ self-injective\\ $k$-algebra}}="3",
(0,0)*\txt{Our setting},
(30,0)*+[F:<5pt>]{\txt{$d$-representation-finite\\ $k$-algebra of $\gldim$ $d$\\ $\cap$\\
$\nu_d$-finite algebra}}="4",
(70,0)*+[F:<5pt>]{\txt{locally\\ $d$-representation-finite\\ $k$-linear category}}="5",
(109,0)*+[F:<5pt>]{\txt{$d$-representation-finite\\ self-injective\\ $k$-algebra}}="6",
\ar@{~>}"1";"2"
\ar@{~>}"2";"3"
\ar@{~>}"4";"5"
\ar@{~>}"5";"6"
\end{xy}\]}

The notion of cluster tilting gives a strong link between higher-dimensional
Auslander--Reiten theory on one hand, and categorification of Fomin and Zelevinsky's
cluster algebras on the other. 
In our context, this connection comes into play in 
Section~\ref{section: higher preprojective} through the use of quivers with potential, 
and in Section~\ref{further}, where we use cluster categories to infer the existence of
non-orbital $2$-cluster-tilting subcategories in the derived categories of certain
algebras of global dimension $2$.

We remark that some of the results in the present paper have been applied in the
  construction of higher Nakayama algebras by Jasso and K\"ulshammer \cite{JK}.

This paper is organized as follows.
In Section~\ref{basicconstruction}--\ref{galois coverings}, main results are
stated. Section~\ref{basicconstruction} contains our generalization of
Riedtmann's construction, which in Section~\ref{section: applications} is applied to show
$d$-re\-pre\-sen\-ta\-tion-fin\-ite\-ness of some important classes of algebras.
In Section~\ref{galois coverings} we give a general result saying $d$-cluster-tilting
subcategories are, under certain conditions, preserved by Galois coverings. 
Section~\ref{section: notation} contains explanations of concepts and notation used in the
paper. 
Proofs of the results in Section~\ref{ourresults} are given in Section~\ref{proofs}.
In Section~\ref{exandappl}, we further investigate some examples and applications of the
general theory, while Section~\ref{section: nakayama} features a complete
  characterization of all $d$-re\-pre\-sen\-ta\-tion-fin\-ite self-injective Nakayama algebras.
We conclude by posing some open questions, together with a few partial results, in
Section~\ref{further}.

\section{Our results} \label{ourresults}

Throughout this article, $d$ is a positive integer, $k$ a field, and $\L$ a
finite-dimensional $k$-algebra of finite global dimension.
While most of the theory is valid for general fields, some results require additional
assumptions. This will then be stated at each such instance.
In order not to make the introductory parts of this paper overly technical, the
definitions of some of the concepts used in this section have been postponed to
Section~2.4.

We start by defining the fundamental concepts of our investigation.
\begin{dfn} \label{CTdef}
 Let $\C$ be an abelian or a triangulated category, and $A$ a
  finite-dimensional $k$-algebra.
\begin{enumerate} \renewcommand{\labelenumi}{(\roman{enumi})} 
\item  A full subcategory $\U$ of $\C$
  is \emph{$d$-cluster-tilting} if it is functorially finite in $\C$ and
  \begin{eqnarray*}
    \U&=&\{X\in\C\mid\Ext_{\C}^i(\U,X)=0\ \mbox{ for all }\ 1\le i\le d-1\}\\
    &=&\{X\in\C\mid\Ext_{\C}^i(X,\U)=0\ \mbox{ for all }\ 1\le i\le d-1\}.
  \end{eqnarray*}
\item  A finitely-generated module $M\in\mod A$ is called a \emph{$d$-cluster-tilting
  module} if its additive closure $\add M$ forms a $d$-cluster-tilting subcategory of
  $\mod A$. 
\item
  The algebra $A$ is said to be \emph{$d$-re\-pre\-sen\-ta\-tion-fin\-ite} if it has a
  $d$-cluster-tilting module.
\end{enumerate}
\end{dfn}
In the case $d=1$, the classical situation is recovered: 
a $1$-cluster-tilting $A$-module is nothing but an additive generator of $\mod A$, and an
algebra is $1$-re\-pre\-sen\-ta\-tion-fin\-ite if and only if it is representation finite in the classical sense.
It is easy to see that if $\mathcal{C}$ is a Frobenius category then the
  natural functor $\mathcal{C}\to\underline{\mathcal{C}}$ induces a bijection between
  $d$-cluster-tilting subcategories of $\mathcal{C}$ and $\underline{\mathcal{C}}$ respectively.
Note that, in contrast with several earlier papers (\eg, \cite{hi11a,hi11b,io11,io13}) we
do not assume $d$-re\-pre\-sen\-ta\-tion-fin\-ite algebras to have global dimension at most $d$.

\begin{dfn} \label{lrfdef}
  A locally bounded $k$-linear Krull-Schmidt category $\C$ is 
  \emph{locally $d$-re\-pre\-sen\-ta\-tion-fin\-ite} if $\mod\C$ has a locally bounded 
  $d$-cluster-tilting subcategory.
\end{dfn}

We remark that a $d$-cluster-tilting subcategory $\U$ of $\mod\C$ is locally bounded if,
and only if, for any $x\in\C$ there exist only finitely many isomorphism classes of
indecomposable objects $U\in\U$ such that $U(x)\neq0$
(Lemma~\ref{lbddlma}\eqref{lbddcriterion}). Thus, for $d=1$ our definition is equivalent
to the classical one \cite[2.2]{bg81}.
Clearly, a finite-dimensional $k$-algebra $A$
is $d$-re\-pre\-sen\-ta\-tion-fin\-ite if and only if the category $\proj A$ of finitely
generated projective $A$-modules is locally $d$-re\-pre\-sen\-ta\-tion-fin\-ite.

\subsection{Basic construction} \label{basicconstruction}
We denote by $\dml$ the bounded derived category of the category $\mod\L$ of
finitely generated right $\L$-modules, and by $\widehat\L$ the repetitive category of $\L$
(see Section \ref{section: notation}).
Then the category $\mod\widehat\L$ of finitely presented $\widehat\L$-modules is Frobenius
and therefore, its stable category $\stmod\widehat{\L}$ is triangulated, with
suspension functor given by the inverse of Heller's syzygy functor $\Omega$. 
Moreover, there exists a triangle equivalence \cite{happel88} (see also
\cite[Section 3.4]{yamaura})
\begin{equation}\label{happel equivalence}
\dml\simeq \stmod\widehat{\L}
\end{equation}
restricting to the identity functor on $\mod\L$, which is a full subcategory of
both $\dml$ and $\stmod\widehat\L$. 
In what follows, we will often view this equivalence as an identification and,
consequently, we generally make no distinction between subcategories of $\dml$ and
$\stmod\widehat{\L}$.

Let $\phi$ be an automorphism of $\widehat{\L}$. We denote by
$\phi_*:\mod\widehat\L\to\mod\widehat\L$ the induced automorphism of the module category,
defined by precomposition with $\phi\inv$, 
and by $F_*:\mod\widehat{\L}\to\mod(\widehat{\L}/\phi)$ the push-down functor,
 induced by the covering functor $F:\widehat{\L}\to\widehat{\L}/\phi$ 
(see Section~\ref{galois coverings}).
A subcategory $\U$ of $\dml\simeq\stmod\widehat{\L}$ is said to be
\emph{$\phi$-equivariant} if $\U$ and $\phi_*(\U)$ have the same isomorphism closure.

The following result is of fundamental importance for our investigation.
It is obtained as an application of Corollary~\ref{bijection} in 
Section~\ref{galois coverings}.

\begin{thm} \label{basic construction}
Let $\L$ be a basic finite-dimensional $k$-algebra of finite global dimension, and
  $\phi$ an admissible automorphism of the repetitive category $\widehat\L$. 
\begin{enumerate}
\item 
If $\dml$ contains a locally bounded $\phi$-equivariant $d$-cluster-tilting subcategory
$\U$, then the algebra $\widehat\L/\phi$ is $d$-re\-pre\-sen\-ta\-tion-fin\-ite.
  The converse implication also holds in case $k$ is algebraically closed.
  \label{basic1}
\item 
If, in \eqref{basic1}, $S$ is a cross-section of the $\phi_*$-orbits of $\ind\U$, then
   \[V=(\widehat{\L}/\phi)\oplus\bigoplus_{X\in S} F_*(X)\]
  is a basic $d$-cluster-tilting $(\widehat{\L}/\phi)$-module.
  \label{basic2}
\end{enumerate}
\end{thm}

The following diagram illustrates the correspondence between $d$-cluster-tilting
  subcategories of $\dml$ and $\mod(\widehat{\L}/\phi)$ underlying
Theorem~\ref{basic construction}.
\[\xymatrix@R1.5em{
\dml 
\ar@{}[r]|{\mbox{\large{$\simeq$}}}
& \stmod\widehat\L & \mod\widehat\L\ar[l]_{\rm nat}\ar[r]^{F_*} & \mod(\widehat\L/\phi) \\
\U\ar@{<.>}[rrr]\ar@{}[u]|{\bigcup}   &&&                           \add V\ar@{}[u]|{\bigcup}
}\]

  \begin{rmk} \label{stable_lbdd}
  In the setting of Theorem~\ref{basic construction}, a $\phi $-equivariant
  $d$-cluster-tilting subcategory $\mathcal {U}$ of $\mathop {\mathcal {D}^\mathrm {b}}(\Lambda  )$ is locally bounded if and only if
  its preimage under the natural functor $\mathop{\mathrm{mod}}\nolimits\widehat\Lambda\to\stmod\widehat\Lambda$ is locally bounded, if and only if
  the number of $\phi_*$-orbits of $\ind \mathcal {U}$ is finite.
A proof of this is given at the end of Section~\ref{lemmata}.
\end{rmk}

Next, we give a systematic construction of pairs $\L$ and $\phi$ satisfying the conditions
in Theorem~\ref{basic construction}, and thus giving rise to self-injective algebras that
are $d$-re\-pre\-sen\-ta\-tion-fin\-ite.
Let $\L$ be a finite-dimensional $k$-algebra of finite global dimension, and let
\begin{equation} \label{derivednu}
\nu=D\circ\mathbb{R}\!\Hom_{\L}(-,\L) \simeq - \otimes^{\mathbb{L}}_{\L}D\L :\dml\to\dml
\end{equation}
be the \emph{Nakayama functor} of $\dml$.
Here $D=\Hom_k(-,k)$ is the usual $k$-dual.
The Nakayama functor satisfies the functorial isomorphism 
$$\Hom_{\dml}(X,Y)\simeq D\Hom_{\dml}(Y,\nu X) \,,$$
in other words, it is a Serre functor on $\dml$.
We denote
\begin{equation}
  \nu_d=\nu\circ[-d]:\dml\to\dml \,,
\end{equation}
and say that $\L$ is \emph{$\nu_d$-finite} if $\gldim\L < \infty$ and
\[\nu_d^i\left(\mathcal{D}^{\ge0}(\L)\right)\subset\mathcal{D}^{\ge1}(\L)\]
for some $i\ge0$. Here $\mathcal{D}^{\ge \ell}(\Lambda)$ denotes the full subcategory of
$\dml$ formed by all objects $X\in\dml$ satisfying $H^j(X)=0$ for all $j<\ell$.
The property of being $\nu_d$-finite is preserved under derived equivalences
\cite[Lemma~5.6]{iyama11}.
All algebras of global dimension at most $d-1$, as well as all $d$-re\-pre\-sen\-ta\-tion-fin\-ite
algebras of global dimension $d$, are $\nu_d$-finite.
In particular, path algebras of quivers of Dynkin type are $\nu_d$-finite for all
$d\ge1$.
Note that $\nu_1=\nu\circ[-1]$ is the Auslander--Reiten translation on $\dml$
\cite[Proposition~I.2.3]{RV}, \cite[Theorem~I:4.6]{happel88}.
Hence, a hereditary algebra is $\nu_1$-finite if and only if it is representation finite.

From \cite[Proposition~5.4, Lemma~5.6, Theorem~1.23]{iyama11} we get the following
result, which plays an important role in the present paper.
For any $T\in\dml$, set
\begin{equation} \label{def:unl}
  \U_d(T)=\add\{\nu_d^i(T)\mid i\in\Z\} \subset \dml \,.
\end{equation}
\begin{prop}\label{U_n}
If $\L$ is $\nu_d$-finite and $T\in\dml$ is a tilting complex such that the algebra
$\End_{\dml}(T)$ has global dimension at most $d$, then $\U_d(T)$ is a 
locally bounded $d$-cluster-tilting subcategory of $\dml$.
\end{prop}

We call a $d$-cluster-tilting subcategory $\U$ of $\dml$ \emph{orbital} if it can be
written as $\U=\U_d(T)$ for some $T$ as in Proposition~\ref{U_n}. 
These will be our primary examples of $d$-cluster-tilting subcategories of
$\dml\simeq\stmod\widehat\L$.
The following result, which is an immediate consequence of 
Theorem~\ref{basic construction}, Remark~\ref{stable_lbdd}  and Proposition~\ref{U_n}, allows us to construct large classes of
$d$-re\-pre\-sen\-ta\-tion-fin\-ite self-injective algebras.

\begin{thm} \label{basic corollary}
Let $\L$ be a finite-dimensional $\nu_d$-finite $k$-algebra, and $\phi$ an admissible
automorphism of the repetitive category $\widehat\L$ of $\L$.
\begin{enumerate}
\item If $\dml$ contains a tilting complex $T$ satisfying
$\gldim(\End_{\dml}(T))\le d$, such that $\U_d(T)$ is $\phi$-equivariant, then
$\widehat{\L}/\phi$ is $d$-re\-pre\-sen\-ta\-tion-fin\-ite and self-injective.
\label{basiccor1}
\item If, in \eqref{basiccor1}, $S$ is a cross-section of the $\phi_*$-orbits of
  $\ind\left(\U_d(T)\right)$, then 
  \[V=(\widehat{\L}/\phi)\oplus\bigoplus_{X\in S} F_*(X)\]
  is a basic $d$-cluster tilting $(\widehat{\L}/\phi)$-module.
\end{enumerate}
\end{thm}

Denote by $\widehat\nu=\nu_{\widehat\L}$ the Nakayama automorphism of $\widehat{\L}$.
If $\L$ is a $\nu_d$-finite algebra of global dimension at most $d$, then $\U_d(\L)$,
viewed as a $d$-cluster-tilting subcategory of $\stmod\widehat{\L}$, has the form 
$$\U_d(\L)=\add\{(\widehat{\nu}_*\Omega^{d+1})^i(\L)\mid i\in\Z\} \,.$$ 
If this subcategory is $\phi$-equivariant then, by 
Theorem~\ref{basic corollary}\eqref{basiccor1}, it follows that $\widehat{\L}/\phi$ is a
$d$-re\-pre\-sen\-ta\-tion-fin\-ite self-injective algebra.

As an illustration, we give a simple example of an application of 
Theorem~\ref{basic corollary}.

\begin{ex} \label{firstexample}
Let $\L$ be the finite-dimensional $k$-algebra given by the quiver
$\xymatrix@C1em{
1\ar[r]^a&2\ar[r]^b&3
}$ 
with the relation $ab=0$.
Being derived equivalent to a hereditary algebra of Dynkin type $A_3$, the algebra
$\L$ is $\nu_2$-finite; moreover, $\gldim\L=2$. Hence, $\U_2(\L)$ is a $2$-cluster-tilting
subcategory of $\dml$, by Proposition~\ref{U_n}.
The Auslander--Reiten quiver of $\dml$ is depicted below, with the objects of $\U_2(\L)$
written in circles. The object $X\in\dml$ is the cone of a non-zero morphism
${\begin{smallmatrix}2\\ 3\end{smallmatrix}}\to{\begin{smallmatrix}1\\ 2\end{smallmatrix}}$.
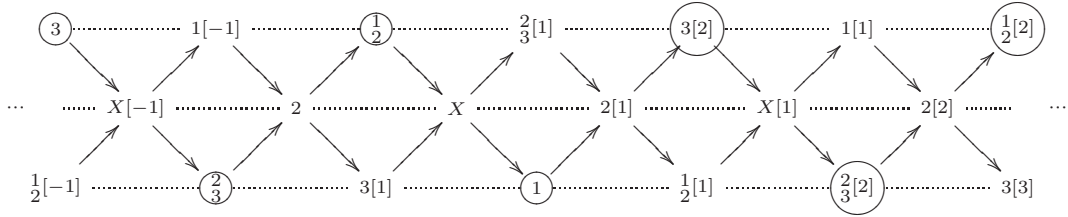
\[\begin{xy}0;<3em,0pt>:<0pt,3em>::
(-0.5,1) *+{{\begin{smallmatrix}\cdots\end{smallmatrix}}},
(0,1) *+{}="01",
(0,0) *+{{\begin{smallmatrix}1\\ 2\end{smallmatrix}}\strut\kern-3pt{\begin{smallmatrix}[-1]\end{smallmatrix}}}="00",
(0,2) *+[Fo]{{\begin{smallmatrix}3\end{smallmatrix}}}="02",
(1,1) *+{{\begin{smallmatrix}X[-1]\end{smallmatrix}}}="11",
(2,0) *+[Fo]{{\begin{smallmatrix}2\\ 3\end{smallmatrix}}}="20",
(2,2) *+{{\begin{smallmatrix}1[-1]\end{smallmatrix}}}="22",
(3,1) *+{{\begin{smallmatrix}2\end{smallmatrix}}}="31",
(4,0) *+{{\begin{smallmatrix}3[1]\end{smallmatrix}}}="40",
(4,2) *+[Fo]{{\begin{smallmatrix}1\\ 2\end{smallmatrix}}}="42",
(5,1) *+{{\begin{smallmatrix}X\end{smallmatrix}}}="51",
(6,0) *+[Fo]{{\begin{smallmatrix}1\end{smallmatrix}}}="60",
(6,2) *+{{\begin{smallmatrix}2\\ 3\end{smallmatrix}}\strut\kern-3pt{\begin{smallmatrix}[1]\end{smallmatrix}}}="62",
(7,1) *+{{\begin{smallmatrix}2[1]\end{smallmatrix}}}="71",
(8,0) *+{{\begin{smallmatrix}1\\ 2\end{smallmatrix}}\strut\kern-3pt{\begin{smallmatrix}[1]\end{smallmatrix}}}="80",
(8,2) *+[Fo]{{\begin{smallmatrix}3[2]\end{smallmatrix}}}="82",
(9,1) *+{{\begin{smallmatrix}X[1]\end{smallmatrix}}}="91",
(10,0) *+[Fo]{{\begin{smallmatrix}2\\ 3\end{smallmatrix}}\strut\kern-3pt{\begin{smallmatrix}[2]\end{smallmatrix}}}="100",
(10,2) *+{{\begin{smallmatrix}1[1]\end{smallmatrix}}}="102",
(11,1) *+{{\begin{smallmatrix}2[2]\end{smallmatrix}}}="111",
(12,0) *+{{\begin{smallmatrix}3[3]\end{smallmatrix}}}="120",
(12,1) *+{}="121",
(12,2) *+[Fo]{{\begin{smallmatrix}1\\ 2\end{smallmatrix}}\strut\kern-3pt{\begin{smallmatrix}[2]\end{smallmatrix}}}="122",
(12.5,1) *+{{\begin{smallmatrix}\cdots\end{smallmatrix}}},
\ar"00";"11",
\ar"02";"11",
\ar"11";"20",
\ar"11";"22",
\ar"20";"31",
\ar"22";"31",
\ar"31";"40",
\ar"31";"42",
\ar"40";"51",
\ar"42";"51",
\ar"51";"60",
\ar"51";"62",
\ar"60";"71",
\ar"62";"71",
\ar"71";"80",
\ar"71";"82",
\ar"80";"91",
\ar"82";"91",
\ar"91";"100",
\ar"91";"102",
\ar"100";"111",
\ar"102";"111",
\ar"111";"120",
\ar"111";"122",
\ar@{.}"00";"20",
\ar@{.}"20";"40",
\ar@{.}"40";"60",
\ar@{.}"60";"80",
\ar@{.}"80";"100",
\ar@{.}"100";"120",
\ar@{.}"01";"11",
\ar@{.}"11";"31",
\ar@{.}"31";"51",
\ar@{.}"51";"71",
\ar@{.}"71";"91",
\ar@{.}"91";"111",
\ar@{.}"111";"121",
\ar@{.}"02";"22",
\ar@{.}"22";"42",
\ar@{.}"42";"62",
\ar@{.}"62";"82",
\ar@{.}"82";"102",
\ar@{.}"102";"122",
\end{xy}\]
The repetitive algebra $\widehat{\L}$ of $\L$ is given by the following quiver with mesh
relations.
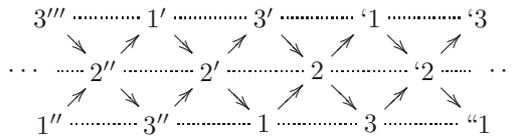
\[\begin{xy}0;<2.5em,0pt>:<0pt,2.5em>::
(-0.5,1) *+{\cdots},
(0,0) *+{1''}="00",
(0,1) *+{}="01",
(0,2) *+{3'''}="02",
(1,1) *+{2''}="11",
(2,0) *+{3''}="20",
(2,2) *+{1'}="22",
(3,1) *+{2'}="31",
(4,0) *+{1}="40",
(4,2) *+{3'}="42",
(5,1) *+{2}="51",
(6,0) *+{3}="60",
(6,2) *+{{}`1}="62",
(7,1) *+{{}`2}="71",
(8,0) *+{{}``1}="80",
(8,1) *+{}="81",
(8,2) *+{{}`3}="82",
(8.5,1) *+{\cdots},
\ar"00";"11",
\ar"02";"11",
\ar"11";"20",
\ar"11";"22",
\ar"20";"31",
\ar"22";"31",
\ar"31";"40",
\ar"31";"42",
\ar"40";"51",
\ar"42";"51",
\ar"51";"60",
\ar"51";"62",
\ar"60";"71",
\ar"62";"71",
\ar"71";"80",
\ar"71";"82",
\ar@{.}"00";"20",
\ar@{.}"20";"40",
\ar@{.}"40";"60",
\ar@{.}"60";"80",
\ar@{.}"01";"11",
\ar@{.}"11";"31",
\ar@{.}"31";"51",
\ar@{.}"51";"71",
\ar@{.}"71";"81",
\ar@{.}"02";"22",
\ar@{.}"22";"42",
\ar@{.}"42";"62",
\ar@{.}"62";"82",
\end{xy}\]
We denote by $\sigma$ the involution of $\widehat\L$ given by reflection in the central
horizontal line. Then $\widehat{\nu}$ and $\sigma$ generate a subgroup of
$\Aut(\widehat\L)$ which is isomorphic to $\Z\times(\Z/2\Z)$.

The Auslander--Reiten quiver of $\mod\widehat\L$ is the following, where the projective
$\widehat\L$-modules are written in squares and the objects in the $2$-cluster-tilting
subcategory $\U_2(\L)\subset\stmod\widehat{\L}$ are indicated with circles.
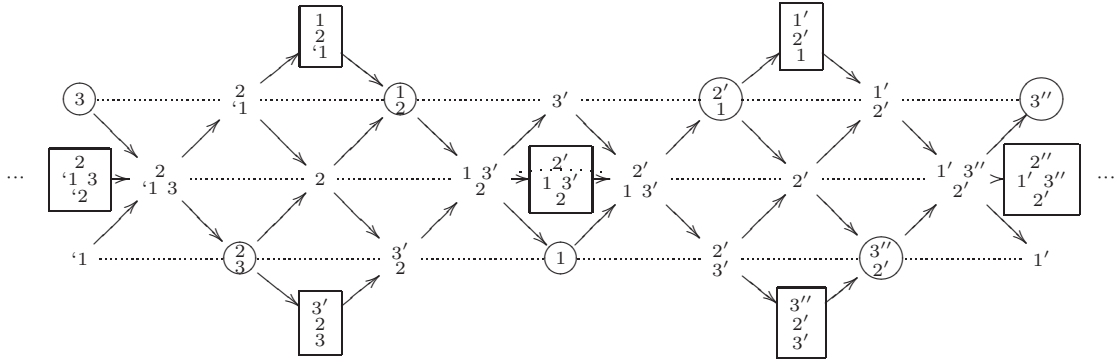
\[
\begin{xy}0;<3em,0pt>:<0pt,3em>::
(-0.8,1) *+{{\begin{smallmatrix}\cdots\end{smallmatrix}}},
(0,0) *+{{\begin{smallmatrix}{}`1\end{smallmatrix}}}="00",
(0,1) *+[F]{{\begin{smallmatrix}2\\ {}`1\ 3\\ {}`2\end{smallmatrix}}}="01",
(0,2) *+[Fo]{{\begin{smallmatrix}3\end{smallmatrix}}}="02",
(1,1) *+{{\begin{smallmatrix}2\\ {}`1\ 3\end{smallmatrix}}}="11",
(2,0) *+[Fo]{{\begin{smallmatrix}2\\ 3\end{smallmatrix}}}="20",
(2,2) *+{{\begin{smallmatrix}2\\ {}`1\end{smallmatrix}}}="22",
(3,-.8) *+[F]{{\begin{smallmatrix}3'\\ 2\\ 3\end{smallmatrix}}}="30",
(3,1) *+{{\begin{smallmatrix}2\end{smallmatrix}}}="31",
(3,2.8) *+[F]{{\begin{smallmatrix}1\\ 2\\ {}`1\end{smallmatrix}}}="32",
(4,0) *+{{\begin{smallmatrix}3'\\ 2\end{smallmatrix}}}="40",
(4,2) *+[Fo]{{\begin{smallmatrix}1\\ 2\end{smallmatrix}}}="42",
(5,1) *+{{\begin{smallmatrix}1\ 3'\\ 2\end{smallmatrix}}}="51",
(6,0) *+[Fo]{{\begin{smallmatrix}1\end{smallmatrix}}}="60",
(6,1) *+[F]{{\begin{smallmatrix}2'\\ 1\ 3'\\ 2\end{smallmatrix}}}="61",
(6,2) *+{{\begin{smallmatrix}3'\end{smallmatrix}}}="62",
(7,1) *+{{\begin{smallmatrix}2'\\ 1\ 3'\end{smallmatrix}}}="71",
(8,0) *+{{\begin{smallmatrix}2'\\ 3'\end{smallmatrix}}}="80",
(8,2) *+[Fo]{{\begin{smallmatrix}2'\\ 1\end{smallmatrix}}}="82",
(9,-.8) *+[F]{{\begin{smallmatrix}3''\\ 2'\\ 3'\end{smallmatrix}}}="90",
(9,1) *+{{\begin{smallmatrix}2'\end{smallmatrix}}}="91",
(9,2.8) *+[F]{{\begin{smallmatrix}1'\\ 2'\\ 1\end{smallmatrix}}}="92",
(10,0) *+[Fo]{{\begin{smallmatrix}3''\\ 2'\end{smallmatrix}}}="100",
(10,2) *+{{\begin{smallmatrix}1'\\ 2'\end{smallmatrix}}}="102",
(11,1) *+{{\begin{smallmatrix}1'\ 3''\\ 2'\end{smallmatrix}}}="111",
(12,0) *+{{\begin{smallmatrix}1'\end{smallmatrix}}}="120",
(12,1) *+[F]{{\begin{smallmatrix}2''\\ 1'\ 3''\\ 2'\end{smallmatrix}}}="121",
(12,2) *+[Fo]{{\begin{smallmatrix}3''\end{smallmatrix}}}="122",
(12.8,1) *+{{\begin{smallmatrix}\cdots\end{smallmatrix}}},
\ar"00";"11",
\ar"01";"11",
\ar"02";"11",
\ar"11";"20",
\ar"11";"22",
\ar"20";"30",
\ar"20";"31",
\ar"22";"31",
\ar"22";"32",
\ar"30";"40",
\ar"31";"40",
\ar"31";"42",
\ar"32";"42",
\ar"40";"51",
\ar"42";"51",
\ar"51";"60",
\ar"51";"61",
\ar"51";"62",
\ar"60";"71",
\ar"61";"71",
\ar"62";"71",
\ar"71";"80",
\ar"71";"82",
\ar"80";"90",
\ar"80";"91",
\ar"82";"91",
\ar"82";"92",
\ar"90";"100",
\ar"91";"100",
\ar"91";"102",
\ar"92";"102",
\ar"100";"111",
\ar"102";"111",
\ar"111";"120",
\ar"111";"121",
\ar"111";"122",
\ar@{.}"00";"20",
\ar@{.}"20";"40",
\ar@{.}"40";"60",
\ar@{.}"60";"80",
\ar@{.}"80";"100",
\ar@{.}"100";"120",
\ar@{.}"11";"31",
\ar@{.}"31";"51",
\ar@{.}@/^.3pc/"51";"71",
\ar@{.}"71";"91",
\ar@{.}"91";"111",
\ar@{.}"02";"22",
\ar@{.}"22";"42",
\ar@{.}"42";"62",
\ar@{.}"62";"82",
\ar@{.}"82";"102",
\ar@{.}"102";"122",
\end{xy}
\]
Then again, $\sigma_*$ acts as reflection in the central horizontal line, and
$\widehat{\nu}_*\simeq\tau^{-3}$, where $\tau$ denotes the Auslander--Reiten translation
in $\stmod\widehat\L$.
Thus, for example, the automorphisms $\widehat{\nu}\sigma$ and $\widehat{\nu}^2$ of
$\widehat{\L}$ satisfy $(\widehat{\nu}\sigma)_*(\U)=\U$ and $(\widehat{\nu}^2)_*(\U)=\U$
 and consequently, the orbit algebras $\widehat{\L}/(\widehat{\nu}\sigma)$ and
 $\widehat{\L}/(\widehat{\nu}^2)$ are $2$-re\-pre\-sen\-ta\-tion-fin\-ite self-injective.
We remark that $\widehat{\L}/(\widehat{\nu}\sigma)$ is isomorphic to the preprojective algebra
$\Pi_2(A_3)$ of the Dynkin diagram $A_3$, while $\widehat{\L}/(\widehat{\nu}^2)$ is the
$2$-fold trivial extension algebra of $\L$ (see Section~\ref{section: applications}).
\end{ex}

Another example illustrating Theorem~\ref{basic corollary} is given in
Section~\ref{toyex}.

\subsection{Applications} \label{section: applications}
Below we give some examples of how Theorem~\ref{basic corollary} can be used to show
$d$-re\-pre\-sen\-ta\-tion-fin\-ite\-ness of some well-known types of self-injective algebras.
First, we show how the first part of Riedtmann's result follows from
Theorem~\ref{basic corollary}.

\begin{proof}[Proof of Theorem~\ref{riedtmann}(a)]
Let $\L$ be a tilted algebra of Dynkin type and $\phi$ an admissible automorphism
of $\widehat\L$. Then the corresponding path algebra $H$ may be viewed as the
endomorphism algebra of a tilting complex $T$ in $\dml$, and
$\gldim(\End_{\dml}(T))=\gldim H\le1$. 
Moreover, since $\U_1(H)=\dm{H}\simeq\dml$, it follows that $\U_1(T)=\dml$, so
$\U_1(T)$ is clearly $\phi$-equivariant.
Therefore $\widehat\L/\phi$, is $1$-re\-pre\-sen\-ta\-tion-fin\-ite and self-injective by
Theorem~\ref{basic corollary}.
\end{proof}

It is natural to view $d$-re\-pre\-sen\-ta\-tion-fin\-ite algebras of global dimension $d$ as
generalizations of re\-pre\-sen\-ta\-tion-fin\-ite hereditary algebras.
If $\L$ is $d$-re\-pre\-sen\-ta\-tion-fin\-ite and $\gldim\L=d$ then $\mod\L$ has a unique basic
$d$-cluster-tilting module $M$, given as $M=\bigoplus_{\ell\ge0}\tau_d^{\ell}(I)$, where
  $I$ is a basic injective cogenerator of $\L$ and 
\[\tau_d=D\Ext^d_{\L}(-,\L):\mod\L\to\mod\L\]
the \emph{$d$-Auslander--Reiten translation} \cite[Proposition~1.3]{iyama11}.

Corollary~\ref{trivext} below extends a famous theorem of Tachikawa--Yamagata
\cite{tachikawa80,yamagata81} stating that the trivial extension algebras of
re\-pre\-sen\-ta\-tion-fin\-ite hereditary algebras are again representation finite.

The \emph{$n$-fold trivial extension algebra} of $\L$ is defined as
$T_n(\L)=\widehat\L/\widehat{\nu}^n$; where $\widehat{\nu}=\nu_{\widehat{\L}}$ is the Nakayama 
automorphism of $\widehat\L$, and $n\ge1$.
In particular, $T_1(\L)$ is the usual trivial extension algebra $T(\L)$.
For $n\ge2$, $T_n(\L)$ may be viewed as the matrix algebra
$$
\begin{pmatrix}
\L &&&&D\L \\
D\L&\L\\
&\ddots&\ddots \\
&&\ddots&\L\\
&&&D\L& \L
\end{pmatrix}
\:,$$ 
where the multiplications $\Lambda\cdot D\Lambda$ and $D\Lambda\cdot\Lambda$ are defined
by the $\L$--$\L$-bimodule structure of $D\L$.
We remark that $T_n(\L)$ is a finite-dimensional self-injective $k$-algebra, and that
there is a natural embedding
$\mathop{\mathrm{mod}}\nolimits \Lambda \to \mathop{\mathrm{mod}}\nolimits T_n(\Lambda)$, given by the projection $T_n(\Lambda)\to \Lambda$ onto the first diagonal entry.
The functor $(\nu_{T_n(\L)})_*:\mod T_n(\L)\to \mod T_n(\L)$, induced by the Nakayama
automorphism $\nu_{T_n(\L)}$ of $T_n(\L)$, satisfies 
$(\nu_{T_n(\L)})_* \simeq D\Hom_{T_n(\L)}(-,T_n(\L))$. 

\begin{cor}\label{trivext}
  Let $\L$ be a basic $d$-re\-pre\-sen\-ta\-tion-fin\-ite algebra of global dimension $d$, and $M$
  the unique basic $d$-cluster-tilting $\L$-module. 
  Then, for any $\ell\ge1$, the $d\ell$-fold trivial extension algebra $T_{d\ell}(\L)$ is
    $d$-rep\-re\-sen\-ta\-tion-finite, and the $T_{d\ell}(\L)$-module 
    \begin{align*}
      U &= T_{d\ell}(\L) \oplus \left(\bigoplus_{i=0}^{\ell(d+1)-1}\left((\nu_{T_{d\ell}(\L)})_*\Omega\right)^i(M/\L)\right) \oplus
      \left(\bigoplus_{j=0}^{\ell-1}\left((\nu_{T_{d\ell}(\L)})_*\Omega^{d+1}\right)^j(\L)\right) 
    \end{align*}
    is a basic $d$-cluster-tilting module.
\end{cor}

Setting $d=\ell=1$ above recovers the classical result by Tachikawa--Yamagata.
Another instance of Corollary~\ref{trivext} was given in Example~\ref{firstexample}, where we
saw that $T_{2}(\L)=\widehat{\L}/\widehat{\nu}^{2}$ was $2$-re\-pre\-sen\-ta\-tion-fin\-ite.

We now consider a wider class of algebras.
Let $a,b\in\Z_{>0}$. An algebra $\L$ of finite global dimension is said to be
\emph{fractionally $\frac{b}{a}$-Calabi--Yau} if $\nu^a\simeq[b]$ holds in $\dml$, and
\emph{twisted fractionally $\frac{b}{a}$-Calabi--Yau} if there exists an automorphism
$\phi$ of $\L$ such that $\nu^a\simeq[b]\circ\phi_*$, 
where $\phi_*:\dml\to\dml$ is the functor induced by the functor
$\phi_*:\mod\L\to\mod\L$ (see Section~\ref{galois coverings}).
For example, any $d$-re\-pre\-sen\-ta\-tion-fin\-ite algebra of global dimension $d$ is
twisted fractionally Calabi--Yau \cite[Theorem~1.1(a)]{hi11a}.

\begin{cor} \label{fracCY}
Let $\L$ be a basic $k$-algebra that is twisted $\frac{b}{a}$-Calabi--Yau, and $d$ a positive integer satisfying
$d\ge\gldim\L$. Set $g=\gcd(d+1,a+b)$, and $n=\ell(ad-b)/g$, where $\ell\ge1$. Then
$T_n(\L)$ is $d$-re\-pre\-sen\-ta\-tion-fin\-ite, and 
\begin{equation}\label{twfraccy}
  V = \, T_n(\L) \, \oplus \bigoplus_{i=0}^{\ell(a+b)/g-1}\left((\nu_{T_n(\L)})_*\Omega^{d+1}\right)^i(\L) 
\end{equation}
is a basic $d$-cluster-tilting $T_n(\L)$-module.
\end{cor}

\begin{ex}\label{tubular}
Let $\L$ be a canonical algebra of tubular type, that is, an algebra of one of the four
types listed in Figure~\ref{fig:tubular}.

\begin{figure}[h]
\begin{minipage}[b]{.35\textwidth}
\[\xymatrix@R=.8em{
&\bullet\ar[dr]|{x_1}\\
\bullet\ar[r]|{x_2}\ar[ru]|{x_1}\ar[rd]|{x_3}\ar[rdd]|{x_4}&\bullet\ar[r]|{x_2}&\bullet\\
&\bullet\ar[ru]|{x_3}\\
&\bullet\ar[ruu]|{x_4}
}\]
\center{\small $x_1^2+x_2^2+x_3^2=0,\; x_1^2+\lambda x_2^2+x_4^2=0$,\; $\l\in k\setminus\{0,1\}$,
}
\caption*{Type $(2,2,2,2)$}
\end{minipage}
\begin{minipage}[b]{.60\textwidth}
\[
\xymatrix@R=1.25em{
&&\bullet\ar[rrd]|{x_1}\\
\bullet\ar[r]|{x_2}\ar[rru]|{x_1}\ar[rd]|{x_3}&\bullet\ar[r]|{x_2}&\bullet\ar[r]|{x_2}&\bullet\ar[r]|{x_2}&\bullet\\
&\bullet\ar[r]|{x_3}&\bullet\ar[r]|{x_3}&\bullet\ar[ru]|{x_3}
}\]
\center{\small $x_1^2+x_2^4+x_3^4=0$,}
\caption*{Type $(2,4,4)$}
\end{minipage}

\vspace*{1em}
\begin{minipage}[b]{.35\textwidth}
\[
\xymatrix@R=1.1em{
&\bullet\ar[r]|{x_1}&\bullet\ar[dr]|{x_1}\\
\bullet\ar[r]|{x_2}\ar[ru]|{x_1}\ar[rd]|{x_3}&\bullet\ar[r]|{x_2}&\bullet\ar[r]|{x_2}&\bullet\\
&\bullet\ar[r]|{x_3}&\bullet\ar[ru]|{x_3}
}\]
\center{\small $x_1^3+x_2^3+x_3^3=0$,
}
\caption*{Type $(3,3,3)$}
\end{minipage}
\begin{minipage}[b]{.60\textwidth}
\[
\xymatrix@R=1.15em@C=2em{
&&&\bullet\ar[rrrd]|{x_1}\\
\bullet\ar[rr]|{x_2}\ar[rrru]|{x_1}\ar[rd]|{x_3}&&\bullet\ar[rr]|{x_2}&&\bullet\ar[rr]|{x_2}&&\bullet\\
&\bullet\ar[r]|{x_3}&\bullet\ar[r]|{x_3}&\bullet\ar[r]|{x_3}&\bullet\ar[r]|{x_3}&\bullet\ar[ru]|{x_3}
}\]
\center{\small $x_1^2+x_2^3+x_3^6=0$.}
\caption*{Type $(2,3,6)$}
\end{minipage}
\caption{Canonical algebras of tubular type.}\label{fig:tubular}
\end{figure}

These are tame algebras of polynomial growth \cite{sy08}, and moreover, they are 
$\frac{p}{p}$-Calabi--Yau, where $p$ is given by the following list.
\[\begin{array}{|c||c|c|c|c|} \hline
\mbox{type}&(2,2,2,2)&(3,3,3)&(2,4,4)&(2,3,6)\\ \hline\hline
p&2&3&4&6\\ \hline
n&\frac{2(d-1)}{\gcd(4,d+1)}&\frac{3(d-1)}{\gcd(6,d+1)}&\frac{4(d-1)}{\gcd(8,d+1)}&\frac{6(d-1)}{\gcd(12,d+1)}\\ \hline
\end{array}\]
Since canonical algebras have global dimension 2, Corollary~\ref{fracCY} implies that, for
$n$ as in the list, $T_{n\ell}(\L)$ is $d$-re\-pre\-sen\-ta\-tion-fin\-ite for any $\ell\ge1$ and
$d\ge2$.

In particular, for any $\ell\ge1$:
if $\L$ is of type $(2,2,2,2)$ then $T_\ell(\L)$   is $3$-re\-pre\-sen\-ta\-tion-fin\-ite,
if $\L$ is of type $(3,3,3)$   then $T_\ell(\L)$   is $2$-re\-pre\-sen\-ta\-tion-fin\-ite,
if $\L$ is of type $(2,4,4)$   then $T_{2\ell}(\L)$ is $3$-re\-pre\-sen\-ta\-tion-fin\-ite, and
if $\L$ is of type $(2,3,6)$   then $T_{2\ell}(\L)$ is $2$-re\-pre\-sen\-ta\-tion-fin\-ite.
\end{ex}

We give one more example of an application of Theorem~\ref{basic corollary} to a particular class of $n$-fold trivial extension algebras.
Let $r$ be a positive integer. A $d$-re\-pre\-sen\-ta\-tion-fin\-ite algebra $\L$ of global
dimension $d$ is said to be \emph{$r$-homogeneous} if $\nu_d^r(\L)\simeq\L[-d]$ in
$\dm{\L}$ or, equivalently, $\tau_d^{r-1}(D\L)\simeq\L$ in $\mod\L$ \cite{hi11a}.

\begin{cor}\label{trivrf}
Let $\L$ be a basic $r$-homogeneous $d$-re\-pre\-sen\-ta\-tion-fin\-ite algebra of global dimension
$d$. Then the $(r-1)$-fold trivial extension algebra $T_{r-1}(\L)$ of $\L$ is
$(r-1)(d+1)$-re\-pre\-sen\-ta\-tion-fin\-ite, and
$M=T_{r-1}(\L)\oplus \L$ is a basic $(r-1)(d+1)$-cluster-tilting $T_{r-1}(\L)$-module.
\end{cor}

In particular, the trivial extension algebra of a $2$-homogeneous
$d$-re\-pre\-sen\-ta\-tion-fin\-ite algebra is $(d+1)$-re\-pre\-sen\-ta\-tion-fin\-ite.

The next application is to higher preprojective algebras, which were introduced in
\cite{io11,io13,keller11}.
For a finite-dimensional $k$-algebra $\L$ with $\gldim\L\le d$, the \emph{$(d+1)$-preprojective algebra} is defined as the tensor algebra
\[\Pi=\Pi(\L)=T_\Lambda\Ext^d_\L(D\L,\L)\]
of the $\L$-bimodule $\Ext^d_\L(D\L,\L)$. Thus $\Pi$ has a natural structure of a
$\Z$-graded $k$-algebra: $\Pi=\bigoplus_{\ell\ge0}\Pi_\ell$. 
More generally, for a positive integer $n$, we define the \emph{$n$-fold $(d+1)$-preprojective
algebra} of $\L$ as the matrix algebra 
\[\Pi^{(n)}=\Pi^{(n)}(\L)=(\Pi_{i-j+n\Z})_{1\le i,j\le n},\]
where $\Pi_{j+n\Z}=\bigoplus_{\ell\in\Z}\Pi_{j+n\ell}$. Note that
$\Pi^{(1)}=\Pi$. 
We apply our construction to give the following result, generalizing parts of \cite[Theorem~2.2]{gls06},\cite[Theorem~1]{gls07} by Geiss--Leclerc--Schr\"oer for the classical preprojective algebras
(i.e., $d=1$ and $\ell=1$) and of \cite[Corollaries~3.4 and 4.16]{io13}
due to Iyama--Oppermann for the higher preprojective algebras ($\ell=1$).

\begin{cor} \label{preproj}
Let $\L$ be a $d$-re\-pre\-sen\-ta\-tion-fin\-ite algebra of global dimension $d$ and $n$ a
positive integer. Then
the $n$-fold $(d+1)$-preprojective algebra $\Pi^{(n)}(\L)$ of $\L$ is
$(d+1)$-re\-pre\-sen\-ta\-tion-fin\-ite and self-injective.
\end{cor}
The corollaries~\ref{trivext}, \ref{fracCY}, \ref{trivrf} and~\ref{preproj} will be proved
in Section~\ref{pfother}. 
An example illustrating Corollary~\ref{preproj} is given in Section~\ref{sec:toypreproj}.

\subsection{Galois coverings and $d$-cluster-tilting} \label{galois coverings}
Our main result, Theorem~\ref{basic construction}, is largely a consequence of a more
general result, Theorem~\ref{stronger theorem} below. It is based on the classical theory
of Galois coverings of $k$-linear categories, initiated by Gabriel
\cite{gabriel2,gabriel-roiter}.

Let $\C$ be a skeletally small locally bounded $k$-linear Krull-Schmidt category and
$G$ an admissible group of automorphisms of $\C$ (see Definition \ref{define admissible}).
Then the orbit category $\C/G$ (see Section \ref{section: notation}) is again a
skeletally small locally bounded $k$-linear Krull-Schmidt category. 
The natural functor
\[ F:\C\to\C/G  \]
induces an exact functor
\[F^*:\Mod(\C/G)\to\Mod\C\]
called the \emph{pull-up}, given by $F^*(M)=M\circ F$.
This functor has a left adjoint
\[F_*:\Mod\C\to\Mod(\C/G)\]
called the \emph{push-down} \cite{gabriel-roiter}, which is also exact. 
If $\Hom_{\C}(-,y)\xrightarrow{(-,f)}\Hom_{\C}(-,x)\to M\to0$ is a projective presentation of
$M\in\mod\C$, then $F_*(M)$ is given as the cokernel of the induced map
$\Hom_{\C/G}(-,y)\xrightarrow{(-,F(f))}\Hom_{\C/G}(-,x)$. In particular, $F_*$ induces a
functor $F_*:\mod\C\to\mod(\C/G)$ between the categories of finitely presented modules of
$\C$ and $\C/G$, respectively.
For a more detailed presentation, see \cite[Section~6]{bl14}.

\begin{rmk}
  In the above, the purpose assuming that $\C$ is skeletally small is to ensure that the collection of morphisms between two objects $X,Y\in\Mod\C$ forms a set.
  As, in the present work, we only concern ourselves with finitely presented modules, this set-theoretic difficulty may be circumvented in the following way:
  For an infinite cardinal $\kappa$, let $\Mod^\kappa\C$ be the category of $\C$-modules $Y$ admitting an epimorphism
  \[ \bigoplus_{x\in\mathcal{X}} \Hom_\C(-,x)\to Y \,, \]
  where $\mathcal{X}$ is a collection of objects in $\C$ with $|\mathcal{X}|\le \kappa$.
  If $G$ is a group acting on $\C$, and $\kappa\ge |G|$, then $F^*(Y)\in\Mod^{\kappa}\C$ for all  $Y\in\Mod^{\kappa}(\C/G)$, and $F_*(X)\in\Mod^{\kappa}(\C/G)$ when $X\in\Mod^{\kappa}\C$.
  Hence, the pull-up and push-down functors may be defined on the categories
  $\Mod^{\kappa}(\C/G)$ and $\Mod^{\kappa}\C$ instead of the full module categories $\Mod\C$ and $\Mod(\C/G)$:
  \[F^*:\Mod^{\kappa}(\C/G)\to \Mod^{\kappa}\C \,, \qquad
  F_*:\Mod^{\kappa}\C \to \Mod^{\kappa}(\C/G) \,.\]
  
  For this reason, in what follows, we will not assume the category $\C$ to be skeletally small. 
\end{rmk}

For simplicity, if $\U\subset\mod\C$ is a full subcategory closed under isomorphisms,
then $F_*(\U)\subset\mod(\C/G)$ denotes the isomorphism closure of the image of $F_*$,
that is, the smallest full subcategory of $\mod(\C/G)$ that is closed under isomorphism
and contains all objects of the form $F_*(U)$ for $U\in\U$.
The preimage $F_*\inv(\V)$ of a full subcategory $\V\subset\mod(\C/G)$ is the full
subcategory of $\mod\C$ consisting of all objects $U\in\mod\C$ such that
$F_*(U)\in\V$.
  
\begin{thm}\label{stronger theorem}
Let $\C$ be a locally bounded $k$-linear Krull-Schmidt category and $G$ a
group acting admissibly on $\C$.
\begin{enumerate} 
\item Assume that $G$ acts admissibly on $\mod\C$, and let $\U$ be a
$G$-equivariant full subcategory of $\mod\C$ such that $F_*(\U)$ is
functorially finite in $\mod(\C/G)$. Then $\U$ is a $d$-cluster-tilting
subcategory of $\mod\C$ if and only if $F_*(\U)$ is a $d$-cluster-tilting
subcategory of $\mod(\C/G)$.
\item Assume that the field $k$ is algebraically closed, and that the group $G$ is
  free abelian of finite rank. Then, for any $d$-cluster-tilting
  subcategory $\V$ of $\mod(\C/G)$, the full subcategory $F_*^{-1}(\V)$ of $\mod\C$ 
  is a $G$-equi\-va\-ri\-ant $d$-cluster-tilting subcategory.
\end{enumerate}
\end{thm}

From Theorem~\ref{stronger theorem} one can deduce the following corollary,
which is an analogue of classical results due to Gabriel 
\cite[Lemma 3.3, Theorem~3.6]{gabriel2}.

\begin{cor}\label{bijection}
Let $\C$ be a  locally bounded $k$-linear Krull-Schmidt category and $G$ a
free abelian group of finite rank, acting admissibly on $\C$.
\begin{enumerate}
\item 
The push-down functor $F_*:\mod\C\to\mod(\C/G)$ induces an injective map from the class of locally bounded $G$-equivariant $d$-cluster-tilting subcategories of $\mod\C$ to the class of locally bounded $d$-cluster-tilting subcategories of $\mod(\C/G)$.
\item The above map is a bijection if $k$ is algebraically closed.
\end{enumerate}
\end{cor}

Corollary~\ref{bijection} implies that $\C$ is locally $d$-re\-pre\-sen\-ta\-tion-fin\-ite
whenever $\C/G$ is locally $d$-re\-pre\-sen\-ta\-tion-fin\-ite. However, in contrast to
the classical case, the converse is not true in general, because $d$-cluster-tilting
subcategories of $\mod\C$ are usually not $G$-equivariant.
For example, let $\C=\widehat{k}$ be the repetitive category of the
field $k$. 
Then $\widehat{k}\simeq k\mathbb{A}_\infty^\infty/I^2$, where
$\mathbb{A}_{\infty}^{\infty}$ is the linearly oriented quiver of type 
$A^{\infty}_{\infty}$, and $I\subset k\mathbb{A}_\infty^\infty$ is the ideal generated by
all arrows in $\mathbb{A}_\infty^\infty$. The group $G=\Z$ acts on
$\widehat{k}$ by translation in the quiver $\mathbb{A}_\infty^\infty$, and
$\widehat{k}/\Z\simeq T(k)\simeq k[X]/(X^2)$ is the algebra of dual numbers over $k$. 
But while $\widehat{k}$ is locally $d$-representation-finite for all $d\ge1$ by
Proposition~\ref{U_n}, the algebra $\widehat{k}/\Z$ is (locally)
$d$-representation-finite only for $d=1$ (see Theorem~\ref{nakayamaalg}).

\subsection{Notation}\label{section: notation}

In this section, we clarify our notational conventions, and recall some
well-known concepts and results.

Let $\C$ be a Hom-finite $k$-linear category.
A $k$-linear autoequivalence $\nu=\nu_{\C}:\C\to\C$ is a \emph{Nakayama functor} if for
all $x,y\in\C$ there is a functorial isomorphism 
\[\Hom_\C(x,y)\simeq D\Hom_{\C}(y,\nu x)\,.\]
A typical example is the Nakayama functor \eqref{derivednu} of $\dml$. 
Any two Nakayama functors of $\C$ are isomorphic. 
A category $\C$ having a Nakayama functor is said to be \emph{self-injective}.
For example, when $\C$ is the category $\proj A$ of finitely generated projective
modules of a $k$-algebra $A$, then $\C$ is self-injective if and only
if $A$ is a self-injective algebra. In this case, the Nakayama functor of $\C$ is the
automorphism of $\C$ induced by the classical Nakayama automorphism of the algebra $A$.

In case $\C$ is a triangulated category, the Nakayama functor of $\C$ is often called the
\emph{Serre functor} and denoted by $S = {}_{\C}S$.
In this case, and given an integer $d$, we write
\[S_d={}_{\C}S_d={}_{\C}S\circ[-d].\]
A fundamental property of the Nakayama functor is that it commutes with any
autoequivalence $\phi$ of $\C$, up to isomorphism of functors: 
\begin{equation}\label{commutes}
\phi\nu\simeq\nu\phi \,.
\end{equation}
This follows from the isomorphisms
\begin{multline*}
  \Hom_\C(y,\nu\phi x)\simeq D\Hom_\C(\phi x,y)\simeq D\Hom_\C(x,\phi\inv y)
  \simeq\Hom_\C(\phi\inv y,\nu x)\simeq\Hom_\C(y,\phi\nu x)
\end{multline*}
and Yoneda's lemma.

Let $\L$ be a finite-dimensional $k$-algebra.
The \emph{trivial extension} algebra $T(\Lambda)$ of $\L$ is defined as
$\L\oplus D\L$, with multiplication given by 
\[(a,x)\cdot(b,y) = (ab,ay+xb)\]
for $a,b\in\L$ and $x,y\in D\L$. 
Setting $T(\L)_0=\L$ and $T(\L)_1=D\L$ defines a $\Z$-grading on $T(\L)$. 
The \emph{repetitive category} $\widehat\L$ of $\L$ is defined as the category
$\proj^{\Z}T(\L)$ of finitely generated $\Z$-graded projective $T(\L)$-modules.
Observe that $\widehat\L$ is self-injective, the Nakayama automorphism being given
by degree shift by $1$:
\[\nu_{\widehat{\L}}=(1).\]
From here on, we will write $\widehat\nu$ to denote this functor. 

We now recall some basic terminology for functor categories \cite{a66}.
Let $\C$ be a  $k$-linear Krull-Schmidt category.
A \emph{$\C$-module} is a contravariant additive functor from $\C$ to the category
$\mathcal{A}b$ of abelian groups. We denote by $\Mod\C$ the category of all $\C$-modules,
and by $\mod\C$ the full subcategory of $\Mod\C$ consisting of finitely presented
modules. A module $M:\C\to\mathcal{A}b$ is \emph{finitely presented} if there exists an
exact sequence
\[\Hom_\C(-,y)\to\Hom_{\C}(-,x)\to M\to0\]
in $\Mod\C$.
A full subcategory $\U$ of $\mod\C$ is called a \emph{generator-cogenerator}
if the $\C$-modules $\Hom_{\C}(-,x)$ and $D\Hom_{\C}(x,-)$ belong to $\U$ for all
$x\in\C$.
Moreover, $\proj\C$ and $\inj\C$ denote the full subcategories of $\mod\C$ consisting
  of projective and injective objects in $\mod\C$, respectively.
We denote by $\ind\C$ a chosen set of representatives of the isomorphism classes of
indecomposable objects in $\C$, and by $\Supp M$ the \emph{support} of a $\C$-module
  $M$, that is,
\[\Supp M=\{x\in\ind\C\mid M(x)\neq0\} \subset \ind\C .\]
Modules of a finite-dimensional $k$-algebra $A$ are identified with modules of the category $\proj A$ of finitely generated projective $A$-modules.

The category $\C$ is said to be \emph{locally bounded} if for any $x\in\ind\C$, we have
\[\sum_{y\in\ind\C}(\dim_k\Hom_{\C}(x,y)+\dim_k\Hom_{\C}(y,x))<\infty.\]
Note that when $\C$ is locally bounded, the objects in $\mod\C$ are precisely the
finite-length modules of $\C$. Therefore we have a duality
\begin{equation} \label{locbdd-dualising}
D = \Hom_k(-,k):\mod\C\leftrightarrow\mod(\C^{\rm op})
\end{equation}
satisfying $D\circ D \simeq \I_{\mod\C}$.

An example of a locally bounded category with infinitely many indecomposable objects
is the repetitive category $\widehat{A}$ of a finite-dimensional $k$-algebra $A$.

Let $G$ be a group. A \emph{$G$-action} on the category $\C$ is an assignment 
$g\mapsto F_g$ of an automorphism $F_g:\C\to\C$ to each element $g\in G$, such that
$F_g\circ F_h=F_{gh}$ for all $g,h\in G$.
The action is said to be \emph{$k$-linear} if each $F_g$ is $k$-linear.

\begin{dfn}\label{define admissible}
  Let $\C$ be a  $k$-linear Krull-Schmidt category, and $G$ a group.
  A $G$-action on $\C$ is \emph{admissible} if $gx\not\simeq x$ holds for every
  $x\in\ind\C$ and $g\in G\setminus\{1\}$.
\end{dfn}

Note that any autoequivalence $\phi$ of $\C$ gives rise to an automorphism of the skeleton
$\C'$ of $\C$ and therefore a $\Z$-action on $\C'$. 
In particular, an autoequivalence $\phi$ of $\widehat\L$ is admissible (as defined in the sentence following Theorem~\ref{riedtmann}) if and only if
the induced action of the additive group $\Z$ on a skeleton of $\C$, given by
$m\mapsto\phi^m$, is admissible.

Let $\phi$ be an autoequivalence of $\C$, and $\phi\inv$ a chosen quasi-inverse of $\phi$.
In case $\phi$ is an automorphism, we always assume $\phi\inv$ to be the inverse of
$\phi$.
We denote by
\begin{equation} \label{inducedphi}
\phi_*:\mod\C\to\mod\C\ \mbox{ and }\ \phi^*:\mod\C\to\mod\C 
\end{equation}
the induced automorphisms of the module category defined by 
$\phi_*(M) = M\circ\phi\inv$ and $\phi^*(M)=M\circ\phi$ respectively.
Hence, any action of a group $G$ on $\C$ induces a $G$-action and a $G^{\rm op}$-action on
$\modC$, defined by $g\mapsto g_*$ and $g\mapsto g^*$ respectively.
Clearly, these actions induce actions on the stable module category and the derived
category too.
A subcategory $\U$ of $\modC$ or $\stmod\C$ is said to be \emph{$G$-equivariant} if
$g_*\U=\U$ for all $g\in G$.

The following is an immediate consequence of Auslander--Reiten duality (see, e.g., \cite[Theorem~IV:2.13]{ass06}) for $k$-linear categories. 

\begin{prop}\label{AR duality}
Let $\C$ be a locally bounded self-injective $k$-linear category with Nakayama functor
$\nu$.
Then $\stmod\C$ has a Serre functor $\nu_*\circ\Omega$.
\end{prop}
Since $\dml\simeq\stmod\widehat{\L}$, uniqueness of the Serre functor gives us the following
diagram, which is commutative up to isomorphism of functors:
\begin{equation}\label{serrecomm0}
\xymatrix@R1.5em@C=3em{
\dml\ar[r]^{\sim\;}\ar[d]^{\nu}&\stmod\widehat{\L}\ar[d]^{\widehat\nu_*\Omega}\\
\dml\ar[r]^{\sim\;}&\stmod\widehat{\L}.
}
\end{equation}

Given an additive category $\C$ and a $G$-action on $\C$,
the \emph{orbit category} $\C/G$ is defined as follows:
the objects of $\C/G$ are all object in $\C$, and morphism sets are defined by
\[\Hom_{\C/G}(x,y)=\bigoplus_{g\in G}\Hom_{\C}(x,gy)\]
for any $x,y\in\C$.
The composition of $a=(a_g)_{g\in G}\in\Hom_{\C/G}(x,y)$ with 
$b=(b_g)_{g\in G}\in\Hom_{\C/G}(y,z)$, where
$a_g\in\Hom_{\C}(x,gy)$ and $b_g\in\Hom_{\C}(y,gz)$,
is given by
\[ba=\left(\sum_{h\in G}h(b_{h\inv g})a_h\right)_{g\in G}
\in\bigoplus_{g\in G}\Hom_{\C}(x,gz)=\Hom_{\C/G}(x,z).\]
We remark that if $\C$ is self-injective, then the orbit category $\C/G$ is again
self-injective.
Notice also that $\C/G$ is not necessarily idempotent complete, although the covering functor $F:\C\to\C/G$ induces a bijection between
$(\ind\C)/G$ and $\ind(\C/G)$.

Let $\U$ be a full subcategory of an additive category $\C$.
The \emph{additive closure} $\add\U$ of $\U$ is the smallest full subcategory of $\C$
which contains $\U$ and is closed under direct sums and direct summands.

The subcategory $\U$ is \emph{contravariantly finite} in $\C$ if every $x\in\C$ has a 
\emph{right $\U$-approximation}, that is, a morphism $f:y\to x$ with $y\in\U$ such that
$\Hom_{\C}(u,f):\Hom_{\C}(u,y)\to\Hom_{\C}(u,x)$ is surjective for all $u\in\U$.
Dually, the concept of a \emph{covariantly finite} subcategory is defined. 
A subcategory that is both contravariantly and covariantly finite in $\C$ is said to be
\emph{functorially finite} in $\C$.

\section{Proofs of our results} \label{proofs}

The proof of Theorem~\ref{stronger theorem} is the main objective of this
section. Once in place, the remaining results in Section~\ref{ourresults} follow
relatively straightforwardly.

\subsection{Lemmata} \label{lemmata}

In this subsection, we collect some fundamental facts about orbit categories that will be
used in the proofs of the theorems in Section~\ref{ourresults}.
While some of these results have already appeared in the literature in some form (e.g.,
in \cite{asashiba97,bl14,gabriel2}), we include the results as well as their proofs here,
for the convenience of the reader.

\begin{lma}\label{lbddlma}
  Let $\C$ be a $k$-linear Krull--Schmidt category.
  \begin{enumerate}
  \item
    Let $G$ be a group acting admissibly on $\C$. Then the orbit category $\C/G$ is
    locally bounded if and only if $\C$ is locally bounded.
    \label{lbddorbit}
  \item
    Let $\mathcal {C}$ be a locally bounded $k$-linear category, and $\mathcal {U}$ a full subcategory of $\mod\mathcal {C}$ which is a generator. Then the following conditions are equivalent.
    \begin{enumerate}
    \item[(i)] $\mathcal {U}$ is locally bounded.
    \item[(ii)] $\mathcal {U}$ is locally bounded and functorially finite in $\mod\mathcal {C}$.
    \item[(iii)] the set $\{U\in\ind\mathcal {U}\mid U(x)\neq0\}$ is finite for all $x\in\ind\mathcal {C}$.
    \item[(iv)]
      The full subcategory $\underline{\mathcal {U}}$ of $\stmod\mathcal {C}$ is locally bounded and contravariantly finite.
    \end{enumerate}
    \label{lbddcriterion}
  \end{enumerate}
\end{lma}

\begin{proof}
\eqref{lbddorbit}
Recall first that $x,y\in\ind\C$ are isomorphic in $\C/G$ if and only if $y\simeq gx$ for some
$g\in G$. Since $G$ acts freely on $\ind\C$, it follows that, for any $x\in\ind\C$,
$$\sum_{y\in\ind(\C/G)}\dim\Hom_{\C/G}(x,y) = 
\sum_{y\in\ind(\C/G)}\dim\left(\bigoplus_{g\in G}\Hom_{\C}(x,gy)\right) =
\sum_{y\in\ind\C}\dim\Hom_{\C}(x,y) $$
and similarly that   
$$\sum_{y\in\ind(\C/G)}\dim\Hom_{\C/G}(y,x) =   \sum_{y\in\ind\C}\dim\Hom_{\C}(y,x).$$ 
Hence $\C/G$ is locally bounded if and only if $\C$ is locally bounded. 

\eqref{lbddcriterion}
The implications (ii)$\Rightarrow$(i) and (ii)$\Rightarrow$(iv) are clear.

(i)$\Rightarrow$(iii):
Take $x\in \ind \mathcal {C}$ and let $P=\Hom_{\mathcal {C}}(-,x)\in \mathop {\mathrm{mod}}\nolimits  \mathcal {C}$.
Then $\{U\in \ind \mathcal {U}\mid U(x)\ne 0\}=\{U\in \ind \mathcal {U}\mid \Hom _{\mathcal {C}}(P,U)\ne 0\}$ and, since
$P\in \mathcal {U}$, locally boundedness of $\mathcal {U}$ implies that this set is finite.

(iii)$\Rightarrow$(ii):
Let $X\in \mathop {\mathrm{mod}}\nolimits  \mathcal {C}$. The projective cover $P\to X$ of $X$ induces a monomorphism
$\Hom _\mathcal {C}(X,Y)\to \Hom _\mathcal {C}(P,Y)$ for any $Y\in \mathop {\mathrm{mod}}\nolimits  \mathcal {C}$. Now, since the set
$\{U\in \ind \mathcal {U}\mid \Hom _{\mathcal {C}}(P,U)\ne 0\}$ is finite, so is
$\{U\in \ind \mathcal {U}\mid \Hom _{\mathcal {C}}(X,U)\ne 0\}$.
Similarly, one proves that the set
$\{U\in \ind \mathcal {U}\mid \Hom _{\mathcal {C}}(U,X)\ne 0\}$
is finite.

(iv)$\Rightarrow$(iii):
Consider the following auxiliary condition:
\begin{enumerate}
\item[(iii)$'$] The set $\{U\in\ind\mathcal {U}\mid (\mathop{\rm top}\nolimits U)(x)\neq0\}$ is finite for all $x\in\ind\mathcal {C}$.
\end{enumerate}
First, we prove (iv)$\Rightarrow$(iii)$'$.
Clearly, $\mathcal {U}$ is contravariantly finite in $\mathop{\mathrm{mod}}\nolimits \mathcal {C}$.
For $x\in\ind\mathcal {C}$, let $S_x$ be the simple $\mathcal {C}$-module supported at $x$, and $a: U_x\to S_x$ its right $\mathcal {U}$-approximation.
Since $\underline{\mathcal {U}}$ is locally bounded, it suffices to prove that
\[\{U\in\ind\mathcal {U}\mid(\mathop{\rm top}\nolimits U)(x)\neq0\} \subset \{U\in\ind\mathcal {U}\mid\underline{\Hom}_{\mathcal {C}}(U,U_x)\neq0\} \cup \{\mathcal {C}(-,x)\}.\]
Clearly $\mathcal {C}(-,x)$ is the unique indecomposable projective $\mathcal {C}$-module $U$ satisfying $(\mathop{\rm top}\nolimits U)(x)\neq0$. Let $U$ be an indecomposable non-projective $\mathcal {C}$-module satisfying $(\mathop{\rm top}\nolimits U)(x)\neq0$.
Then there exists a surjective morphism $f:U\to S_x$, and a morphism $g:U\to U_x$ satisfying $f=ag$.
Since $U$ is non-projective, $f$ is non-zero in $\stmod\mathcal {C}$, and so is $g$. Thus $\underline{\Hom}_{\mathcal {C}}(U,U_x)\neq0$, as desired.

Next, we prove (iii)$'\Rightarrow$(iii). Fix $x\in\ind\mathcal {C}$. Since $\mathcal {C}$ is locally bounded, it suffices to prove
\[\{U\in\ind\mathcal {U}\mid U(x)\neq0\}\subset\bigcup_{{y\in\ind\mathcal {C}},\ {\Hom_{\mathcal {C}}(x,y)\neq0}}\{U\in\ind\mathcal {U}\mid (\mathop{\rm top}\nolimits U)(y)\neq0\}.\]
If $U\in\ind U$ satisfies $U(x)\neq0$, then the projective cover $f:P\to U$ satisfies $P(x)\neq0$. Thus $P$ has an indecomposable direct summand $\Hom_{\mathcal {C}}(-,y)$ satisfying $\Hom_{\mathcal {C}}(x,y)\neq0$. Then $(\mathop{\rm top}\nolimits U)(y) = (\mathop{\rm top}\nolimits P)(y)\neq0$.
\end{proof}

We remark that the assumption that $\mathcal {U}$ is a generator is only used in the proof of the implication (i)$\Rightarrow$(iii) above.
Indeed, the conditions (ii)--(iv) in Lemma~\ref{lbddlma}(b) are equivalent even without this assumption on $\mathcal {U}$.

\begin{lma} \label{ab}
Let $\U$ be a $k$-linear Krull-Schmidt category, and $\psi$ and $\phi$ autoequivalences of
$\U$ such that $\psi\phi(x)\simeq\phi\psi(x)$ for all $x\in\ind\U$.
If $\phi$ is admissible, and the number of $\psi$-orbits of $\ind\U$ is finite, then
the number of $\phi$-orbits of $\ind\U$ is finite.
\end{lma}

\begin{proof}
Since $\phi$ is an autoequivalence, and $\psi\phi(x)\simeq\phi\psi(x)$ for all
$x\in\ind\U$, the action of $\phi$ induces a permutation $\s$ of the $\psi$-orbits of
$\ind\U$:
$\s(o_\psi(x))=o_\psi(\phi(x))$, where $o_\psi(x)$ is the $\psi$-orbit of $x$.
Since the number of $\psi$-orbits is finite, there exists a positive integer $a$ such that $\s^a=\I$; that is, such that for
every $x\in\ind\U$,  
there exists a number $b_x\in\Z$ such that $\phi^a(x)\simeq\psi^{b_x}(x)$.
The number $b_x$ must be non-zero for every $x$, otherwise $\phi^a(x)\simeq x$,
contradicting admissibility. 
Let $S$ be a set of representatives of the $\psi$-orbits of $\ind\U$, so that
$\ind\U=\{\psi^i(x)\mid x\in S,\ i\in\Z\}$.
Since $\phi^a(\psi^i(x))\simeq\psi^{b_x+i}(x)$ holds for all $x\in\ind\C$ and $i\in\Z$, 
it follows that the finite set
$\{\psi^i(x) \mid x\in S,\ 0\le i<|b_x|\}$ meets all $\phi$-orbits in $\ind\U$.
\end{proof}

Under a few additional assumptions, one can show that there exist non-zero $a,b\in\Z$ such
that 
$\phi^a(x)\simeq\psi^b(x)$ for all $x\in\ind\U$. We record it here though we do not use in
this paper.

\begin{rmk} \label{connectedab}
Suppose, in addition to the assumptions in Lemma~\ref{ab}, that $\psi$ is admissible, that
the orbit category $\U/\psi$ is indecomposable and that for every 
$x\in\ind\U$ there exist only finitely many $y\in\ind\U$ such that $\Hom_\U(x,y)\ne0$. Then
there exist non-zero $a,b\in\Z$ such that $\phi^a(x)\simeq\psi^b(x)$ for all $x\in\ind\U$.
\end{rmk}

\begin{proof}
We use the same notation as in the proof of Lemma~\ref{ab}. It suffices to show that
$b_x$ is the same for all $x\in\ind\U$.
Assume that $x,y\in\ind\U$ and $\Hom_\U(x,y)\ne0$. 
Then, for all $i\in\Z$, 
\begin{equation*}
  0\ne \Hom_{\U}(x,y)\simeq \Hom_{\U}(\phi^{ai}x,\phi^{ai}y) 
  \simeq\Hom_{\U}(\psi^{b_xi}(x),\psi^{b_yi}(y))\simeq \Hom_{\U}(x,\psi^{(b_y-b_x)i}(y))
\end{equation*}
holds. Admissibility of $\psi$ together with the finiteness assumption implies
$b_x=b_y$. Since $\U/\psi$ is indecomposable, we have the desired assertion.
\end{proof}

\begin{lma}\label{admissibility}
Let $\L$ be a finite-dimensional $k$-algebra. 
For any automorphism $\phi$ of $\widehat\L$, the following conditions are equivalent.
\begin{enumerate} \renewcommand{\labelenumi}{(\alph{enumi})}
\item the automorphism $\phi$ of $\widehat\L$ is admissible;
\item the category $\widehat\L/\phi$ is Hom-finite;
\item the set $\ind(\widehat\L/\phi)$ is finite.
\end{enumerate}
\end{lma}

\begin{proof}
(a)$\Rightarrow$(b): 
Immediate from Lemma~\ref{lbddlma}\eqref{lbddorbit}.

(b)$\Rightarrow$(a):
Let $x\in\ind\widehat\L$. If $\phi^{\ell} x\simeq x$ for some $\ell$, then
$\End_{\widehat\L/\phi}(x)=\bigoplus_{i\in\Z}\Hom_{\widehat\L}(x,\phi^ix)$
is infinite dimensional, a contradiction.

(a)$\Rightarrow$(c):
The $\widehat\nu$-orbits of $\ind\widehat\L$ are indexed by the indecomposable projective 
$T(\L)$-modules, hence the number of orbits is finite.
Moreover, we have $\widehat\nu \phi \simeq \phi \widehat\nu$ by \eqref{commutes} and since
$\phi$ is admissible, Lemma~\ref{ab} implies that the number of $\phi$-orbits of
$\ind\widehat\L$ is finite, \ie, $\ind(\widehat\L/\phi)$ is finite.

(c)$\Rightarrow$(a):
We may assume that $\L$ is indecomposable as a ring. Then the category $\widehat\L$ is
indecomposable. 
Now if $f:x\to y$ is a non-zero morphism between two indecomposable objects in
$\widehat\L$, then $\phi^i(f)\in\Hom_\C(\phi^ix,\phi^iy)$ is non-zero for all
$i\in\Z$. Since $\widehat\L$ is locally bounded, it follows that the $\phi$-orbit 
of $y$ is finite if and only if so is the $\phi$-orbit of $x$. Since $\widehat\L$ is
indecomposable, either all $\phi$-orbits are finite, or all $\phi$-orbits are infinite. The
assumption that $\ind(\widehat\L/\phi)=(\ind\widehat\L)/\phi$ is finite implies that the
orbits must be infinite, from which follows that $\phi$ is admissible.
\end{proof}

A full subcategory $\U$ of $\modC$ is called \emph{$d$-rigid} if $\Ext_{\C}^i(X,Y)=0$ for
all $i\in\{1,\ldots,d-1\}$ and $X,Y\in\U$. By definition, any $d$-cluster-tilting
subcategory of $\modC$ is $d$-rigid.
Denote by $(\mod\C)/G$ the orbit category of the $G$-action on $\mod\C$ given by 
$g\mapsto g_*$.

A functor $F:\C\to\mathcal{D}$ between additive categories $\C$ and $\mathcal{D}$ is
said to be an \emph{equivalence up to summands} if it is full and faithful, and
satisfies $\add F(\C)=\mathcal{D}$. 
The natural functor $\C\to\proj\C,\:x\mapsto\Hom_{\C}(-,x)$ is an equivalence up to
summands for any additive category $\C$. It is an equivalence of categories if and only if
$\C$ is idempotent complete.

\begin{lma} \label{ext}
 Let $\C$ be a locally bounded $k$-linear Krull--Schmidt category, and $G$ a group acting on $\C$. 
\begin{enumerate}
\item \label{projeq} 
  The push-down functor $F_*$ induces functors
  $(\proj\C)/G\to\proj(\C/G)$ and $(\inj\C)/G\to\inj(\C/G)$, which are equivalences up to summands.
\item \label{extsum}
  For all $i\ge0$ and $X,Y\in\modC$, there is a functorial isomorphism
  $$\Ext^i_{\mathcal{C}/G}(F_*(X),F_*(Y)) \simeq \bigoplus_{g\in G}\Ext^i_{\mathcal{C}}(X,g_*Y).$$
\item \label{pushdown} 
  The functor $F_*:\mod\C\to\mod(\C/G)$ induces a fully faithful
  functor $F_*:(\mod\C)/G\to\mod(\C/G)$.
\item \label{rigid} 
  Let $\U\subset\modC$ be a $G$-equivariant full subcategory.
  Then $\U$ is $d$-rigid if and only if $F_*(\U)\subset\modCG$ is $d$-rigid. 
\end{enumerate}
\end{lma}

\begin{proof}
In view of the equivalence of categories $\proj\C\simeq\C$, and the compatibility of this
equivalence with the $G$-actions on $\proj\C$ and $\C$ respectively, we have that
$$(\proj\C)/G\simeq \C/G\to \proj(\C/G)$$
is an equivalence up to summands.
The equivalence up to summands $(\inj\C)/G\to \inj(\C/G)$ follows similarly from the
equivalence $\inj\C\simeq\C$.
This proves the statement \eqref{projeq}.

For the remaining statements, it suffices to prove \eqref{extsum}.
It is well-known \cite[Section~6]{bl14} that there is an isomorphism
$F^*F_*\simeq\bigoplus_{g\in G}g_*$ of functors $\Mod\C\to\Mod\C$. Thus for any
$X,Y\in\mod\C$, we have functorial isomorphisms 
\[\Hom_{\C/G}(F_*(X),F_*(Y))\simeq\Hom_{\C}(X,F^*F_*(Y))\simeq
\bigoplus_{g\in G}\Hom_{\C}(X,g_*Y) \,,\]
and so the assertion holds for $i=0$. Now let 
$\cdots\xrightarrow{f_3} P_2 \xrightarrow{f_2} P_1\xrightarrow{f_1} P_0 \to 0$
be a projective resolution of $X$ in $\modC$.
By \eqref{projeq} and the exactness of the functor $F_*$,
\[\cdots\xrightarrow{F_*(f_3)} F_*(P_2) \xrightarrow{F_*(f_2)} F_*(P_1)\xrightarrow{F_*(f_1)} F_*(P_0) \to 0\]
is a projective resolution of $F_*(X)$ in $\modCG$. Thus we have
\begin{align*}
  \Ext^i_{\C/G}(F_*(X),F_*(Y)) &= 
\frac{\Ker\Hom_{\mathcal{C}/G}(F_*(f_{i+1}),F_*(Y))}{\Im\Hom_{\C/G}(F_*(f_i),F_*(Y))} =
\frac{\Ker\left(\bigoplus_{g\in G}\Hom_{\mathcal{C}}(f_{i+1},g_*Y)\right)}{\Im\left(\bigoplus_{g\in G}\Hom_{\mathcal{C}}(f_i,g_*Y)\right)}
\\
&= \bigoplus_{g\in G}\frac{\Ker\Hom_{\C}(f_{i+1},g_*Y)}
         {\Im\Hom_{\mathcal{C}}(f_i,g_*Y)} = \bigoplus_{g\in G}\Ext^i_{\mathcal{C}}(X,g_*Y).\qedhere
\end{align*}
\end{proof}

We denote by $\rad_\C$ the \emph{Jacobson radical} of the category $\C$ (see, e.g.,
  \cite[A.3]{ass06}).
In the following lemma, we collect some basic properties of orbit categories; 
cf.\ \cite[Section~2]{asashiba97}.

\begin{lma} \label{indlemma}
  Let $\C$ be a locally bounded $k$-linear Krull-Schmidt category,
  and $G$ a group acting admissibly on $\C$.
  \begin{enumerate}
  \item We have $\rad_{\C/G} = (\rad_\C)/G$, that is, 
    $\rad_{\C/G}(x,y)= \bigoplus_{g\in G}\rad_{\C}(x,gy)$ for all $x,y\in\C$. 
    In particular, $\C/G$ is Krull--Schmidt, and  $\End_{\C/G}(x)/\rad_{\C/G}(x,x)=\End_{\C}(x)/\rad_{\C}(x,x)$ for all
  $x\in\ind\C$.  
  \label{radical}
  \item 
    Let $x\in\C$, $y\in\ind\C$ and $a=(a_g)_{g\in G}\in\Hom_{\C/G}(x,y)$. 
    If $a_h\in\Hom_\C(x,hy)$ is a retraction for some $h\in G$, then $a$ is a
    retraction.
    \label{retraction}
  \item 
    Let $X\in\modC$ be indecomposable. If $g_*X\not\simeq X$ for all 
    $g\in G\setminus\{1\}$, then $F_*(X)\in\modCG$ is indecomposable.
    \label{pdind}
  \item 
    For any $X\in\mod\C$, the subgroup $G_X=\{g\in G \mid g_*X\simeq X\}$ of $G$ is finite. In particular, if $G$ is torsion free, then $G$ acts admissibly on $\mod\C$.
    \label{torsion2}
  \item 
    If $G$ acts admissibly on $\mod\C$, then the push-down functor
    $F_*:\mod\C\to\mod(\C/G)$ preserves indecomposability. 
    \label{tfadmissible}
  \end{enumerate}
\end{lma}

\begin{proof}
(\ref{radical})
Set $I=(\rad_{\C})/G$.
First, since $\C$ is locally bounded, any $x\in\C$ satisfies $\rad_{\C}^n(x,-)=0$ for some
$n\gg0$. Thus the ideal $I(x,x)=\bigoplus_{g\in G}\rad_{\C}(x,gx)$ of $\End_{\C/G}(x)$ is
nilpotent, and hence $I$ is contained in $\rad_{\C/G}$.
It remains to show that the factor category $(\C/G)/I$ is semisimple.
Setting $\overline{\C}=\C/\rad_{\C}$, it is clear that $(\C/G)/I=\overline{\C}/G$.
Since the natural functor $\C\to\overline{\C}$ gives a bijection
$\ind\C\simeq\ind\overline{\C}$, it follows that $G$ acts on $\overline{\C}$ admissibly.
From the semisimplicity of $\overline{\C}$ it follows that 
$\End_{\overline{\C}/G}(x)\simeq\End_{\overline{\C}}(x)$ and $\Hom_{\overline{\C}/G}(x,y)=0$
hold for all non-isomorphic $x,y\in\ind\overline{\C}$.
Thus, $\overline{\C}/G$ is semisimple.

(\ref{retraction})
Assume that $a_h\in\Hom_{\C}(x,hy)$ is a retraction. Let $s:hy\to x$ be such that
$a_hs=\mathbb{I}_{hy}\in\End_\C(hy)$, and define $b=(b_g)_{g\in G}\in\Hom_{\C/G}(y,x)$ by
$b_{h\inv}=h\inv(s)\in\Hom_{\C}(y,h\inv x)$ and $b_g=0$ for $g\ne h\inv$.
Hence $ab\in\End_{\C/G}(y)$ satisfies $(ab)_1=\I_y$, while for $g\ne1$ we have
$y\not\simeq gy$, implying $(ab)_g\in\Hom_\C(y,gy)=\rad_\C(y,gy)$.
Consequently, $ab-\I_y\in \rad_{\C/G}(y,y)$ by \eqref{radical}, so $ab\in\End_{\C/G}(y)$ is invertible. Hence
$a$ is a retraction.

(\ref{pdind})
First, observe that from the assumptions follow that
$G$ acts admissibly on $\C'=\add\{g_*X\mid g\in G\}$. Using
Lemma~\ref{ext}\eqref{pushdown} and applying the second identity in \eqref{radical} to
$\C=\C'$, $x=X$, we get 
\[  \frac{\End_{\C/G}(F_*(X))}{\rad_{\modCG}(F_*(X),F_*(X))} = 
  \frac{\End_{(\mod\C)/G}(X)}{\rad_{(\modC)/G}(X,X)} =
   \frac{\End_\C(X)}{\rad_{\modC}(X,X)}\,,\]
which is a division algebra, since $\End_{\C}(X)$ is local.
Hence, $\End_{\C/G}(F_*(X))$ is local, so $F_*(X)$ is indecomposable.

(\ref{torsion2})
Since the group $G_X$ acts admissibly on the support $\Supp X$ of $X$, which is a
finite subset of $\ind\C$, it follows that $|G_X|\le|\Supp X|$.

(\ref{tfadmissible})
The first assertion follows from \eqref{pdind}, and the second one follows from
\eqref{torsion2}.
\end{proof}

\begin{lma} \label{functfin} 
Let $\C$ be a locally bounded $k$-linear Krull-Schmidt category, and $G$ a
group acting on $\C$.
If $\U\subset\modC$ is a full subcategory such that $F_*(\U)$ is functorially finite in
$\modCG$, then $\U$ is functorially finite in $\modC$.
\end{lma}

\begin{proof}
We only prove that $\U\subset\modC$ is covariantly finite. Observe that, by
Lemma~\ref{ext}\eqref{pushdown}, the image in $\modCG$ of the push-down functor is
equivalent to the orbit category $(\modC)/G$. Set $\V=F_*(\U)\simeq\U/G$. 

Fix $X\in\modC$. By our assumptions, there exists a left $\V$-approximation $a:F_*(X)\to V$
in $\mod(\C/G)$, and $V=F_*(W)$ for some $W\in\U$. 
Viewing $a$ as a morphism in $(\modC)/G$, we write
$a=(a_g)_{g\in G}$, where $a_g\in\Hom_{\C}(X,g_*W)$. 
Then the set $I=\{g\in G\mid a_g\neq0\}$ is finite.  
Set $U=\bigoplus_{g\in I}g_*W\in\modC$ and $b=(a_g)_{g\in I}\in\Hom_{\C}(X,U)$. We show
that $b$ is a left $\U$-approximation in $\modC$. 

Take any $c\in\Hom_{\C}(X,Y)$ with $Y\in\U$. For $F_*(c)\in\Hom_{\C/G}(F_*(X),F_*(Y))$,
there exists $d\in\Hom_{\C/G}(V,F_*(Y))$ such that $F_*(c)=da$, since $a$ is
a left $\V$-approximation in $\modCG$. Write $d=(d_g)_{g\in G}$, where
$d_g\in\Hom_{\C}(W,g_*Y)$. Then we have 
\[c=\sum_{h\in G}h_*(d_{h\inv})a_h = \sum_{h\in I}h_*(d_{h\inv})a_h \,,\]
which shows that $e=(h_*(d_{h\inv}))_{h\in I}\in\Hom_{\C}(U,Y)$ satisfies $c=eb$. Thus the assertion follows.
\end{proof}

\begin{proof}[Proof of Remark~\ref{stable_lbdd}.]
  Let $\mathcal{U}'\subset \mathop{\mathrm{mod}}\nolimits\widehat\Lambda$ be the preimage of $\mathcal{U}$ under the natural functor $\mathop{\mathrm{mod}}\nolimits\widehat\Lambda\to\stmod\widehat\Lambda$.
  From Lemma~\ref{ext}(c) follows that $\mathcal{U}'/\phi_*\simeq F_*(\mathcal{U}')$.
  As the $\phi$-action on $\mathop{\mathrm{mod}}\nolimits \widehat{\Lambda}$ is admissible by Lemma~\ref{indlemma}(d), we have 
\begin{align*}
  \mathcal{U} \mbox{ is locally bounded} \:&\Leftrightarrow\:
  \mathcal{U}' \mbox{ is locally bounded}
  &&\mbox{by Lemma~\ref{lbddlma}(b)}, \\
  \:&\Leftrightarrow\: 
  \mathcal{U}'/\phi_*\simeq F_*(\mathcal{U}') \mbox{ is locally bounded}
  &&\mbox{by Lemma~\ref{lbddlma}(a)}, \\
  \:&\Leftrightarrow\: 
  \ind(\mathcal{U}'/\phi_*) \mbox{ is finite}
  &&\mbox{by Lemma~\ref{lbddlma}(b)}, \\
  \:&\Leftrightarrow\: 
\ind(\mathcal{U}/\phi_*) \mbox{ is finite}
\end{align*}
where we used the fact that, by Lemma~\ref{admissibility}, $\ind(\widehat{\Lambda}/\phi)$ is finite.
\end{proof}

\subsection{Proof of Theorem~\ref{stronger theorem} and Corollary~\ref{bijection}}

First we prove the `if' part of Theorem~\ref{stronger theorem}(a).

\begin{proof}[Proof of `if' part of Theorem~\ref{stronger theorem}(a)]
It follows from Lemma~\ref{ext}\eqref{rigid} and Lemma~\ref{functfin} that $\U$ is $d$-rigid and functorially finite in $\mod(\C)$. 

Let $X\in\modC$.  
By Lemma~\ref{ext}\eqref{extsum} we have, for all $U\in\U$,
$$\Ext^i_{\C/G}(F_*(X),F_*(U))\simeq\bigoplus_{g\in G}\Ext^i_{\C}(X,g_*U)\,.$$
Since $g_*U\in\U$ for all $g\in G$, it follows that $\Ext^i_{\C}(X,U)=0$ for all $U\in\U$ and $i\in\{1,\ldots,d-1\}$ if and only if
$\Ext^i_{\C/G}(F_*(X),F_*(U))=0$ for all $U\in\U$ and $i\in\{1,\ldots,d-1\}$. As $F_*(\U)$ is a $d$-cluster-tilting subcategory of $\modCG$, this is true precisely when 
$F_*(X)\in F_*(\U)$, that is, by Lemma~\ref{ext}\eqref{pushdown}, when $X\in\U$.

By a similar argument, one can show that $\Ext^i_{\C}(U,X)=0$ for all $U\in\U$ and $i\in\{1,\ldots,d-1\}$ if and only if $X$ belongs to $\U$.
\end{proof}

For the proof of the rest of Theorem~\ref{stronger theorem} we shall need a result,
Proposition~\ref{exseq} below, which characterizes $d$-cluster-tilting subcategories of
$\mod\C$ in terms of the existence of so-called $d$-almost split
sequences.
This is a generalization of \cite[Proposition~2.4]{io11}, which deals with the case
of a finite-dimensional $k$-algebra.

Recall that a Hom-finite $k$-linear Krull-Schmidt category $\C$ is called a \emph{dualizing $k$-variety} \cite{ARstable}
if the functors $D:\Mod\C\to \Mod(\C^{\rm op})$ and $D:\Mod(\C^{\rm op})\to \Mod\C$ induce
dualities 
\[D:\mod\C\to\mod(\C^{\rm op})\ \mbox{ and }\ D:\mod(\C^{\rm op})\to \mod\C\]
respectively. In this case, $\mod\C$ is an abelian subcategory of $\Mod\C$ which is closed
under kernels, cokernels and extensions, and has enough projective objects and injective
objects. Moreover, all simple objects in $\Mod\C$ and $\Mod(\C^{\rm op})$ are finitely
presented. 
This implies that for every indecomposable object $x\in\C$ there exists a right
  almost split morphism $f:y\to x$, and a left almost split morphism $g:x\to z$.
A morphism $f\in\Hom_{\C}(y,x)$ is said to be 
\emph{right almost split} if it is not a retraction, and any non-retraction
$g\in\Hom_{\C}(z,x)$ factors through $f$. Left almost split morphisms are defined dually.

For example, by \eqref{locbdd-dualising} in Section~\ref{section: notation}, 
any locally bounded $k$-linear category is a dualizing $k$-variety. 
If $\C$ is a dualizing $k$-variety then so is $\mod\C$ \cite[Proposition 2.6]{ARstable}.
A $k$-linear triangulated category is a dualizing $k$-variety if and only if it has a
Serre functor \cite[Proposition 2.11]{iy08}.
Furthermore, any functorially finite subcategory of a dualizing $k$-variety is again a
dualizing $k$-variety \cite[Proposition 1.2]{iyama07b}.

\begin{prop} \label{globaldim}
Let $\C$ be a dualizing $k$-variety, and $\U$ a functorially finite subcategory of $\mod\C$. 
\begin{enumerate} \renewcommand{\labelenumi}{(\alph{enumi})}
\item The category $\U$ is a dualizing $k$-variety; in particular, $\mod\U$ is abelian.
\item The following identities hold:
\begin{eqnarray*}
\gldim(\mod\U)&=& \sup\{\injdim S\mid S\mbox{ is a simple $\U$-module}\}\\
&=& \sup\{\projdim S\mid S\mbox{ is a simple $\U$-module}\}.
\end{eqnarray*}
\end{enumerate}
\end{prop}

\begin{proof}
(a) The assertion follows from the remarks in the paragraph preceding the proposition.

(b) Since any object $M\in\mod\U$ has a minimal projective (respectively, injective) resolution in $\mod\U$,
we have that $\projdim M\le\ell$ (respectively, $\injdim M\le\ell$) if and only if
$\Ext^{\ell+1}_{\U}(M,S)=0$ (respectively, $\Ext^{\ell+1}_{\U}(M,S)=0$) 
for any simple $\U$-module $S$.
Thus $\gldim(\mod\U)$ is the supremum of $\injdim S$ (respectively, $\projdim S$) for simple $\U$-modules $S$.
\end{proof}

The following result is a category version of \cite[Proposition~2.4]{io11} (see also \cite[Theorem 5.1(3)]{iyama07b}).

\begin{prop}\label{exseq}
Let $\C$ be a dualizing $k$-variety, and $\U$ an additively closed functorially
finite $d$-rigid subcategory of $\mod\C$ that is a generator-cogenerator. 
The following conditions are equivalent.
\begin{enumerate} \renewcommand{\labelenumi}{(\alph{enumi})}
\item $\U$ is a $d$-cluster-tilting subcategory of $\mod\C$.
\item For any $X\in\modC$, there exists an exact sequence
$$0\to M_d\to\cdots\to M_2\to M_1\to X\to0$$
with $M_i\in\U$.
\item For any indecomposable object $X\in\U$, there exists an exact sequence
$$0\to M_d\to\cdots\to M_1\to M_0\stackrel{f}{\to} X$$
with $M_i\in\U$ and $f$ right almost split in $\U$. 
\item $\gldim(\mod\U)\le d+1$.
\end{enumerate}
\end{prop}

\begin{proof}
(a)$\Leftrightarrow$(b) This is shown in
\cite[Proposition~2.2.2(1)$\Leftrightarrow$(2-0)]{iyama07}.

(b)$\Rightarrow$(c) Let $X\in\U$ be an indecomposable object. Since $\U$ is a dualizing
$k$-variety, there exists a right almost split morphism $f_0:M_0\to X$ in $\U$. Let
$g:K\to M_0$ be kernel of $f_0$ in $\mod\C$.  
By our assumption (b), there exists an exact sequence
$0\to M_d\xrightarrow{f_d}\cdots\xrightarrow{f_3}M_2\xrightarrow{f_2}M_1\xrightarrow{h}K\to0$
with $M_i\in\U$. 
Let $f_1=gh:M_1\to M_0$, then the sequence 
\[0\to M_d\xrightarrow{f_d}\cdots\xrightarrow{f_3}M_2\xrightarrow{f_2}M_1\xrightarrow{f_1}M_0\xrightarrow{f_0}X\]
has the desired properties.

(c)$\Rightarrow$(d) 
Any simple $\U$-module $S$ can be written as $S=\Hom_{\U}(-,X)/\rad_{\U}(-,X)$ for some indecomposable object $X\in\U$.
Applying the Hom-functor to the sequence given in the condition (c),
and using that $\mathcal{U}$ is $d$-rigid,
we get an exact sequence
\[0\to\Hom_{\U}(-,M_d)\to\cdots\to\Hom_{\U}(-,M_0)\to\Hom_{\U}(-,X)\to S\to 0.\]
So $S$ has projective dimension at most $d+1$, whence Proposition~\ref{globaldim}(b)
shows that $\gldim(\mod\U)\le d+1$.

(d)$\Rightarrow$(b) Fix $X\in\modC$. Since $\C$ is a dualizing $k$-variety, we can take an injective copresentation
$0\to X\to I^0\to I^1$ in $\mod\C$.
This gives rise to an exact sequence
\[0\to\Hom_{\C}(-,X)|_{\U}\to\Hom_{\U}(-,I^0)\to\Hom_{\U}(-,I^1)\]
in $\mod\U$.
Since $\U$ is a generator-cogenerator, $I^i$ belongs to $\U$ for $i=0,1$. Therefore the $\U$-modules $\Hom_{\U}(-,I^i)$ are projective.
Since $\gldim(\mod\U)\le d+1$, the projective dimension of $\Hom_{\C}(-,X)|_{\U}$ is at most $d-1$.
Taking a projective resolution
$$0\to\Hom_{\U}(-,M_d)\to\cdots\to\Hom_{\U}(-,M_1)\to\Hom_{\C}(-,X)|_{\U}\to0,$$
of $\Hom_{\C}(-,X)|_{\U}$, Yoneda's Lemma gives us the desired sequence.
\end{proof}

Now we are ready to prove Theorem~\ref{stronger theorem}(a).

\begin{proof}[Proof of the `only if' part of Theorem~\ref{stronger theorem}(a)]
Let $\V=F_*(\U)$. It follows from Lemma~\ref{ext}\eqref{projeq} and \eqref{rigid} that
$\V$ is a $d$-rigid generator-cogenerator and, by assumption, $\V$ is functorially finite
in $\mod(\C/G)$. 

We shall show that $\V$ satisfies the condition (c) of Proposition~\ref{exseq}. 
Since $G$ acts admissibly on $\mod\C$,
Lemma~\ref{indlemma}\eqref{tfadmissible} gives that the push-down functor $F_*$ preserves indecomposability, implying
that $\V=\add\V$.  
Moreover, any indecomposable object in $\V$ is isomorphic to $F_*(U)$ for some
indecomposable object $U\in\U$. 
Let $f_0:U_0\to U$ be a right almost split map in $\mathcal{U}$, and
$g:K\to U_0$ its kernel in $\mod\C$.
Applying Proposition~\ref{exseq}(b) to $\U$, we get an exact sequence
$$0\to U_d\xrightarrow{f_d}\cdots\xrightarrow{f_2}U_1\xrightarrow{h}K\to 0 $$
in $\mod\C$ with $U_i\in\U$.
Setting $f_1=gh$, we obtain an exact sequence
$$0\to U_d\xrightarrow{f_d}\cdots\xrightarrow{f_2}U_1\xrightarrow{f_1}
U_0\xrightarrow{f_0}U\,.$$
Applying $F_*$ gives an exact sequence 
\begin{equation} \label{stangle}
0\to F_*(U_d)\xrightarrow{F_*(f_d)}\cdots\xrightarrow{F_*(f_2)}F_*(U_1)\xrightarrow{F_*(f_1)}
F_*(U_0)\xrightarrow{F_*(f_0)}F_*(U) 
\end{equation}
in $\mod(\C/G)$, with $F_*(U_i)\in\V$ for all $i$.

It now suffices to show that the morphism $F_*(f_0):F_*(U_0)\to F_*(U) $ is right almost split in $\V$.
Let $a:F_*(X)\to F_*(U)$ be a morphism in $\V$, that is not a retraction.
By Lemma~\ref{ext}\eqref{pushdown}, we may view $a$ as a morphism in the orbit category $(\mod\C)/G$,
whence $a=(a_g)_{g\in G}$, with $a_g\in\Hom_{\C}(X,g_*U)$.
From Lemma~\ref{indlemma}\eqref{retraction} follows that $a_g$ is not a retraction for any $g\in G$. Since the map
$g_*(f_0):g_*U_0\to g_*U$ is right almost split in $\U$, there exist $b_g:X\to g_*U_0$
such that $a_g=g_*(f_0)b_g$ for all $g\in G$.
Now $a=F_*(f_0)b$ holds, where $b=(b_g)_{g\in G}\in\Hom_{\C/G}(F_*(X),F_*(U_0))$. 
Hence $F_*(f_0)$ is a right almost split map in $\V$. 
This concludes the proof of Theorem~\ref{stronger theorem}(a).
\end{proof}

We turn now to the proof of Theorem~\ref{stronger theorem}(b). For this, the following
result by Amiot and Oppermann will be needed.

\begin{lma}\cite[Corollary 4.5]{ao}\label{amiot-oppermann}
Let $A$ be a finitely generated $\Z$-graded algebra over an algebraically closed field
$k$, and $M$ a finite-dimensional $A$-module such that $\Ext^1_A(M,M)=0$.
Then there exists a $\Z$-graded $A$-module $N$ which is isomorphic to $M$ as an ungraded
$A$-module.
\end{lma}

From Lemma~\ref{amiot-oppermann} we deduce the following proposition, which plays an
important role in the proof of Theorem~\ref{stronger theorem}(b).

\begin{prop} \label{gradable}
Let $\C$ be a locally bounded $k$-linear Krull-Schmidt category,
where $k$ is algebraically closed, and $G$ a finitely generated free abelian group, acting
admissibly on $\C$. 
Then every  module $M\in\mod(\C/G)$ satisfying $\Ext^1_{\C/G}(M,M)=0$ is in the
essential image of the push-down functor $F_*:\modC\to\mod(\C/G)$.
\end{prop}

\begin{proof}
(i) First we prove the statement for the case $G=\langle\phi\rangle\simeq\Z$.

Let $M\in\mod(\C/G)$ be a module satisfying $\Ext^1_{\C/G}(M,M)=0$.
Since $\C$ is locally bounded, the support $\Supp M$ of $M$ is finite. 
Replacing $\C$ by the full subcategory 
$\add\{\phi^i x\mid i\in\Z,\ x\in\Supp M\}\subset\C$, we may assume that $(\ind\C)/G$ is a
finite set.
Let $S$ be a complete set of representatives of the $G$-orbits of $\ind\C$, and set 
$A=\End_{\C/G}(y)$ for $y:=\bigoplus_{x\in S}x$.
Then $A$ is a finite-dimensional $\Z$-graded $k$-algebra, with grading defined by
$A_i=\Hom_{\C}(y,\phi^iy)$, $i\in\Z$.
This gives us a commutative diagram of functors
\[\xymatrix{
\mod\C\ar[r]^{\sim}\ar[d]^{F_*}&\mod^{\Z}A\ar[d]^{\rm forget}\\
\mod(\C/G)\ar[r]^{\sim}&\mod A
}\]
in which the horizontal arrows are equivalences of categories.
From Lemma \ref{amiot-oppermann} it now follows that $M\in\mod(\C/G)\simeq\mod A$ is in
the essential image of the push-down functor.

(ii) We now prove the statement for general case, by induction on the rank of the group
$G$. 

Take $\phi\in G$ such that $G'=G/\langle\phi\rangle$ is a free abelian group, and set
$\C'=\C/\phi$. Clearly, it follows that $\C/G = \C'/G'$.
Now $F:\C\to\C/G$ is the composition of the natural functors $F':\C\to \C/\phi=\C'$ and
$F'':\C'\to \C'/G'=\C/G$. Since the action of $G$ on $\C$ is admissible, the category
$\C'$ is a locally bounded $k$-linear Krull-Schmidt category, and the action of $G'$ on
$\C'$ is again admissible. By the induction hypothesis, the module $M$ is in the
essential image of $F''_*:\mod\C'\to\mod(\C/G)$. Take $N\in\mod\C'$ such that 
$M\simeq F''_*N$. Then $\Ext^1_{\C'}(N,N)=0$ by Lemma~\ref{ext}\eqref{extsum}, and by (i), the module $N$ is contained in
the essential image of $F'_*:\mod\C\to\mod\C'$.
The assertion follows.
\end{proof}

We now have the tools needed to prove Theorem~\ref{stronger theorem}(b).
The proof is similar to that of \cite[Proposition~3.1]{ao14}.

\begin{proof}[Proof of Theorem~\ref{stronger theorem}(b)]
Let $\V$ be a $d$-cluster-tilting subcategory of $\modCG$. 
If $d=1$ then $\V=\modCG$ and hence $F_*\inv(\V)=\modC$, which is a
  $1$-cluster-tilting subcategory of itself. 
Assume that $d>1$. Then $\V$ is $2$-rigid and thus, by Proposition~\ref{gradable}, it is
contained in the essential image of the push-down functor $F_*$.
Thus $\U=F_*\inv(\V)$ satisfies $F_*(\U)=\V$. Since $G$ acts admissibly on $\mod\C$ by Lemma~\ref{indlemma}\eqref{torsion2} and $\U$ is a $G$-equivariant full subcategory of $\mod\C$ such that $F_*(\U)$ is a $d$-cluster tilting subcategory of $\mod(\C/G)$, it follows from Theorem~\ref{stronger theorem}(a) that $\U$ is a $d$-cluster tilting subcategory of $\mod\C$.
\end{proof}

We now show how Corollary~\ref{bijection} follows from Theorem~\ref{stronger theorem}. 

\begin{proof}[Proof of Corollary~\ref{bijection}]
(a) 
By Lemma~\ref{ext}\eqref{pushdown}, any full subcategory $\U\subset\mod\C$ satisfies
$\U/G\simeq F_*(\U)\in\mod(\C/G)$. 
So if $\U\subset\mod\C$ is a locally bounded $d$-cluster-tilting subcategory, then 
$F_*(\U)$ is locally bounded by Lemma~\ref{lbddlma}\eqref{lbddorbit}, and thus
functorially finite in $\mod(\C/G)$ by Lemma~\ref{lbddlma}\eqref{lbddcriterion}. 
Theorem~\ref{stronger theorem}(a) now implies that $F_*(\U)$ is a $d$-cluster-tilting
subcategory of $\mod(\C/G)$.

(b) 
Assume that $k$ is algebraically closed.
If $\V$ is a locally bounded $d$-cluster-tilting subcategory of $\mod(\C/G)$,
then $\U=F_*\inv(\V)$ is a $G$-equivariant $d$-cluster-tilting subcategory of $\mod\C$ by
Theorem~\ref{stronger theorem}(b). Since $\V=F_*(\U)\simeq\U/G$,
Lemma~\ref{lbddlma}\eqref{lbddorbit} implies that $\U$ is locally bounded. 
\end{proof}

\subsection{Proofs of other results} \label{pfother}

We start by proving Theorem~\ref{basic construction}.

\begin{proof}[Proof of Theorem~\ref{basic construction}]
\eqref{basic1} First, observe that  $\widehat{\L}/\phi$ is a finite-dimensional
$k$-algebra by Lemma~\ref{admissibility}. 
Hence, for the first assertion, it suffices to show that $\widehat{\L}/\phi$ is
locally $d$-re\-pre\-sen\-ta\-tion-fin\-ite.
If $\U$ is a locally bounded $\phi$-equivariant $d$-cluster-tilting subcategory of
$\dml$ then, by Lemma~\ref{lbddlma}(b)(iv)$\Rightarrow$(i), the preimage $\U'$ of $\U$ under the natural functor
$\mod\widehat\L\to\stmod\widehat\L\simeq\dml$ is a locally bounded $\phi$-equivariant
$d$-cluster-tilting subcategory of $\mod\widehat\L$. Now, Corollary~\ref{bijection}(a)
implies that $F_*(\U')\subset\mod(\widehat{\L}/\phi)$ is a locally bounded
$d$-cluster-tilting subcategory, so $\widehat{\L}/\phi$ is locally $d$-re\-pre\-sen\-ta\-tion-fin\-ite. 

The second assertion in Theorem~\ref{basic construction}\eqref{basic1} is immediate from
Corollary~\ref{bijection} and the equivalence $\dml\simeq\stmod\widehat{\L}$ 
\eqref{happel equivalence}. 

\eqref{basic2}
Any additive generator of $F_*(\U)\subset\stmod(\widehat{\L}/\phi)$
is a $d$-cluster-tilting module of $\widehat{\L}/\phi$, and elements of $\ind F_*(\U)$
correspond bijectively to orbits of the $\phi_*$-action on
$\ind\U$. 
\end{proof}

We now turn to the proofs of the corollaries~\ref{trivext}, \ref{fracCY}, \ref{trivrf} and
\ref{preproj}. 
In relation to Corollary~\ref{trivext}, observe that 
$F\circ\widehat{\nu}\simeq \nu_{T_n(\L)}\circ F:\widehat{\L}\to \widehat{\L}/\widehat{\nu}^{n} = T_{n}(\L)$, where
\[F:\widehat{\L}\to \widehat{\L}/\widehat{\nu}^{n}\]
is the covering functor. Hence  
$F_*\circ\widehat{\nu}_*\simeq (\nu_{T_n(\L)})_*\circ F_*:\mod\widehat{\L}\to\mod T_{n}(\L)$. 
Again, uniqueness of the Serre functor gives us the following diagram, which is commutative up to isomorphisms of functors:
\begin{equation}\label{serrecomm}
\xymatrix@R1.5em@C=3em{
\dml\ar[r]^{\sim\;}\ar[d]^{\nu}&\stmod\widehat{\L}\ar[d]^{\widehat\nu_*[-1]}\ar[r]^{F_*\quad}&\stmod T_n(\L)\ar[d]^{(\nu_{T_n(\L)})_*[-1]}\\
\dml\ar[r]^{\sim\;}&\stmod\widehat{\L}\ar[r]_{F_*\quad}&\stmod T_n(\L).
}
\end{equation}

As we shall see, $d$-re\-pre\-sen\-ta\-tion-fin\-ite\-ness of $T_{d\ell}(\L)$ is a consequence of
the functorial isomorphisms in \eqref{serrecomm}. However, to show that the
$d$-cluster-tilting module $U$ in Corollary~\ref{trivext} is basic, we need the following
technical observation.

\begin{lma}\label{cross}
Let $I$ be a set, and $I_+$ a subset of $I$.
\begin{enumerate}
\item Assume that $f$ is a permutation of $I$, satisfying $f(I_+)\subset I_+$,
  $I=\bigcup_{i\in\Z}f^i(I_+)$ and $\emptyset=\bigcap_{i\in\Z}f^i(I_+)$.
Then for any $i\in\Z$, the set $I(f,i):=f^i(I_+)\setminus f^{i+1}(I_+)$ is a cross-section
for the $f$-orbits of $I$. \label{cross1}
\item Assume that $f$ and $g$ both satisfy the conditions in \eqref{cross1}, and that
  $fg=gf$.
Then for any $a,b>0$, the set 
$\left(\coprod_{0\le i<a}I(f,i)\right)\sqcup\left(\coprod_{-b\le i<0}I(g,i)\right)$ 
is a cross-section for the $f^ag^b$-orbits of $I$. 
\label{cross2}
\end{enumerate}
\end{lma}

\begin{proof}
\eqref{cross1}
It is clear from the assumptions that for each $x\in I$ there is a unique $m\in\Z$ such
that $f^j(x)\in I_+$ if and only if $j\ge m$. This implies that $I(f,0)$ is a
cross-section for the $f$-orbits of $I$, and therefore the same is true for
$I(f,i)=f^i(I(f,0))$ for any $i\in\Z$.

\eqref{cross2}
As $a,b>0$, we have $f^ag^b(I_+) \subset f^a(I_+)\subset I_+$.
Therefore,
\[
\bigcup_{i\in\Z}(f^ag^b)^i(I_+) = \bigcup_{n\ge0}(f^ag^b)^{-n}(I_+) =
\bigcup_{n\ge0}f^{-an}\left(g^{-bn}(I_+)\right) \supset \bigcup_{n\ge0}f^{-an}(I_+) = I
\]
and hence $I = \bigcup_{i\in\Z} (f^ag^b)^i(I_+)$.
Similarly,
\[
\bigcap_{i\in\Z}(f^ag^b)^i(I_+) = \bigcap_{n\ge0}(f^ag^b)^n(I_+) =
\bigcap_{n\ge0}f^{an}\left(g^{bn}(I_+)\right)\subset
\bigcap_{n\ge0}f^{an}(I_+) = \emptyset
\]
so $\emptyset = \bigcap_{i\in\Z}(f^ag^b)^i(I_+)$.

The above shows that the permutation $f^ag^b$ of $I$ satisfies the conditions in
\eqref{cross1}.
It follows that $I(f^ag^b,0)=I_+\setminus f^ag^b(I_+)$ is a cross-section
for the $f^ag^b$-orbits of $I$, and hence, so is 
\[g^{-b}\left(I(f^ag^b,0)\right)= g^{-b}(I_+)\setminus f^a(I_+)=
(I_+\setminus f^a(I_+))\sqcup(g^{-b}(I_+)\setminus I_+).\qedhere\]
\end{proof}

\begin{proof}[Proof of Corollary~\ref{trivext}]
From the assumptions follows that $\L$ is $\nu_d$-finite and
$\gldim\L\le d$. Thus $\U=\U_d(\L)$ is a $d$-cluster-tilting subcategory of
$\dml\simeq\stmod\widehat\L$ by Proposition~\ref{U_n}. Moreover,
by \cite[Theorem~3.1]{io13}, 
\[\nu(\U)=\U\ \mbox{ and }\ \U[d]=\U\]
hold, because $\L$ is $d$-re\-pre\-sen\-ta\-tion-fin\-ite and $\gldim\L\le d$. By \eqref{serrecomm}, 
\begin{equation}\label{nuofU}
(\widehat\nu^{d\ell})_*\simeq(\nu[1])^{d\ell}=\nu^{d\ell}[d\ell]=\nu^{d\ell}\nu^{\ell}\nu_d^{-\ell}=\nu^{\ell(d+1)}\nu_d^{-\ell}
\end{equation}
and hence, in particular, $(\widehat\nu^{d\ell})_*(\U)=\U$.
Applying Theorem~\ref{basic corollary}(a) to $\phi=\widehat\nu^{d\ell}$, it
follows that $\widehat\L/\widehat\nu^{d\ell}=T_{d\ell}(\L)$ is $d$-re\-pre\-sen\-ta\-tion-fin\-ite.
This proves the first statement. 

For the second statement, we need to calculate a cross-section for the
$(\widehat\nu^{d\ell})_*$-orbits in $\ind\U$.
Let $I=\ind\U$ and $I_+=\ind\{\nu_d^{-i}(\L)\mid i>0\}$.
Then $\nu_d^{-1}:I\to I$ and $\nu: I\to I$ satisfy the assumptions in Lemma
\ref{cross}\eqref{cross1}, so $I(\nu_d^{-1},-1)=\ind(\add\Lambda)$ and
$I(\nu,0)=\ind(\add(M/\L))$ are cross-sections for the orbits of the $\nu_d^{-1}$-action
and the $\nu$-action on $\ind\U$ respectively (cf. \cite[Lemma~4.9]{io13}).
As $(\widehat\nu^{d\ell})_*\simeq\nu^{\ell(d+1)}\nu_d^{-\ell}$ by \eqref{nuofU}, 
Lemma~\ref{cross}\eqref{cross2} implies that
\[
\left(\bigsqcup_{0\le i<\ell(d+1)} \ind\left(\add\left(\nu^i(M/\L)\right)\right) \right) \sqcup
\left(\bigsqcup_{0\le j<\ell}\ind\left(\add\left(\nu_d^j(\L)\right)\right)\right)
\]
is a cross-section for the orbits of the $\widehat{\nu}^{d\ell}$-action on $\ind\U$.
By \eqref{serrecomm}, the image of this cross-section under the equivalence
$\dml\simeq\stmod\widehat{\L}$ is isomorphic to
\[
S = \left(\bigsqcup_{0\le i<\ell(d+1)}
\ind\left(\add\left((\widehat{\nu}_*\Omega)^i(M/\L)\right)\right)\right) 
\sqcup 
\left(\bigsqcup_{0\le j<\ell}
\ind\left(\add\left((\widehat{\nu}_*\Omega^{d+1})^j(\L)\right)\right)\right) 
\]
It follows that
\[T_{d\ell}(\L) \oplus \bigoplus_{V\in S}F_*(V) \simeq
T_{d\ell}(\L)  \oplus\left(\bigoplus_{i=0}^{\ell(d+1)-1}\left((\nu_{T_{d\ell}(\L)})_*\Omega\right)^i(M/\L)\right)
\oplus\left(\bigoplus_{j=0}^{\ell-1}\left((\nu_{T_{d\ell}(\L)})_*\Omega^{d+1}\right)^j(\L)\right)\]
is a basic $d$-cluster-tilting $T_{d\ell}(\L)$-module.
\end{proof}

\begin{proof}[Proof of Corollary~\ref{fracCY}]
The twisted $(b/a)$-Calabi--Yau property means that
$\nu^a\simeq[b]\circ\phi_*:\dml\to\dml$ for some $\phi\in\Aut(\L)$, and hence
$\nu_d^a\simeq [b-da]\circ\phi_*$. By \cite[Proposition~2.1.10(c)]{himo17}, 
$b/a<\gldim\L$ and since $\gldim\L\le d$, we get $b-da<0$. 
Clearly, $\phi_*(\mathcal{D}^{\ge0}(\L))= \mathcal{D}^{\ge0}(\L)$, so it follows that
$\nu_d^a(\mathcal{D}^{\ge0}(\L)) \subset\mathcal{D}^{\ge1}(\L)$.
This implies that the algebra $\L$ is $\nu_d$-finite.

Recall that $g=\gcd(d+1,a+b)$.
Setting $p=(d+1)/g$ and $q=(a+b)/g$, we get the following identities:
\begin{align}
ap-q&= \frac{a(d+1)}{g} - \frac{a+b}{g} = \frac{ad-b}{g}, \label{dac}\\
p(a+b) &= q(d+1). \label{cab}
\end{align}
By \eqref{serrecomm0}, there are functorial isomorphisms
  $\widehat{\nu}_*^a \simeq \nu^a\circ[a]\simeq [b]\circ\phi_*\circ[a]\simeq[a+b]\circ\phi_*$
and hence
\begin{align*}
  \widehat{\nu}_*^{ap} &= \left( [a+b]\circ\phi_* \right)^p \simeq [p(a+b)]\circ \phi_*^p 
\stackrel{\eqref{cab}}{\simeq}
  [q(d+1)]\circ \phi_*^p\\
\intertext{which, in turn, gives}
  \widehat{\nu}^n_*=\widehat{\nu}_*^{\ell(ad-b)/g} &\stackrel{\eqref{dac}}{=} \widehat{\nu}_*^{\ell(ap-q)} \simeq
  \widehat{\nu}_*^{-\ell q}\circ[\ell q(d+1)]\circ\phi_*^{\ell p} \stackrel{\eqref{serrecomm0}}{\simeq}
\nu^{-\ell q}[\ell qd]\circ\phi_*^{\ell p} \simeq
  \nu_d^{-\ell q}\circ\phi_*^{\ell p}.
\end{align*}
The subcategory $\U_d(\L)\subset\dml$ is invariant under $\nu_d$ and $\phi_*$, and hence
under $\widehat{\nu}^n_*$. By Theorem~\ref{basic corollary}, this implies that $T_n(\L)$ is
$d$-re\-pre\-sen\-ta\-tion-fin\-ite.  
Since $\phi_*(\L)=\L$, it also means that 
$\ind\left(\add\{\nu_d^i(\L) \mid 0 \le i < \ell q\}\right)$ is a cross-section for the
$\widehat{\nu}^n_*$-orbits of $\U_d(\L)$. The identity \eqref{twfraccy} follows.
\end{proof}

\begin{proof}[Proof of Corollary~\ref{trivrf}]
Since $\L$ is $r$-homogeneous, $\nu\simeq\nu_d^{1-r}\phi_*$ for some $\phi\in\Aut(\L)$
\cite[Theorem~1.3]{hi11a}.
Hence, $\nu^{r-1}\simeq\nu\inv[(r-1)d]\phi_*$ and thus
\begin{equation}\label{homonu}
  \widehat{\nu}_*^{r-1}\simeq\nu^{r-1}[r-1]\simeq\nu\inv[\,(r-1)d\!+\!(r-1)\,]\phi_*\simeq\nu_{(r-1)(d+1)}\inv\phi_*
\end{equation}
on $\stmod\widehat{\L}\simeq\dml$. 
It follows that $\widehat{\nu}^{r-1}_*(\U_{(r-1)(d+1)}(\L)) = \U_{(r-1)(d+1)}(\L)$, whence
$T_{r-1}(\L)$ is $(r-1)(d+1)$-re\-pre\-sen\-ta\-tion-fin\-ite by Theorem~\ref{basic corollary}(a).
Moreover, \eqref{homonu} implies that
$\ind\L\subset\mod\widehat{\L}$ is a cross-section for the
$\widehat{\nu}_*^{r-1}$-orbits of $\U_{(r-1)(d+1)}(\L)$. Thus $T_{r-1}(\L)\oplus\L$ is basic and $(r-1)(d+1)$-cluster-tilting. 
\end{proof}

To prove Corollary~\ref{preproj}, some additional results are needed.

\begin{prop}\cite[Propositions~3.6,~4.2,~Theorem~4.5]{io13} \label{serre and serre}
Let $\TT$ be a triangulated category with a Serre functor ${}_{\TT}S$.
Let $\U$ be an $\ell$-cluster-tilting subcategory of $\TT$ satisfying $\U[\ell]=\U$. Then
the following statements hold:
\begin{enumerate}\renewcommand{\labelenumi}{(\alph{enumi})}
\item The identity ${}_{\TT}S(\U)=\U$ holds, so $\U$ is self-injective;
\item the category $\stmod\U$ is triangulated, and has a Serre functor ${}_{\stmod\U}S$;
\item there is a natural isomorphism ${}_{\stmod\U}S_{\ell+1}\simeq({}_{\TT}S_\ell)_*$ of
  autoequivalences of $\stmod\U$.
\end{enumerate}
\end{prop}

\begin{lma}\label{preproj-orbit}
Let $\L$ be a finite-dimensional $k$-algebra of global dimension at most $d$.
Taking a skeleton $\U$ of $\U_d(\L)$ and regarding $\nu_d$ as an automorphism of $\U$,
the $n$-fold $(d+1)$-preprojective algebra $\Pi^{(n)}$ of $\L$ satisfies
$\proj\Pi^{(n)}\simeq \U/\nu_d^n$.  
\end{lma}

\begin{proof}
This is an easy adaptation of the proof of \cite[Proposition~4.7]{amiot09}. For the
convenience of the reader, we give an outline of the argument here.

The Nakayama functor $\nu:\dml\to\dml$ has a quasi-inverse 
$\nu\inv=-\otimes^{\mathbb{L}}_{\L}\R\!\Hom_\L(D\L,\L)$. Hence, the functor
$\nu_d\inv=-\otimes^{\mathbb{L}}_\L\Theta$, where
$\Theta=\R\!\Hom_\L(D\L,\L)[d]$, is a quasi-inverse of $\nu_d$.
Since $\gldim\L\le d$, it follows that $\Hom_{\dml}(\nu_d^i(\L),\nu_d^j(\L))=0$ whenever
$i<j$. For $i\ge j$, we have
\begin{align*}
  &\Hom_{\dml}(\nu_d^i(\L),\nu_d^j (\L))\simeq \Hom_{\dml}(\L,\nu_d^{-(i-j)}(\L))\simeq 
  \Hom_{\dml}(\L,\Theta^{\otimes^\mathbb{L}_\L(i-j)})\simeq H^0(\Theta^{\otimes^\mathbb{L}_\L(i-j)})\\
  &\simeq H^0(\Theta)^{\otimes_\L(i-j)} \simeq
 H^0(\R\!\Hom_{\dml}(D\L,\L)[d])^{\otimes_\L(i-j)}\simeq
 \Ext_\L^d(D\L,\L)^{\otimes_\L(i-j)}\simeq \Pi_{i-j},
\end{align*}
where the fourth isomorphism comes from \cite[Lemma~4.8]{amiot09}.

It is now easy to verify that multiplication in the matrix algebra
$\Pi^{(n)}=(\Pi_{i-j+n\Z})_{1\le i,j\le n}$ corresponds to composition of morphisms in the
orbit category $\U/\nu_d^n$. The equivalence $\proj\Pi^{(n)}\simeq\U/\nu_d^n$
follows.
\end{proof}

\begin{proof}[Proof of Corollary~\ref{preproj}]
Let $\Gamma=\stend_{\L}(\Pi)$.
By \cite[Theorem~1.14, Proposition~1.12(d)]{iyama11}, the algebra $\Gamma$ is
$\nu_{d+1}$-finite and $\gldim\Gamma\le d+1$ (in the
terminology of \cite{iyama11}, $\L$ is $d$-complete and $\Gamma$ is $(d+1)$-complete). 
By \cite[Theorem~4.7]{io13}, there exists an equivalence of categories
\begin{equation}\label{L and D}
\U_{d}(\L)\simeq\widehat{\Gamma},
\end{equation}
and thus we get triangle equivalences
\[\stmod\U_{d}(\L)\simeq\stmod\widehat{\Gamma}\simeq\dm{\Gamma}.\]
On the other hand, from Lemma~\ref{preproj-orbit} follows that
\[\widehat{\Gamma}/\nu_{\L,d}^n\simeq\U/\nu_{\L,d}^n\simeq\Pi^{(n)}.\]
Applying Theorem~\ref{basic corollary}(a) to $\Gamma$, it is enough to
show that the autoequivalence $(\nu_{\L,d}^n)_*=(\nu_{\L,d})_*^n$ of
$\stmod\widehat{\Gamma}\simeq\dm{\Gamma}$ satisfies 
$(\nu_{\L,d}^n)_*(\U_{d+1}(\Gamma))=\U_{d+1}(\Gamma)$.

Proposition \ref{serre and serre}(c) gives a commutative diagram
\[\xymatrix@C=3em@R=1.5em{
  \dm{\Gamma}\ar[r]\ar[d]^{\nu_{\Gamma,d+1}}&\stmod\U\ar[d]^{(\nu_{\L,d})_*}\\
\dm{\Gamma}\ar[r]&\stmod\U
}\]
and thus we have $\U_{d+1}(\Gamma)=\nu_{\Gamma,d+1}^n(\U_{d+1}(\Gamma))=(\nu_{\L,d})_*^n(\U_{d+1}(\Gamma))$, as desired.
\end{proof}

See Section~\ref{sec:toypreproj} for an example illustrating the proof of
Corollary~\ref{preproj}.

\section{Examples and applications} \label{exandappl}

\subsection{A simple example} \label{toyex}

With the purpose of illuminating the general theory from a somewhat more concrete
viewpoint, we show here how some of the main constructions work out in a particular,
simple example.

For $n\ge3$, let $\L_n=k\mathbb{A}_n/I^{n-1}$, where $\mathbb{A}_n$ is linearly oriented
of Dynkin type $A_n$, and $I\subset k\mathbb{A}_n$ the ideal generated by all
arrows of $k\mathbb{A}_n$.

\[\L_n \;:\;\xymatrix{
1\ar[r]\ar@{.}@/^1.2pc/[rrrr] &2\ar[r]&\cdots\ar[r] &n-1\ar[r] &n
}\]
The algebra $\L_n$ is $2$-re\-pre\-sen\-ta\-tion-fin\-ite of global dimension $2$
\cite[Theorem~3.12]{io13}, and the subcategory $\add(\L_n\oplus D\L_n)$ of $\mod\L_n$ is
the unique $2$-cluster-tilting subcategory. 
Moreover, by Proposition~\ref{U_n}, the subcategory 
\begin{equation} \label{toyU_2first}
\U_2(\L_n)=\add\{\nu_2^i(\L_n)\mid i\in\Z\} = 
\add\{(\L_n\oplus D\L_n)[2\ell]\mid\ell\in\Z\} \subset \dm{\L_n}
\end{equation}
is $2$-cluster-tilting.

Let $\mathbb{D}_n$ be a Dynkin quiver of type $D_n$. 
The Dynkin diagram $D_n$ has an automorphism of order $2$ (unique for $n>4$), which
induces an automorphism $\s$ of $\dm{k\mathbb{D}_n}$. Let $X\in\dm{k\mathbb{D}_n}$ be an
indecomposable object such that $\s(X)\not\simeq X$. 
Then $T=\bigoplus_{i=0}^{n-1}(\nu_1\s)^i(X)$ is a tilting complex, and 
$\End_{\dm{k\mathbb{D}_n}}(T)\simeq \L_n$. Thus the algebra $\L_n$ is derived equivalent to
$k\mathbb{D}_n$, via a triangle equivalence sending $P_j\in\dm{\L_n}$ to 
$(\nu_1\s)^{n-j}(X)\in\dm{k\mathbb{D}_n}$ (here, $P_j$ denotes the projective cover of the
simple $\L_n$-module supported at the vertex $j$).

\subsubsection{Orbit algebra} \label{sec:toyexample}

The automorphism $\s$ of $\dm{\L_n}\simeq\dm{k\mathbb{D}_n}$ satisfies 
(see\ \cite[Theorem~4.1]{my01})
\begin{align}
\label{dmshift}
[1]&\simeq \nu_1^{1-n}\s^n\,, \quad\mbox{and}  \\
  P_j&\simeq\nu_1\s(P_{j-1}) \quad\mbox{for } 1<j\le n
\end{align}
Hence we have
\begin{equation} \label{toynu2}
\nu_2=\nu_1\circ[-1]\simeq \nu_1\nu_1^{n-1}\s^n = (\nu_1\s)^n \,  
\end{equation}
which, together with \eqref{toyU_2first}, implies that
\begin{equation} \label{toyU_2}
\U_2(\L_n)=\add\{(\nu_1\s)^i(P_1)\mid i\in\Z\} \,.
\end{equation}

The repetitive category $\widehat{\L}_n = \proj^{\Z}T(\L_n)$ of $\L_n$ 
is given by the following infinite quiver with relations:
{\scriptsize
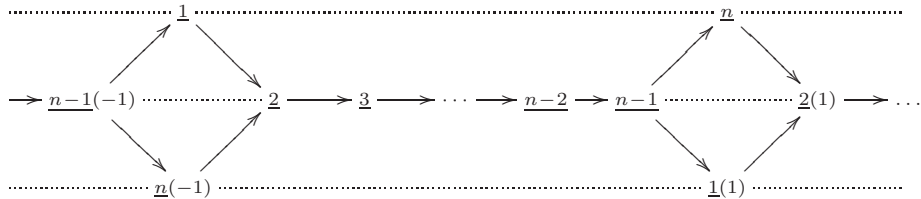
\[
\xymatrix@C=1.4em@R=1.3em@!=.8cm
{
\ar@{.}[rr]&&\underline{1}\ar[dr]\ar@{.}[rrrrrr]&&&&&&\underline{n}\ar[dr]\ar@{.}[rr]&&\\
\ar[r]&\underline{n\!-\!1}(-1)\ar[ur]\ar[dr]\ar@{.}[rr]&&\underline{2}\ar[r]&\underline{3}\ar[r]&\cdots\ar[r]&\underline{n\!-\!2}\ar[r]&\underline{n\!-\!1}\ar[ur]\ar[dr]\ar@{.}[rr]&&\underline{2}(1)\ar[r]&\ldots\\
\ar@{.}[rr]&&\underline{n}(-1)\ar[ur]\ar@{.}[rrrrrr]&&&&&&\underline{1}(1)\ar[ur]\ar@{.}[rr]&&
}\]}
Here, $\underline{j}$ denotes the projective $\L_n$-module $P_j$ concentrated in
degree $0$, and $(i)$ degree shift by $i$.  
The Nakayama automorphism $\widehat{\nu}\simeq(1)$ of $\widehat{\L}_n$ induces an
automorphism $\widehat{\nu}_*$ of $\stmod\widehat{\L}_n\simeq\dm{\L_n}$, which satisfies
$$\widehat{\nu}_*\simeq \nu\circ[1]=
\nu_1\circ[2]\stackrel{\mbox{\tiny\eqref{dmshift}}}{\simeq} \nu_1^{3-2n}\,.$$
Since $\s^2=\I$ and $(\nu_1\s)(\U_2(\L_n))=\U_2(\L_n)$, we get 
$$\widehat{\nu}_*^{2\ell}(\U_2(\L_n))=
\nu_1^{(6-4n)\ell}(\U_2(\L_n))=(\nu_1\s)^{(6-4n)\ell}(\U_2(\L_n)) = \U_2(\L_n)$$ 
for all $\ell\ge1$. By Theorem~\ref{basic corollary}, this implies that the orbit algebra
$\widehat{\L}_n/\widehat{\nu}^{2\ell}$ is $2$-re\-pre\-sen\-ta\-tion-fin\-ite (see Figure \ref{T2Lm}
for the case $\ell=1$).

\begin{figure}[h]
\[
  \xymatrix@R=.8cm@C=.8cm{
&\bullet\ar[dl]\ar[dr]\ar@{.}[dd]&\bullet\ar[l]&\cdots\ar[l]&\bullet\ar[l]&\bullet\ar[l]\\
\bullet\ar[dr]&&\bullet\ar[dl]\ar@{.}@/_2pc/[rrrr]\ar@{.}@/^2pc/[rrrr]&&\bullet\ar[ur]\ar@{.}@/_2pc/[llll]\ar@{.}@/^2pc/[llll]&&\bullet\ar[ul]\\
&\bullet\ar[r]&\bullet\ar[r]&\cdots\ar[r]&\bullet\ar[r]&\bullet\ar[ul]\ar[ur]\ar@{.}[uu]
}
\]
\caption{The algebra $\widehat{\L}_n/\widehat{\nu}^2$.}\label{T2Lm}
\end{figure}
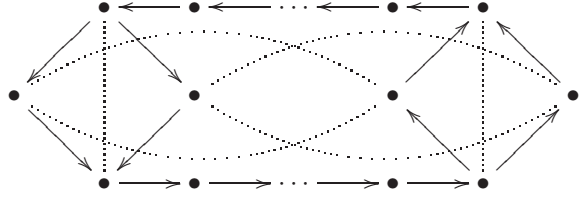

A $2$-cluster-tilting module of $\widehat{\L}_n/\widehat{\nu}^{2\ell}$ can be constructed
as follows: 
The preimage $\U$ of $\U_2(\L_n)$ under the natural functor 
$\mod\widehat{\L}_n\to \stmod\widehat{\L}_n\simeq\dm{\L_n}$ is a $2$-cluster-tilting subcategory of
$\mod\widehat{\L}_n$, which is equivariant under the induced automorphism
$\widehat{\nu}^{2\ell}_*:\mod\widehat{\L}_n\to\mod\widehat{\L}_n$. From
Theorem~\ref{stronger theorem} in Section~\ref{galois coverings}, it follows that the
subcategory $F_*(\U)\subset \mod(\widehat{\L}_n/\widehat{\nu}^{2\ell})$ is
$2$-cluster-tilting. Now any additive generator of $F_*(\U)$ is a $2$-cluster-tilting
$\widehat{\L}_n/\widehat{\nu}^{2\ell}$-module.

The quiver with relations of the algebra $T_2(\L_n)$ is given in Figure~\ref{T2Lm}, and
the module category of $\widehat{\L}_4$ in Figure~\ref{modwL4}.

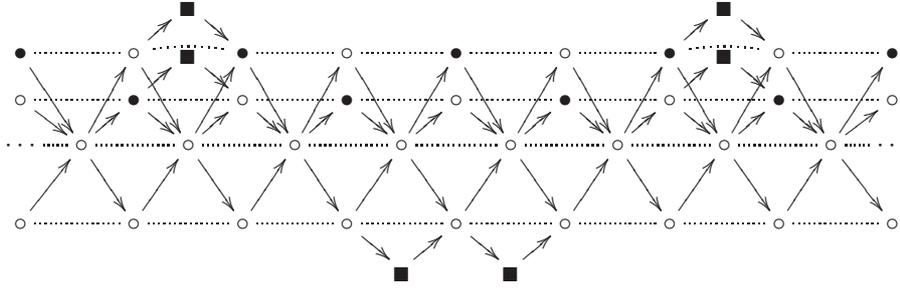
\begin{figure}[h]
 \[
  \xymatrix@R=.23cm@C=.3cm{
    &&&{}_\blacksquare\ar[dr]&&&&&&&&&&{}_\blacksquare\ar[dr]&&&& \\
   \bullet\ar[ddr]\ar@{.}[rr]&&\circ\ar[ddr]\ar[ur]\ar@{.}@/^.2pc/[rr]&{}_\blacksquare\ar[dr]&\bullet\ar[ddr]\ar@{.}[rr]&&\circ\ar[ddr]\ar@{.}[rr]&&\bullet\ar[ddr]\ar@{.}[rr]&&\circ\ar[ddr]\ar@{.}[rr]&&\bullet\ar[ddr]\ar[ur]\ar@{.}@/^.2pc/[rr]&{}_\blacksquare\ar[dr]&\circ\ar[ddr]\ar@{.}[rr]&&\bullet\ar@{.}[rr]\ar[ddr] &&\circ\ar@{.}[rr]\ar[ddr] &&\bullet\\
    \circ\ar[dr]\ar@{.}[rr]&&\bullet\ar[dr]\ar[ur]\ar@{.}[rr]&&\circ\ar[dr]\ar@{.}[rr]&&\bullet\ar[dr]\ar@{.}[rr]&&\circ\ar[dr]\ar@{.}[rr]&&\bullet\ar[dr]\ar@{.}[rr]&&\circ\ar[dr]\ar[ur]\ar@{.}[rr]&&\bullet\ar[dr]\ar@{.}[rr]&&\circ\ar@{.}[rr]\ar[dr] &&\bullet\ar@{.}[rr]\ar[dr] &&\circ\\
    \cdots\ar@{.}[r]&\circ\ar[ddr]\ar@{.}[rr]\ar[uur]\ar[ur]&&\circ\ar[ddr]\ar@{.}[rr]\ar[uur]\ar[ur]&&\circ\ar[ddr]\ar@{.}[rr]\ar[uur]\ar[ur]&&\circ\ar[ddr]\ar@{.}[rr]\ar[uur]\ar[ur]&&\circ\ar[ddr]\ar@{.}[rr]\ar[uur]\ar[ur]&&\circ\ar[ddr]\ar@{.}[rr]\ar[uur]\ar[ur]&&\circ\ar[ddr]\ar@{.}[rr]\ar[uur]\ar[ur]&&\circ\ar[ddr]\ar[uur]\ar[ur]\ar@{.}[rr]&&\circ\ar[uur]\ar[ur]\ar[ddr]\ar@{.}[rr]&&\circ\ar[uur]\ar[ur]\ar[ddr]\ar@{.}[r]&\cdots\\
    \\
    \circ\ar[uur]\ar@{.}[rr]&&\circ\ar[uur]\ar@{.}[rr]&&\circ\ar[uur]\ar@{.}[rr]&&\circ\ar[dr]\ar[uur]\ar@{.}[rr]&&\circ\ar[dr]\ar[uur]\ar@{.}[rr]&&\circ\ar[uur]\ar@{.}[rr]&&\circ\ar[uur]\ar@{.}[rr]&&\circ\ar[uur]\ar@{.}[rr]&&\circ\ar@{.}[rr]\ar[ruu]\ar[dr]&&\circ\ar@{.}[rr]\ar[ruu]\ar[dr]&&\circ\\
&&&&&&&{}_\blacksquare\ar[ur]&&{}_\blacksquare\ar[ur]&&&&&&&&{}_\blacksquare\ar[ur]&&{}_\blacksquare\ar[ur]
  }
  \]
\caption{The module category of $\widehat{\L}_4$, with the $2$-cluster-tilting subcategory
  given by $\L_4$ indicated in black (boxes represent projective-injective modules). The
  module category of $\widehat{\L}_4/\widehat{\nu}_2$ is equivalent to
  the orbit category $(\mod\widehat{\L}_4)/\tau^{10}$, where $\tau=\nu_1$ is the
  Auslander--Reiten translation on $\mod\widehat{\L}_4$. }\label{modwL4}
\end{figure}

\subsubsection{Preprojective algebra} \label{sec:toypreproj}

Again let $\L=\L_n=k\mathbb{A}_n/I^{n-1}$.
Here, we shall construct the $3$-preprojective algebra $\Pi$ of $\L$ and show that it
is $3$-re\-pre\-sen\-ta\-tion-fin\-ite, along the lines of the proof of Corollary~\ref{preproj}.

Setting $\Gamma=\stend_{\L}(\Pi)$, we have 
$$\Gamma=\stend_{\L}(\Pi) \simeq \stend_{\L}(D\L)
\simeq k\mathbb{A}_{n-2} \,.$$ 
From Equation~\eqref{toyU_2}, we know that
$\U_2(\L)=\add\{(\nu_1\s)^i(P_1)\mid i\in\Z\}\subset\dm{\L}$, and it is now
easy to verify that 
$\U_2(\L)\simeq k\mathbb{A}_{\infty}^{\infty}/I^{n-1}\simeq\widehat{\Gamma}$,
where
$\mathbb{A}_{\infty}^{\infty}$ is linearly oriented of type $A^{\infty}_{\infty}$:
\begin{equation} \label{quivertilde}
  \mathbb{A}_{\infty}^{\infty} \;:\;
  \cdots\stackrel{a_{-2}}{\longrightarrow}-1\stackrel{a_{-1}}{\longrightarrow}0\stackrel{a_0}{\longrightarrow}1\stackrel{a_1}{\longrightarrow}2\stackrel{a_2}{\longrightarrow}\cdots
\end{equation}
Let $\psi$ be the automorphism of $\widehat{\Gamma}$ given by shift one step to the left in
$\mathbb{A}_{\infty}^{\infty}$: $\psi(i)=i-1$ and $\psi(a_i)=a_{i-1}$ for all
$i\in\mathbb{Z}$; then $\nu_1\s\simeq\psi$ as automorphisms of
$\U_2(\L)\simeq\widehat{\Gamma}$. 
Since $\Pi\simeq\U_2(\L)/\nu_{\L,2}$ and $\nu_{\L,2}\simeq(\nu_1\s)^n\simeq\psi^n$ by
\eqref{toynu2}, we now get that 
$$\Pi\simeq \mathbb{A}_\infty^\infty/\psi^n\simeq\tilde{\mathbb{A}}_n/I^{n-1},$$
where $\tilde{\mathbb{A}}_n$ is a cyclically oriented quiver of extended Dynkin type
$\widetilde{A}_{n-1}$.
In other words, $\Pi$ is the self-injective Nakayama algebra with $n$ isomorphism classes
of simple modules and Loewy length $n-1$.

The induced automorphism $\psi_*$ of $\stmod\widehat{\Gamma}\simeq\dm{\Gamma}$ is isomorphic to $\nu_{\Gamma,1}$ -- the Auslander--Reiten translation on $\dm{\Gamma}$.
The category $\dm{\Gamma}$ is fractionally $\frac{n-3}{n-1}$-Calabi--Yau, meaning that
$[n-3]\simeq\nu_\Gamma^{n-1}$ as autoequivalences of $\dm{\Gamma}$, and hence
$[-2]\simeq\nu_{\Gamma,1}^{n-1}$.
Consequently,
\begin{equation} \label{nu2nu3}
\nu_{\Gamma,3}\simeq\nu_{\Gamma,1}\circ[-2]\simeq\nu_{\Gamma,1}^n\simeq\psi_*^n \simeq
(\nu_{\L,2})_*
\end{equation}
on $\dm{\Gamma}\simeq\stmod\widehat{\Gamma}\simeq\stmod\U_2(\L)$.
Since $\gldim\Gamma=1\le3$, the subcategory $\U_3(\Gamma)$ of $\dm{\Gamma}$ is
$3$-cluster-tilting, and $(\nu_{\L,2})_*(\U_3(\Gamma))=\nu_{\Gamma,3}(\U_3(\Gamma)) = \U_3(\Gamma)$ by \eqref{nu2nu3}. 
It follows from Theorem~\ref{basic corollary} that $\Pi\simeq \widehat{\Gamma}/\psi^n$ is
$3$-re\-pre\-sen\-ta\-tion-fin\-ite.

Let $\U'\subset\mod\widehat{\Gamma}$ be the preimage of $\U_3(\Gamma)$ under the natural
functor $\mod\widehat{\Gamma}\to\stmod\widehat{\Gamma}\simeq\dm{\Gamma}$. 
Denoting by $F:\widehat{\Gamma}\to\widehat{\Gamma}/\nu_{\Gamma,3}\simeq\Pi$ the covering
functor, the subcategory $F_*(\U')$ of $\mod\Pi$ is $3$-cluster-tilting. 
See Figure~\ref{modPD5} for a picture in the case $n=5$.

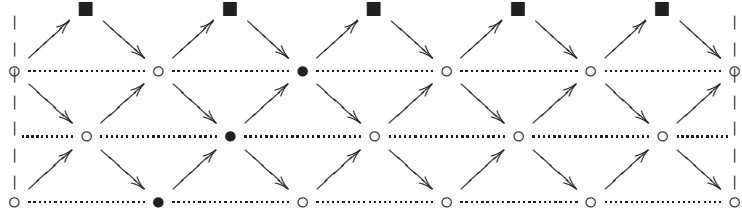
\begin{figure}[h]
  \[
  \xymatrix@R=.5cm@C=.53cm{
    \ar@{--}[ddd]& {}_\blacksquare\ar[dr] && {}_\blacksquare\ar[dr] && {}_\blacksquare\ar[dr] &&
        {}_\blacksquare\ar[dr] && {}_\blacksquare\ar[dr] & \ar@{--}[ddd] \\
    \circ\ar[ur]\ar[dr]\ar@{.}[rr] && \circ\ar[ur]\ar[dr]\ar@{.}[rr] &&
    \bullet\ar[ur]\ar[dr]\ar@{.}[rr] && \circ\ar[ur]\ar[dr]\ar@{.}[rr] && \circ\ar[ur]\ar[dr]\ar@{.}[rr] && \circ \\
    \ar@{.}[r]& \circ\ar[ur]\ar[dr]\ar@{.}[rr] && \bullet\ar[ur]\ar[dr]\ar@{.}[rr] &&
    \circ\ar[ur]\ar[dr]\ar@{.}[rr] && \circ\ar[ur]\ar[dr]\ar@{.}[rr] && \circ\ar[ur]\ar[dr]\ar@{.}[r] & \\
    \circ\ar[ur]\ar@{.}[rr] && \bullet\ar[ur]\ar@{.}[rr] && \circ\ar[ur]\ar@{.}[rr] && \circ\ar[ur]\ar@{.}[rr] && \circ\ar[ur]\ar@{.}[rr] && \circ
  }
  \]
  \caption{The module category of $\Pi(\L_5)$, with a $3$-cluster-tilting subcategory
  indicated in black. The vertical dashed
  lines at the left and right ends of the picture are identified with each
  other.}\label{modPD5}
\end{figure}

\subsection{Trivial extensions of homogeneous $d$-re\-pre\-sen\-ta\-tion-fin\-ite algebras}

In the case of $\L$ being an $r$-homogeneous $d$-re\-pre\-sen\-ta\-tion-fin\-ite algebra
of global dimension $d$, either Corollary~\ref{trivext} or \ref{fracCY} can be employed to
find a basic $d$-cluster-tilting $T_{d\ell}(\L)$-module.
Specializing Corollary~\ref{fracCY} to this situation, we get the following
simplified description.

\begin{prop} \label{homogeneoustrivext}
  Let $\L$ be an $r$-homogeneous $d$-re\-pre\-sen\-ta\-tion-fin\-ite algebra of global dimension
  $d$, and $m$ a positive integer.
  Then 
  \begin{equation} \label{homogeneousCT}
  V = T_{dm}(\L)\oplus
  \bigoplus_{i=0}^{m(dr-d+r)-1}\left((\nu_{T_{dm}(\L)})_*\Omega^{d+1}\right)^i(\L)\\ 
  \end{equation}
  is a basic $d$-cluster-tilting $T_{dm}(\L)$-module.
\end{prop}

\begin{proof}
By \cite[Theorem~1.3]{hi11a}, the algebra $\L$ is twisted
$\frac{d(r-1)}{r}$\,-\,Calabi--Yau. 
We apply Corollary~\ref{fracCY} to the case $(a,b,\ell) = (r,d(r-1),m g)$, where
$g=\gcd(d+1,a+b)$. 
Then $n=\ell(ad-b)/g$ is $dm$, and $\ell(a+b)/g-1$ is $m(dr-d+r)-1$. Thus the result
follows.
\end{proof} 
It is easy to check that the $d$-cluster-tilting module $V$ in
Proposition~\ref{homogeneoustrivext} coincides with the module $U$ given in
Corollary~\ref{trivext}.
Indeed, $\add V\subset\add\{S_d^i(\L)\mid i\in\Z\}$ holds as a subcategory of 
$\stmod T_{dm}(\L)$ for $S_d=(\nu_{T_{dm}(\L)})_*\Omega^{d+1}$. Since $\L\in\add U$
and any $d$-cluster-tilting subcategory is closed under $S_d$, it follows that $\add
V\subset \add U$. 
But since $U$ and $V$ are basic $d$-cluster-tilting, we get $U\simeq V$.

The following result by Herschend and Iyama, presented here in a slightly generalized
form, gives a rich source of homogeneous $d$-re\-pre\-sen\-ta\-tion-fin\-ite algebras. 

\begin{prop} \label{homogeneoustp}
  For $i=1,\ldots,n$, let $\L_i$ be an $r$-ho\-mo\-gene\-ous 
  $d_i$-re\-pre\-sen\-ta\-tion-fin\-ite algebra of global dimension $d_i$ such that $\L_i/\rad\L_i$
  is a separable $k$-algebra.
  Then $\bigotimes_{i=1}^n\L_i$ is an $r$-homogeneous
  $d$-re\-pre\-sen\-ta\-tion-fin\-ite algebra of global dimension $d$, where $d=\sum_{i=1}^nd_i$, and
  $\bigoplus_{j=0}^{r-1}\left(\tau_{d_1}^j\L_1\otimes\cdots\otimes\tau_{d_n}^j\L_n\right)$
  is a $d$-cluster-tilting module of $\bigotimes_{i=1}^n\L_i$.
\end{prop}

\begin{proof}
This result is proved in \cite[Corollary~1.5]{hi11a} under the additional condition that
$k$ is a perfect field, which is stronger than the algebras $\L_i/\rad\L_i$ being
separable.
In the original proof, the assumption that $k$ is perfect is used to ensure that the
algebra $\L=\bigotimes_{i=1}^n\L_i$ has finite global dimension or, equivalently, that the
projective dimension of the $\L$-module $\L/\rad\L$ is finite. Our assumption that
$\L_i/\rad\L_i$ is a separable $k$-algebra is enough for this, since, by 
\cite[Proposition~7.7]{cr90}, $\L/\rad\L\simeq\bigotimes_{i=1}^n(\L_i/\rad\L_i)$ holds, 
and the latter module clearly has finite projective dimension.
\end{proof}

Let $Q_1,\ldots,Q_n$ be quivers of Dynkin type $A_m$ with symmetric
orientation, in the sense that the orientation of the arrows of each quiver $Q_i$ is
invariant under the canonical automorphism of order $2$ of the Dynkin diagram
$A_m$. In particular, this implies that $m$ is odd. The path algebras $\L_i=kQ_i$
are $r$-homogeneous and $1$-re\-pre\-sen\-ta\-tion-fin\-ite for $r=(m+1)/2$. 
By Proposition~\ref{homogeneoustp}, the algebra 
$\L = \L_1\otimes\ldots\otimes\L_n$ is $r$-homogeneous and
$n$-re\-pre\-sen\-ta\-tion-fin\-ite, and $\gldim\L=n$.  
By Proposition~\ref{homogeneoustrivext}, the module $V$ defined by
Equation~\eqref{homogeneousCT} is a basic $n$-cluster-tilting $T_{n\ell}(\L)$-module.
Moreover, noting that 
$$((\nu_{T_{n\ell}(\L)})_*\Omega^{n+1})^i(\L) \simeq F_*\!\left(\nu_n^i(\L)\right)\simeq
F_*\!\left(\nu_1^i(\L_1)\otimes\cdots\otimes\nu_1^i(\L_n)\right)$$
for the covering functor $F:\widehat{\L}\to T_{n\ell}(\L)$, we get an alternative presentation of
the module $V$:
\begin{equation} \label{homogeneousCTtp}
  V  
  \simeq
 T_{n\ell}(\L)\oplus
  \bigoplus_{i=0}^{\ell(rn-n+r)-1}F_*\!\left(\nu_1^i(\L_1)\otimes\cdots\otimes\nu_1^i(\L_n)\right) \,.
\end{equation}

\begin{ex} \label{tpexample}
Consider the quiver $\mathbb{A}_3$ with arrows pointing towards the middle point.
The path algebra $k\mathbb{A}_3$ is a $2$-homogeneous $1$-re\-pre\-sen\-ta\-tion-fin\-ite algebra of
global dimension $1$, so the algebra $\L=(k\mathbb{A}_3)^{\otimes2}$ is $2$-homogeneous and
$2$-re\-pre\-sen\-ta\-tion-fin\-ite of global dimension $2$. 

By Corollary~\ref{trivrf}, the trivial extension algebra $T(\L)$ of
$\L$ is $3$-rep\-re\-sen\-ta\-tion-finite.

In addition, Proposition~\ref{homogeneoustrivext} implies that the $2$-fold trivial
extension algebra $T_2(\L)$ of $\L$ is $2$-re\-pre\-sen\-ta\-tion-fin\-ite, and
\begin{align*}
  V &= T_2(\L) \oplus \bigoplus_{i=0}^3((\nu_{T_2(\L)})_*\Omega^3)^i(\L) 
  \stackrel{\eqref{homogeneousCTtp}}{\simeq} 
  T_2(\L) \oplus   \bigoplus_{i=0}^3F_*\!\left(\nu_1^i(k\mathbb{A}_3)^{\otimes2}\right) \\
  &\simeq T_2(\L) \oplus F_*\!\left(k\mathbb{A}_3^{\otimes2}\right) \oplus
  F_*\!\left(D(k\mathbb{A}_3)^{\otimes2}\right) \oplus 
  F_*\!\left(k\mathbb{A}_3[-1]^{\otimes2}\right) \oplus 
  F_*\!\left(D(k\mathbb{A}_3)[-1]^{\otimes2}\right) \\
  &\simeq T_2(\L)\oplus\L\oplus D\L\oplus\Omega^2\L\oplus\Omega^2(D\L)
\end{align*}
is a basic $2$-cluster-tilting $T_2(\L)$-module. 

\begin{figure}[h]
  \begin{minipage}{3.5cm}
    \[
    \xymatrix{
      \bullet\ar[r]\ar[d]\ar@{.}[dr] &\bullet\ar[d]&\bullet\ar[d]\ar[l] \\
      \bullet\ar[r] & \bullet\ar@{.}[ur]\ar@{.}[dr] & \bullet\ar[l] \\
      \bullet\ar[u]\ar[r]\ar@{.}[ur] & \bullet\ar[u] & \bullet\ar[u]\ar[l]
    }
    \]
\end{minipage}
\begin{minipage}{3.5cm}
    \[
    \xymatrix{
      \bullet\ar[r]\ar[d] &\bullet\ar[d]&\bullet\ar[d]\ar[l] \\
      \bullet\ar[r] & \bullet\ar[ul]\ar[ur]\ar[dl]\ar[dr]\ar@{.}[dr] & \bullet\ar[l] \\
      \bullet\ar[u]\ar[r]\ar@{.}[ur] & \bullet\ar[u] & \bullet\ar[u]\ar[l]
    }
    \]
\end{minipage}
\begin{minipage}{7.5cm}
  \[
  \xymatrix{
    \bullet\ar[r]\ar[d]\forget{\ar@{.}[dr]} &\bullet\ar[d]&\bullet\ar[d]\ar[l] &
    \bullet\ar[r]\ar[d]\forget{\ar@{.}[dr]} &\bullet\ar[d]&\bullet\ar[d]\ar[l] \\
    \bullet\ar[r] & \bullet\forget{\ar@{.}[ur]\ar@{.}[dr]}\ar[urr]\ar[drr]\ar@/^2pc/[urrrr]\ar@/_2pc/[drrrr] & \bullet\ar[l] &
    \bullet\ar[r] & \bullet\forget{\ar@{.}[ur]\ar@{.}[dr]}\ar[ull]\ar[dll]\ar@/_2pc/[ullll]\ar@/^2pc/[dllll] & \bullet\ar[l] \\
    \bullet\ar[u]\ar[r]\forget{\ar@{.}[ur]} & \bullet\ar[u] & \bullet\ar[u]\ar[l] &
    \bullet\ar[u]\ar[r]\forget{\ar@{.}[ur]} & \bullet\ar[u] & \bullet\ar[u]\ar[l]
  }
  \]
\end{minipage}
\caption{The algebra $\L=k\mathbb{A}_3\otimes k\mathbb{A}_3$, and the quivers of
    $T(\L)$ and $T_2(\L)$.}
\label{T(A32)}
\end{figure}
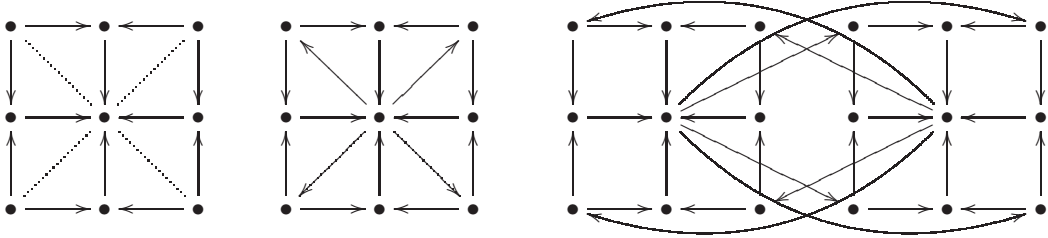
\end{ex}

\subsection{$3$-preprojective algebras}\label{section: higher preprojective}
Corollary~\ref{preproj} provides a rich source of $(d+1)$-repre\-sen\-ta\-tion-finite
self-injective algebras, constructed from $d$-re\-pre\-sen\-ta\-tion-fin\-ite
algebras of global dimension $d$. Many instances of ($1$-fold) higher preprojective
algebras have already been described in the literature, see, for example,
\cite{hi11b,io11,io13,jasso15a}.
In the case $d=2$, Keller \cite{keller11} has given a description of the
$3$-preprojective algebras in terms of quivers with potential (see \cite{DWZ}), which we
 shall recall below. 
This context is well suited for computing explicit examples.
 
Let $k$ be an algebraically closed field, and $\L$ a finite-dimensional $k$-algebra of
global dimension at most $2$. 
We denote by $\widehat{\Pi}=\widehat{\Pi}(\L)$ the complete $3$-preprojective algebra of
$\L$, that is, $\widehat{\Pi}=\prod_{i\ge0}\Pi_i$, where $\Pi=\bigoplus_{i\ge0}\Pi_i$ is the $3$-preprojective algebra of $\Lambda$.
Take a presentation $\L=\widehat{kQ}/\overline{\langle r_1,\ldots,r_\ell\rangle}$ of
$\L$ by a quiver $Q$ with a minimal set of relations $r_1,\ldots,r_\ell$, where
$\widehat{kQ}$ is the complete path algebra of $Q$, and 
$\overline{\<r_1,\ldots,r_\ell\>}$ the closure of the ideal $\<r_1,\ldots,r_\ell\>$
with respect to the $(\rad\widehat{kQ})$-adic topology.
Denote by $s(r_i)$ and $t(r_i)$ the initial and terminal vertices of a relation $r_i$,
respectively. A quiver with potential $(Q_\L,W_\L)$ is defined as follows:
\begin{eqnarray*}
Q_\L=Q\sqcup\{a_i:t(r_i)\to s(r_i)\mid 1\le i\le \ell\},\ \ \ W_\L=\sum_{i=1}^\ell a_ir_i.
\end{eqnarray*}
In this case, there is an isomorphism
\begin{equation}\label{Jacobi is preprojective}
\widehat{{\mathcal P}}(Q_\L,W_\L)\simeq\widehat{\Pi}
\end{equation}
of $k$-algebras \cite[Theorem 6.10]{keller11}, where $\widehat{{\mathcal P}}(Q_\L,W_\L)$
is the complete Jacobi algebra \cite{DWZ}. 
Note that, by Lemma~\ref{preproj-orbit}, the algebra $\L$ is $\nu_2$-finite if and
only if $\Pi$ (or, equivalently, $\widehat{\Pi}$) is finite dimensional. 
In this case, we have $\Pi=\widehat{\Pi}\simeq\widehat{{\mathcal P}}(Q_\L,W_\L)$. 
Therefore $\L$ is $2$-re\-pre\-sen\-ta\-tion-fin\-ite if and only if 
$\widehat{\mathcal{P}}(Q_\L,W_\L)$ is a finite-dimensional self-injective algebra
\cite{hi11b}. 

In \cite{hi11b}, it was shown that a basic $k$-algebra $\L$ is $2$-re\-pre\-sen\-ta\-tion-fin\-ite
of global dimension $2$ if and only if it is isomorphic to 
$\widehat{\mathcal{P}}(Q,W)/\overline{\langle C\rangle}$ for a self-injective quiver with potential
$(Q,W)$ and a cut $C$.
A large number of self-injective quivers with potential and corresponding
$2$-re\-pre\-sen\-ta\-tion-fin\-ite algebras were also given.
In this situation, Corollary~\ref{preproj} tells us that the $m$-fold $3$-preprojective
algebras $\Pi^{(m)}$ of $\L$ are self-injective and $3$-re\-pre\-sen\-ta\-tion-fin\-ite for all
$m\ge1$.
Here, we shall only briefly mention one example.

\begin{ex}
Let $\L=(k\mathbb{A}_3)^{\otimes2}$. The ($1$-fold) $3$-preprojective algebra
$\Pi$ of $\L$ is isomorphic to the Jacobi algebra of the quiver with potential
$(Q_\L,W_\L)$ described in Figure~\ref{QP}.
\begin{figure}[h]
\[ Q_\L \;=\;
\vcentered{\xymatrix{
  \bullet\ar[r]\ar[d] &\bullet\ar[d]&\bullet\ar[d]\ar[l] \\
  \bullet\ar[r] & \bullet\ar[ur]\ar[dr]\ar[ul]\ar[dl] & \bullet\ar[l] \\
  \bullet\ar[u]\ar[r]\ar@{.}[ur] & \bullet\ar[u] & \bullet\ar[u]\ar[l]
}}
\:,
\qquad
W_\L=\sum\smatr{\rm clockwise \\ \rm triangles} - 
\sum\smatr{\rm counterclockwise \\ \rm triangles}
\] 
\caption{Quiver with potential $(Q_\L,W_\L)$ associated to $\L=(k\mathbb{A}_3)^{\otimes2}$.}
\label{QP}
\end{figure}
The Nakayama automorphism of $\Pi$ has order two, and thus this algebra is not
symmetric. In particular, it is not isomorphic to the trivial extension algebra
$T(\L)$ which, by Corollary~\ref{trivrf}, is also $3$-re\-pre\-sen\-ta\-tion-fin\-ite.
\end{ex}

\subsection{A family of examples of wild representation type}

Let $m\ge2$ and $\L_m=k\mathbb{A}_m^{\otimes2}$, where $\mathbb{A}_m$ is a quiver of
Dynkin type $A_m$ of arbitrary orientation. As the hereditary algebra $k\mathbb{A}_m$ is
fractionally $\frac{m-1}{m+1}$-Calabi--Yau, it follows that $\L_m$ is
$\frac{2(m-1)}{m+1}$-Calabi--Yau. Applying Corollary~\ref{fracCY} we get that the algebra
$T_n(\L_m)$, where $\displaystyle n=\ell\frac{(d+1)(m+1)-(3m-1)}{\gcd(d+1,3m-1)}$ with
$\ell\ge1$, is $d$-re\-pre\-sen\-ta\-tion-fin\-ite for any $d\ge2$.
In particular, setting $d=2$ and $d=3m-2$, respectively, gives the following results.
\begin{prop} Let $m\ge2$, $\ell\ge1$ and $\L_m=k\mathbb{A}_3^{\otimes2}$ as above. 
  \begin{enumerate}
  \item The algebra $T_{4\ell}(\L_m)$ is $2$-re\-pre\-sen\-ta\-tion-fin\-ite.
  \item The algebra $T_{m\ell}(\L_m)$ is $(3m-2)$-re\-pre\-sen\-ta\-tion-fin\-ite.
  \end{enumerate}
\end{prop}
Observe that the algebra $\L_m$ is of wild representation type whenever $m\ge4$ 
(see \cite[Theorem~2.5]{leszczynski94}), and hence
that the same holds for $T_n(\L_m)$ for all $n\ge1$.
See Figure~\ref{wild example} for an illustration of the case of $\L_4$ with linear
orientation. 
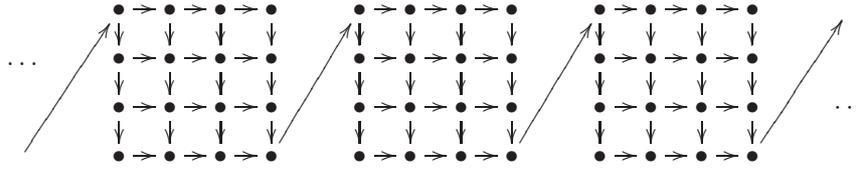
\begin{figure}[h]
    \[
    \xymatrix@R=.8em@C=.8em{
      &&
      \bullet\ar[r]\ar[d] & \bullet\ar[d]\ar[r] & \bullet\ar[d]\ar[r] & \bullet\ar[d] &&
      \bullet\ar[r]\ar[d] & \bullet\ar[d]\ar[r] & \bullet\ar[d]\ar[r] & \bullet\ar[d] &&
      \bullet\ar[r]\ar[d] & \bullet\ar[d]\ar[r] & \bullet\ar[d]\ar[r] & \bullet\ar[d] &&\\ 
      \ldots      &&
      \bullet\ar[r]\ar[d] & \bullet\ar[r]\ar[d] & \bullet\ar[r]\ar[d] & \bullet\ar[d] && 
      \bullet\ar[r]\ar[d] & \bullet\ar[r]\ar[d] & \bullet\ar[r]\ar[d] & \bullet\ar[d] && 
      \bullet\ar[r]\ar[d] & \bullet\ar[r]\ar[d] & \bullet\ar[r]\ar[d] & \bullet\ar[d]  \\
      &&
      \bullet\ar[r]\ar[d] & \bullet\ar[r]\ar[d] & \bullet\ar[r]\ar[d] & \bullet\ar[d] && 
      \bullet\ar[r]\ar[d] & \bullet\ar[r]\ar[d] & \bullet\ar[r]\ar[d] & \bullet\ar[d] && 
      \bullet\ar[r]\ar[d] & \bullet\ar[r]\ar[d] & \bullet\ar[r]\ar[d] & \bullet\ar[d] &&\cdots \\
      \ar[uuurr]      &&
      \bullet\ar[r] & \bullet\ar[r] & \bullet\ar[r] & \bullet\ar[uuurr] &&
      \bullet\ar[r] & \bullet\ar[r] & \bullet\ar[r] & \bullet\ar[uuurr] &&
      \bullet\ar[r] & \bullet\ar[r] & \bullet\ar[r] & \bullet\ar[uuurr]  \\
    }
    \]
  \caption{The quiver of the repetitive category $\widehat{\L}_4$ in the case of linear
    orientation.}
    \label{wild example}
\end{figure}

\section{$d$-re\-pre\-sen\-ta\-tion-fin\-ite self-injective Nakayama algebras} \label{section: nakayama}

In this section, let $\Gamma$ be a self-injective Nakayama algebra over an arbitrary field $k$.
Our aim is to prove the following theorem, which gives a necessary and sufficient
condition for $\Gamma$ to be $d$-rep\-re\-sen\-ta\-tion-fi\-ni\-te.
\begin{thm} \label{nakayamaalg}
Let $\Gamma$ be a ring-indecomposable self-injective Nakayama $k$-algebra with $n$
isomorphism classes of simple modules and Loewy length $\ell\ge2$. Then $\Gamma$ is
$d$-re\-pre\-sen\-ta\-tion-fin\-ite if and only if at least one of the following two
conditions is satisfied:
\begin{enumerate}
\item $(\ell(d-1)+2)\mid 2n$\,;
\item $(\ell(d-1)+2)\mid tn, \quad\mbox{where}\quad t=\gcd(d+1,2(\ell-1))\,.$
\end{enumerate}
\end{thm}
In relation to Theorem~\ref{nakayamaalg}, we remark that the
  $d$-representation-finite Nakayama algebras of global dimension $d$ have recently been
  classified by Vaso \cite{vaso17}.

Our proof builds on the characterization of $d$-cluster-tilting objects in $d$-cluster
categories of type $A$ as $(d+1)$-angulations of regular polygons. Below,
we briefly recapitulate the necessary background, mostly from \cite{bm08}.

Let $H$ be a hereditary algebra. The $d$-\emph{cluster category} $\C_d(H)$ of $H$ (denoted
in \cite{bm08} by $\C^{d-1}_Q$, with $Q$ being the quiver $Q$ of $H$) of $H$ is defined as
the orbit category $\C_d(H)=\dm{H}/\nu_d$. The $d$-cluster category is triangulated
\cite{keller05}. It is easy to see that $d$-cluster-tilting subcategories of $\dm{H}$
correspond bijectively to basic $d$-cluster-tilting objects in $\C_d(H)$ (see
\cite[Proposition~2.2]{bmrrt06} for the case $d=2$).

Let $\ell$ and $d$ be positive integers, and set
\[N=(d-1)\ell+2=(d-1)(\ell-1)+(d+1).\]
Let $P_N$ be a regular $N$-gon with corners indexed by the numbers $0,\ldots,N-1$ in
the clockwise direction.
We denote by $[x,y]=[y,x]$ the edge between corners $x$ and $y$ of $P_N$.
A \emph{$(d-1)$-diagonal} of $P_N$ is a diagonal that dissects $P_N$ into a
$((d-1)\ell'+2)$-gon and a $((d-1)(\ell-\ell')+2)$-gon, for some $1< \ell'<\ell$.
An edge $[x,y]$ between corners $x$ and $y$ of $P_N$ is a $(d-1)$-diagonal if and only if
$$|y-x|>1\quad\mbox{and}\quad |y-x|-1\in(d-1)\Z.$$
A \emph{partial $(d+1)$-angulation} of $P_N$ is a set of non-crossing
  $(d-1)$-diagonals of $P_N$. The maximal (with respect to inclusion) elements of the set
  of partial $(d+1)$-angulations are the \emph{$(d+1)$-angulations} of $P_N$. 
A permutation $\rho$ of the set of partial $(d+1)$-angulations of $P_N$ is defined by
rotation one step in the anti-clockwise direction in $P_N$; i.e., 
$\rho([x,y]) = [x-1,y-1]$ (where the numbers are interpreted modulo $N$).

Let $K$ be an algebraically closed field, and $H$ the path algebra $K\mathbb{A}_{\ell-1}$
over $K$ of the quiver $\mathbb{A}_{\ell-1}$ of type $A_{\ell-1}$ with linear orientation.

\begin{prop}[{\cite[Proposition~5.5]{bm08}, \cite[Proposition 2.14]{murphy10}}] \label{Ccatbijection}
There is a bi\-jec\-tion between the set $\mathcal{X}$ of basic $d$-cluster-tilting objects of
$\C_d(H)$, and the set $\mathcal{Y}$ of $(d+1)$-angulations of $P_N$.
Under this bijection, the permutation of  $\mathcal{X}$ induced by the Auslander--Reiten
translation $\tau=\nu_1$ on $\C_d(H)$ corresponds to the permutation $\rho^{d-1}$ of
$\mathcal{Y}$. 
\end{prop}

We are now ready to embark upon the proof of Theorem~\ref{nakayamaalg}.
A first characterization of $d$-re\-pre\-sen\-ta\-tion-fin\-iteness of the algebra $\Gamma$ follows
readily from Proposition~\ref{Ccatbijection} together with our previous results. 

\begin{prop} \label{nakayamaprop}
  The algebra $\Gamma$ is $d$-re\-pre\-sen\-ta\-tion-fin\-ite if and only if there exists a
  $(d+1)$-angulation $\mathscr{T}$ of $P_N$ satisfying $\rho^{n(d-1)}(\mathscr{T})=\mathscr{T}$. 
\end{prop}

\begin{proof} 
Let $K$ be an arbitrary field, and $\Gamma'=K\tilde{\mathbb{A}}_{n-1}/I^\ell$, where $\tilde{\mathbb{A}}_{n-1}$ is a cyclically oriented
quiver of extended Dynkin type $\widetilde{A}_{n-1}$, and $I$ is the ideal of
$K\tilde{\mathbb{A}}_{n-1}$ generated by all arrows.
Recall that the indecomposable modules of a Nakayama algebra are uniserial, and determined up to isomorphism by their tops and lengths, and that the almost split sequences are described by these data, independently of the ground field (see the proof of Theorem~V:2.1 in \cite{ars95}).
In particular, the Auslander--Reiten quiver of $\Gamma$ is isomorphic to that of $\Gamma'$, and we have a bijection between the isomorphism classes of indecomposable $\Gamma$-modules and those of $\Gamma'$. 
This bijection preserves the (non-)vanishing of $\Ext^i$, and therefore $\Gamma$ is
$d$-re\-pre\-sen\-ta\-tion-fin\-ite if and only if $\Gamma'$ is
$d$-re\-pre\-sen\-ta\-tion-fin\-ite.
Thus we can assume that $k$ is algebraically closed, and
  $\Gamma=k\tilde{\mathbb{A}}_{n-1}/I^\ell$, without loss of generality.

As before, we write $H=k\mathbb{A}_{\ell-1}$ where $\mathbb{A}_{\ell-1}$ is linearly
oriented. 
Then the repetitive category $\widehat H$ of $H$ is equivalent to 
$k\tilde{\mathbb{A}}_{\infty}^{\infty}/I^\ell$, where $\tilde{\mathbb{A}}_{\infty}^{\infty}$
is linearly oriented of type $A_{\infty}^{\infty}$
and $I$ is the ideal generated by all arrows in $\tilde{\mathbb{A}}_{\infty}^{\infty}$.
Denoting by $\psi$ the automorphism of $\widehat H$ given by a shift one step to the
  left in $\tilde{\mathbb{A}}_{\infty}^{\infty}$, we get $\Gamma\simeq \widehat H/\psi^n$.

By Theorem~\ref{basic construction}, the algebra $\Gamma$ is $d$-re\-pre\-sen\-ta\-tion-fin\-ite if and
only if $\dm{H}$ has a $\psi^n$-equivariant $d$-cluster-tilting subcategory. 
In $\stmod\widehat{H}\simeq\dm{H}$, we have $\psi_*\simeq \nu_1$ -- the Auslander--Reiten translation on $\dm{H}$ --
and this functor descends to $\nu_1$ on the $d$-cluster category $\C_d(H)$. 
Since $d$-cluster-tilting subcategories in $\dm{H}$ correspond bijectively to basic
$d$-cluster-tilting objects in $\C_d(H)$, Proposition~\ref{Ccatbijection} now implies
that $\dm{H}$ has a $\psi^n$-equivariant $d$-cluster-tilting subcategory if and only if
there exists a $(d+1)$-angulation of $P_N$ that is invariant under
$\left(\rho^{d-1}\right)^n=\rho^{n(d-1)}$. 
\end{proof}

In view of Proposition~\ref{nakayamaprop}, Theorem~\ref{nakayamaalg} follows directly from
the following result.

\begin{prop} \label{combinatoricprop}
  There exists a $(d+1)$-angulation $\mathscr{T}$ of $P_N$ such that
  $\rho^{n(d-1)}(\mathscr{T})=\mathscr{T}$, if and only if at least one of the following
  conditions is satisfied: 
\begin{enumerate}
\item $N\mid 2n$\,;
\item $N\mid tn, \quad\mbox{where}\quad t=\gcd(d+1,2(\ell-1))\,.$
\end{enumerate}
\end{prop}

The remainder of this section is devoted to the proof of
Proposition~\ref{combinatoricprop}.
To abbreviate notation, we write $a\wedge b$ for $\gcd(a,b)$ and $a\vee b$ for
$\lcm(a,b)$, where $a,b\in\Z$.

\begin{lma} \label{centrangulation}
Let $q$ be a positive integer dividing $N$. 
The following statements are equivalent:
\begin{enumerate}
\item There exists a $(d+1)$-angulation $\mathscr{T}$ of $P_N$, containing a $(d+1)$-angle $G$,
  such that $\rho^{N/q}(\mathscr{T})=\mathscr{T}$ and $\rho^{N/q}(G)=G$;
  \label{tri}
\item $q\mid(\ell-1)\wedge(d+1)$.
  \label{r}
\end{enumerate}
\end{lma}

\begin{proof}
In the proof, the number $s=(d+1)/q$ plays a role. Since $N=(d-1)(\ell-1)+(d+1)$, we have $\frac{N}{q} - s = (d-1)\frac{\ell-1}{q}$.

\eqref{tri}$\Rightarrow$\eqref{r}:
Let $G_0\subset\{0,1,\ldots,N-1\}$ be the set of corners of $G$.
Without loss of generality, we may assume that $0\in G_0$ and hence $iN/q\in G_0$ for any $i$.
Setting $X=\left\{0,1,\ldots,N/q-1\right\}$, it follows that 
$G_0$ can be written as a disjoint union
$$G_0=\bigsqcup_{i=0}^{q-1}\rho^{iN/q}(G_0\cap X) \,.$$
Since $|G_0|=d+1$, this implies that $q$ divides $d+1$, and $|G_0\cap X|=(d+1)/q=s$. 

Let $G_0\cap X=\{x_1,\ldots,x_s\}$, where $0=x_1<x_2<\cdots<x_s$, and set $x_{s+1}=N/q$.
Then, for each $i=1,\ldots,s$, the edge $[x_i,x_{i+1}]$ is either an outer edge or a
$(d-1)$-diagonal, so $x_{i+1}-x_i-1\in(d-1)\Z$.
Consequently,
$$(d-1)\frac{\ell-1}{q} = \frac{N}{q} - s = x_{s+1} - x_1- s = \sum_{i=1}^s(x_{i+1}-x_i-1) \in(d-1)\Z$$ 
and thus $q$ divides $\ell-1$.
We have proved that $q\mid(\ell-1)$ and $q\mid(d+1)$, i.e., $q\mid(\ell-1)\wedge(d+1)$.

\eqref{r}$\Rightarrow$\eqref{tri}:
Let $G$ be the $(d+1)$-angle with corners
\[\bigsqcup_{i=0}^{q-1}\rho^{iN/q}\{0,1,\ldots,s-1\}.\]
Each edge of $G$ is either an outer edge or a $(d-1)$-diagonal,
since 
\[\frac{(i+1)N}{q} - \left(\frac{iN}{q}+s-1\right)-1 = \frac{N}{q} - s = (d-1)\frac{\ell-1}{q} \in (d-1)\Z \,.\]
For each $i=0,\ldots,q-1$, the edge $[iN/q +s-1,(i+1)N/q]$ of $G$ cuts out an
  $N'$-gon $P^{(i)}_{N'}$ with corners $\{iN/q +s-1,iN/q+s ,\ldots,(i+1)N/q\}$, where
  $N'=(d-1)\frac{\ell-1}{q}+2$. Thus $P_N$ is partitioned into $G$ and $P^{(i)}_{N'}$,
  $i=0,\ldots, q-1$, and $P^{(i)}_{N'}=\rho^{iN/q}\left(P^{(0)}_{N'}\right)$ for each $i$.
Any $(d+1)$-angulation $\mathscr{T}'$ of $P^{(0)}_{N'}$ gives a $(d+1)$-angulation
\[\mathscr{T}=\bigsqcup_{i=0}^{q-1}\rho^{iN/q}(\mathscr{T}'\sqcup\{[s-1,N/q]\})\]
of $P_N$, which is clearly $\rho^{N/q}$-invariant. 
This proves the statement \eqref{tri}.
\end{proof}

\begin{lma}\label{disangulation}
  Assume that $d\ge2$. 
  The following statements are equivalent:
  \begin{enumerate}
  \item there exists a $(d+1)$-angulation $\mathscr{T}$ of $P_N$ for which the centre point of
    $P_N$ lies on a $(d-1)$-diagonal, and $\rho^{N/2}(\mathscr{T})=\mathscr{T}$;
  \item $2\mid \ell$. 
  \end{enumerate}
\end{lma}

\begin{proof}
A $(d+1)$-angulation of the specified type exists if and only if there exists a
$(d-1)$-diagonal passing through the centre of $P_N$. This is equivalent to $N$ being even
and the edge $[0,N/2]$ a $(d-1)$-diagonal. Since $\frac{N}{2}-1 = \frac{(d-1)\ell}{2}$, this is equivalent to $2\mid \ell$.
\end{proof}

\begin{lma} \label{lcm}
Let $r=(\ell-1)\wedge(d+1)$, $t=(d+1)\wedge2(\ell-1)$, and $n\in\Z$. Then 
$N\mid nr(d-1)$ is equivalent to $N\mid nt$.
\end{lma}

\begin{proof}
In view of the computation
\begin{equation*} 
  \begin{split}
  N\wedge r(d-1)&=N\wedge[(\ell-1)\wedge(d+1)](d-1) \\
  &= [(d-1)(\ell-1)+(d+1)]\wedge (\ell-1)(d-1)\wedge(d+1)(d-1) \\
  &= (d+1)\wedge (\ell-1)(d-1)\wedge(d+1)(d-1) = (d+1)\wedge (\ell-1)(d-1) \\
  &= (d+1)\wedge[(d+1)(\ell-1) -2(\ell-1)] = (d+1)\wedge2(\ell-1) = t \,,
  \end{split}
\end{equation*}
the numbers $N/t$ and $r(d-1)/t$ are integers and coprime to each other. 
Thus we have the following chain of equivalences:
\begin{equation*} N\mid nr(d-1) \;\Leftrightarrow\; \frac{N}{t}\mid n\frac{r(d-1)}{t} \;\Leftrightarrow\;
  \frac{N}{t}\mid n \;\Leftrightarrow\;
  N\mid nt \,.\qedhere\end{equation*}
\end{proof}

\begin{proof}[Proof of Proposition~\ref{combinatoricprop}]
First, we prove that $N\mid tn$ if and only if $P_N$ has a $\rho^{n(d-1)}$-invariant
$(d+1)$-angulation that contains a $(d+1)$-angle $G$ such that $\rho^{n(d-1)}(G)=G$.
Define an integer $q$ by $N/q=N\wedge n(d-1)$. Then a $(d+1)$-angulation (respectively,
$(d+1)$-angle) of $P_N$ is $\rho^{n(d-1)}$-invariant if and only if it is
$\rho^{N/q}$-invariant. 
By Lemma~\ref{centrangulation}, the existence of a $(d+1)$-angulation $\mathscr{T}$
containing a $(d+1)$-angle $G$ such that $\rho^{N/q}(\mathscr{T})=\mathscr{T}$ and
$\rho^{N/q}(G)=G$ is equivalent to $q\mid r$, where $r=(\ell-1)\wedge(d+1)$.
Since $q$ and $\frac{n(d-1)}{N/q}$ are coprime, we have
\[q\mid r \;\Leftrightarrow\;
q\mid r\left(\frac{n(d-1)}{N/q}\right) \;\Leftrightarrow\; N\mid nr(d-1)\;\Leftrightarrow\; N\mid tn,\]
where the last equivalence comes from Lemma~\ref{lcm}. Thus the claim holds.

Next, we show that if $P_N$ has a $\rho^{n(d-1)}$-invariant $(d+1)$-angulation
$\mathscr{T}$ such that $\rho^{n(d-1)}(G)\ne G$ for every $(d+1)$-angle $G$, then 
$N\mid2n$. 
Observe that such a $\mathscr{T}$ must contain an edge that passes through the centre
point of $P_N$ and, consequently, this edge is fixed by $\rho^{n(d-1)}$. 
It follows that $n(d-1)$ is a multiple of $N/2$ but, since $\rho^{n(d-1)}\ne\I$, not of
$N$. Hence $\rho^{n(d-1)}=\rho^{N/2}$ which implies, by Lemma~\ref{disangulation}, that
the number $\ell$ is even. Thus $d-1$ and $N/2=(d-1)\ell/2 + 1$ are
coprime, so $\frac{N}{2}\mid n$ and consequently $N\mid 2n$ hold.

Finally, we prove that if $N\mid 2n$ then there exists a $\rho^{n(d-1)}$-invariant
triangulation of $P_N$. 
If $\ell$ is an even number, then, by Lemma~\ref{disangulation}, there exists
a $\rho^{N/2}$-invariant $(d+1)$-angulation $\mathscr{T}$ of $P_N$. 
The assumption $N\mid 2n$ implies that $\rho^{n}$ is a power of
$\rho^{N/2}$, and hence $\rho^{n(d-1)}(\mathscr{T})=\mathscr{T}$. 

Suppose that $\ell$ is an odd number.
If $t$ is even, then $N$ divides $tn$ and we are done.
Assume instead that $t$ is odd.
This implies that $d+1$ is odd, whence $N=(d-1)\ell+2$ is odd, too. 
Therefore, $N\mid2n$ implies $N\mid n$ and thus $N\mid n(d-1)$.
Since every $(d+1)$-angulation of $P_N$ is invariant under
$\rho^N=\I$, the result follows.
\end{proof}

\section{Open problems}\label{further}
Here, we point to some directions of further inquiry, and give some partial results.

In view of Corollaries~\ref{trivext}, \ref{fracCY} and \ref{trivrf}, it is natural to seek to
understand for which numbers $n$ and $d$ the algebra $T_n(\L)$ is
$d$-re\-pre\-sen\-ta\-tion-fin\-ite. 
To this end, we pose the following questions.

\begin{question} \renewcommand{\theenumi}{\arabic{enumi}}
For any finite-dimensional $k$-algebra $\L$, set 
$$\RF(\L)=\{(n,d)\in\Z_{>0}\!\times\!\Z_{>0}\mid 
T_n(\L)\mbox{ is $d$-re\-pre\-sen\-ta\-tion-fin\-ite}\}.$$
\begin{enumerate}
\item Given an algebra $\L$, describe the set $\RF(\L)$. \label{setRF}
\item 
  Is $\L$ twisted fractionally Calabi--Yau whenever $\RF(\L)$ is
  non-empty? \label{fraccyRF}
\end{enumerate}
\end{question}

By Corollary~\ref{fracCY}, we know that the converse of \eqref{fraccyRF} is true.
With respect to \eqref{setRF}, we hypothesize that for given $\L$ and $n\in\Z_{>0}$, there
are at most finitely many positive integers $d$ such that $(n,d)\in\RF(\L)$.
More generally, we make the following conjecture.
\begin{conj}
For a given finite-dimensional $k$-algebra $A$, there are only finitely many integers $d$
such that $A$ is $d$-re\-pre\-sen\-ta\-tion-fin\-ite.
\end{conj}
This expectation stems from the general phenomenon that, for a given $A$,
the number of indecomposable summands in a $d$-cluster-tilting $A$-modules tends to become
smaller as $d$ increases.

If a self-injective algebra $\Gamma$ can be written as an orbit algebra $\widehat{\L}/\phi$
for some algebra $\L$ of finite global dimension, then Theorem~\ref{stronger theorem},
together with the equivalence $\dml\simeq\stmod\widehat{\L}$ \eqref{happel equivalence}
gives a strong connection between $d$-cluster-tilting subcategories of $\mod\Gamma$ and
$\dml$. A natural question is therefore to which extent this construction exhausts all
$d$-re\-pre\-sen\-ta\-tion-fin\-ite self-injective algebras.

\begin{question}
  Assume that $k$ is an algebraically closed field of characteristic $0$ (or sufficiently
  large). Is any $d$-re\-pre\-sen\-ta\-tion-fin\-ite self-injective $k$-algebra $\Gamma$ isomorphic
  to $\widehat{\L}/\phi$ for some $\nu_d$-finite algebra $\L$ of finite global
  dimension and admissible automorphism $\phi$ of $\widehat{\L}$?
\end{question}

Whenever $\Gamma\simeq \widehat{\L}/\phi$ as above, Corollary~\ref{bijection} gives a
bijection between, on one hand, $d$-cluster-tilting $\Gamma$-modules and, on the other,
$\phi$-equivariant locally bounded $d$-cluster-tilting subcategories $\U$ of
$\dml\simeq\stmod\widehat{\L}$.

Recall that for a $\nu_d$-finite algebra $\L$, every tilting object $T\in\dml$
satisfying $\gldim(\End_{\dml}(T))\le d$ gives rise to an orbital
$d$-cluster-tilting subcategory $\U_d(T)$ of $\dml$. However, as we shall see below, this
construction is not exhaustive in general. 
Therefore, understanding the structure of $d$-clus\-ter-til\-ting subcategories
of $\dml$ is an important step towards understanding $d$-re\-pre\-sen\-ta\-tion-fin\-ite\-ness
for orbit algebras $\widehat{\L}/\phi$.

\begin{question} \label{UendT}
Given an algebra $\L$ of finite global dimension, which $d$-cluster-tilting
  subcategories of $\dml$ are orbital?
In particular, for which pairs $(\L,d)$ are all $d$-cluster-tilting subcategories of $\dml$
orbital?
\end{question}

The following observation gives a partial answer to Question~\ref{UendT}. 

\begin{prop} \label{2ct}
  Let $\L$ be an iterated tilted algebra. Then every $2$-cluster-tilting subcategory of
  $\dml$ is orbital.
\end{prop}

\begin{proof}
If $\L$ is iterated tilted then it is derived equivalent to some hereditary algebra $H$. 
Let $\C(H)=\dm{H}/\nu_2$ be the cluster category \cite{bmrrt06} of $H$, and
$\pi:\dm{H}\to\C(H)$ the natural functor. 
Suppose that $\U\subset\dml\simeq\dm{H}$ is a $2$-cluster-tilting subcategory.
From \cite[Proposition~2.2]{bmrrt06} follows that $\U=\pi\inv(V)$ for some cluster-tilting
object $V\in\C(H)$, and \cite[Theorem~3.3(a)]{bmrrt06} gives that $V = \pi(T)$, where
$T\in\dm{H}$ is a tilting complex with $\gldim(\End_{\dm{H}}(T))\le2$.
Hence, $\U=\pi\inv(\pi(T))=\U_2(T)$.  
\end{proof}

Let $k$ be an algebraically closed field of characteristic different from $2$, and
$\Gamma$ a finite-di\-men\-sion\-al representation-finite self-injective $k$-algebra.

By Theorem~\ref{riedtmann}, $\Gamma$ is isomorphic to $\widehat\L/\phi$ for some
re\-pre\-sen\-ta\-tion-fin\-ite tilted algebra $\L$ and autoequivalence $\phi$ of $\widehat\L$.
Since $\L$ is a tilted algebra, there exists a triangle equivalence
$\dml\simeq\dm{H}$ for some hereditary re\-pre\-sen\-ta\-tion-fin\-ite algebra $H$.
The following is an immediate consequence of Proposition~\ref{2ct} and 
Theorem~\ref{basic construction}.

\begin{cor} \label{rf-nrfsi}
  In the above setting, the algebra $\Gamma$ is $2$-re\-pre\-sen\-ta\-tion-fin\-ite if and only if
  $\dml$ has a $\phi$-equivariant orbital $2$-cluster-tilting subcategory.
\end{cor}

The following two examples show that neither the assumption that $d=2$, nor that $\L$ is
iterated tilted, can be removed from Proposition~\ref{2ct}.

\begin{ex}
Let $\L$ be the path algebra a quiver of type $A_2$. 
Then $\dml$ has a $3$-cluster-tilting subcategory (indicated with black dots in the figure
below) that is not orbital.
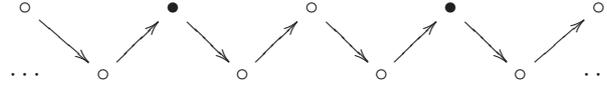
\[\xymatrix@R1.5em@C1.5em{
\circ\ar[rd]&&\bullet\ar[rd]&&\circ\ar[rd]&&\bullet\ar[rd]&&\circ\\
\cdots&\circ\ar[ru]&&\circ\ar[ru]&&\circ\ar[ru]&&\circ\ar[ru]&\cdots
}\]
\end{ex}

To any algebra $\L$ of global dimension at most $2$ is associated a cluster category
$\C(\L)$. The cluster category is triangle equivalent to the Verdier quotient 
$\mathop{\rm per}\Gamma/\dm{\Gamma}$ where $\Gamma=\Gamma(Q_\L,W_\L)$ is the Ginzburg dg
algebra  of the quiver $(Q_\L,W_\L)$ with the potential constructed in 
Section~\ref{section: higher preprojective}. There is a fully faithful functor
$\dml/\nu_2\to\C(\L)$ \cite[Section 4.3]{amiot09} giving rise to an isomorphism
\begin{equation}\label{End is Pi3}
\End_{\C(\L)}(\L)\simeq\bigoplus_{i\in\Z}\Hom_{\dml}(\L,\nu_2^{-i}(\L))=\Pi_3(\L)
\end{equation}
of $\Z$-graded $k$-algebras.

From the following result, we can infer the existence of algebras $\L$ for which the
bounded derived category $\dml$ contains a non-orbital $2$-cluster-tilting subcategory.

\begin{prop} \label{tiltit}
Let $\L$ be an algebra of global dimension $2$ that is cluster-equivalent to
a hereditary algebra $H$. 
If every $2$-cluster-tilting subcategory of $\dml$ is orbital, then $\L$ is derived
equivalent to $H$.
\end{prop}

Note that $\nu_2$-finiteness is equivalent to the cluster category being Hom-finite, hence
this property is preserved by cluster equivalence. 

\begin{proof} 
Let $F:\C(\L)\to\C(H)$ be a triangle equivalence, and $\pi_\L:\dml\to\C(\L)$ and
$\pi_H:\dm{H}\to\C(H)$ the natural functors.
The object $V=F\inv\pi_H(H)\in\C(\L)$ is a $2$-cluster-tilting object
and hence
$\U=\add(\pi_{\L}\inv(V))$ is a $2$-cluster-tilting subcategory of $\dml$
\cite[Propositions~3.1, 3.2]{ao14}.
Take $T\in\dml$ to be a tilting complex such that $\U=\U_2(T)$, and set $\L'=\End_{\dml}(T)$.
Then $\L$ and $\L'$ are derived equivalent and hence cluster equivalent. 
As $V=\pi_\L(T)$, by \eqref{End is Pi3} we get
\[H=\End_{\C(H)}(\pi_H(H))\simeq\End_{\C(\L)}(V)\simeq\End_{\C(\L')}(\pi_{\L'}(\L'))\simeq\Pi_3(\L')\,.\]
Since $H$ is hereditary, it follows from \eqref{Jacobi is preprojective} that 
so is the degree zero part $(\Pi_3(\L'))_0=\L'$ of $\Pi_3(\L')$.
Thus $\Pi_3(\L')=\L'$ and hence $H=\L'=\End_{\dml}(T)$ holds, so $H$ and $\L$ are derived
equivalent.
\end{proof}

Let $\L$ be an algebra of global dimension $2$ that is cluster equivalent
to a hereditary algebra but not iterated tilted. Then Proposition~\ref{tiltit} implies
that the derived category $\dml$ contains a non-orbital $2$-cluster-tilting subcategory.
Examples of such algebras $\L$ can be found in \cite{ao13}, in which a classification is
given of all algebras of global dimension $2$ that are cluster equivalent to a hereditary
algebra of extended Dynkin type $\widetilde{A}$.

\begin{ex} \label{tiltitnot} \cite[Example~8.8]{ao14}
Let $H=kQ$ and $\L=kQ'/I$ be given as below:
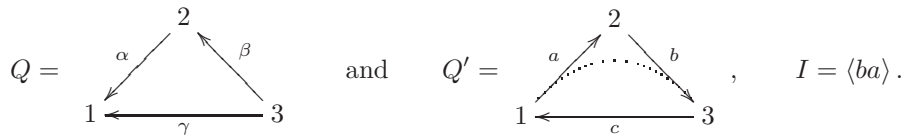
\[
Q = \vcentered{
\xymatrix{
  &2\ar[dl]_{\alpha} \\
  1&& 3\ar[ul]_{\beta}\ar[ll]^{\gamma}
}}
\qquad\mbox{and}\qquad
Q' = \vcentered{\xymatrix{
&2\ar[dr]^b\\
1\ar[ur]^a\ar@{..}@/^7.5mm/[rr]&&3\ar[ll]^c
}}
\;,\qquad
I = \langle ba\rangle \,.
\]
The associated quivers with potential are $(\widetilde{Q},W) = (Q,0)$ and
$(\widetilde{Q'},W')$, where

\[ \widetilde{Q'} = \vcentered{\xymatrix{
  &2\ar[dr]^(.4)b&\\ 1\ar[ur]^(.6)a&&3\ar@<1mm>[ll]^d\ar@<-1mm>[ll]_c}}
\;,\qquad W'=dba \,.
\]
Now, $(\widetilde{Q'},W')$ and $(\widetilde{Q},W)$ are related by a mutation
at the vertex $2$; hence, 
\[\C(\Lambda)=\C(\widetilde{Q'},W')\simeq\C(\widetilde{Q},W)=\C(H)\]
holds \cite{amiot09,ky11}.
Since the quiver $Q'$ is not acyclic, the algebra $\L=kQ'/I$ is not iterated tilted. 
By Proposition~\ref{tiltit}, it follows that $\dml$ has a non-orbital $2$-cluster-tilting
subcategory.

The $2$-cluster-tilting subcategory $\U=\pi_\L\inv\pi_H(H)$ of $\dml$, constructed in
  the proof of Proposition~\ref{tiltit}, is given in Figure~\ref{nonorbitalCT}.

\begin{figure}[h]
\[ 
\xymatrix@R=.25cm@C=.53cm{
&&\bullet &&\bullet &&\bullet &&\bullet \\
\cdots&\bullet\ar[ur] &&\bullet\ar[ur] &&\bullet\ar[ur] &&\bullet\ar[ur]&\cdots  \\
\bullet\ar[ur]&& \bullet\ar[ur]\ar[uu] &&\bullet\ar[ur]\ar[uu] &&\bullet\ar[ur]\ar[uu] 
}\]
\caption{The $2$-cluster-tilting subcategory $\U=\pi_\L\inv\pi_H(H)$ of $\dml$.
\label{nonorbitalCT}}
\end{figure}
\end{ex}

\section*{Acknowledgements}

The authors are indebted to Karin Erdmann, Andrzej Skowro\'nski and Kunio Yamagata for
helpful comments and suggestions about the paper. They also wish to thank Claire Amoit for
drawing their attention to the example~6.9, and an anonymous referee for helpful comments.

\def\cprime{$'$}


\begin{thebibliography}{11}

\bibitem{amiot09}
C. Amiot.
\newblock Cluster categories for algebras of global dimension 2 and quivers with potential.
\newblock Ann. Inst. Fourier (Grenoble) 59(6):2525--2590, 2009.

\bibitem{ao13}
C. Amiot and S. Oppermann.
\newblock Algebras of acyclic cluster type: tree type and type {$\widetilde{A}$}.
\newblock {\em Nagoya Math. J.}, 211:1--50, 2013.

\bibitem{ao}
C. Amiot, S. Oppermann.
\newblock The image of the derived category in the cluster category.
\newblock {\em Int. Math. Res. Not. IMRN} 4:733--760, 2013.

\bibitem{ao14}
C. Amiot and S. Oppermann.
\newblock Cluster equivalence and graded derived equivalence.
\newblock {\em Doc. Math.}, 19:1155--1206, 2014.

\bibitem{ABa} J.~Arias, E.~Backelin.
\newblock Higher Auslander--Reiten sequences and t-structures.
\newblock {\em J. Algebra} 459 (2016), 280--308.

\bibitem{asai18}
S. Asai.
\newblock The {G}rothendieck groups and stable equivalences of mesh algebras.
\newblock {\em Algebr Represent Theor}, 21:635--681, 2018.

\bibitem{asashiba97}
H. Asashiba. 
\newblock A covering technique for derived equivalence.
\newblock {\em J. Algebra} 191 (1997), 382--415.

\bibitem{ass06}
I. Assem, D. Simson, and A. Skowro{\'n}ski.
\newblock {\em Elements of the representation theory of associative algebras.
  {V}ol. 1}, volume~65 of {\em London Mathematical Society Student Texts}.
\newblock Cambridge University Press, Cambridge, 2006.
\newblock Techniques of representation theory.

\bibitem{a66}
M.~Auslander.
\newblock Coherent functors.
\newblock 1966 Proc. Conf. Categorical Algebra (La Jolla, Calif., 1965) pp. 189--231 Springer, New York.

\bibitem{a71}
M.~Auslander.
\newblock Representation dimension of Artin algebras,
in Selected works of Maurice Auslander. Part 1. 
Edited and with a foreword by Idun Reiten, Sverre O. Smal\o, and \O yvind Solberg.
\newblock American Mathematical Society, Providence, RI, 1999.


\bibitem{ARstable}
M.~Auslander and I.~Reiten.
\newblock Stable equivalence of dualizing $R$-varieties.
\newblock {\em Adv. Math.} 12:306--366, 1974.

\bibitem{ars95}
M. Auslander, I. Reiten, and S.~O. Smal{\o}.
\newblock {\em Representation theory of {A}rtin algebras}, volume~36 of {\em
  Cambridge Studies in Advanced Mathematics}.
\newblock Cambridge University Press, Cambridge, 1995.

\bibitem{bm08}
K.~Baur and R.~J. Marsh.
\newblock A geometric description of {$m$}-cluster categories.
\newblock {\em Trans. Amer. Math. Soc.}, 360(11):5789--5803, 2008.

\bibitem{bl14}
R.~Bautista and S.~Liu.
\newblock Covering theory for linear categories with application to derived categories.
\newblock {\em J. Algebra} 406:173--225, 2014.


\bibitem{bgrs85}
R.~Bautista, P.~Gabriel, A.~V.~Roiter, L.~Salmeron.
\newblock Representation-finite algebras and multiplicative bases. 
\newblock {\em Invent. Math.} 81 (1985), no. 2, 217--285. 

\bibitem{bg81}
K.~Bongartz, P.~Gabriel.
\newblock Covering spaces in representation theory. 
\newblock {\em Invent. Math.} 65 (1981/82), 331--378.

\bibitem{br86}
S.~Brenner.
\newblock A combinatorial characterisation of finite Auslander--Reiten quivers.
\newblock Representation theory, I (Ottawa, Ont., 1984), 13--49, Lecture Notes in Math., 1177, Springer, Berlin, 1986.

\bibitem{BLR}
O.~Bretscher, C.~L{\"a}ser and C.~ Riedtmann.
\newblock {Self-injective and simply connected algebras},
\newblock {\em Manuscripta Math.} 36(3):253--307, 1981/82.

\bibitem{bmrrt06}
A.~Bakke~Buan, R.~Marsh, M.~Reineke, I.~Reiten, and G.~Todorov.
\newblock Tilting theory and cluster combinatorics.
\newblock {\em Adv. Math.} 204(2):572--618, 2006.

\bibitem{cr90}
C.~W. Curtis and I. Reiner.
\newblock {\em Methods of representation theory. {V}ol. {I}}.
\newblock Wiley Classics Library. John Wiley \& Sons, Inc., New York, 1990.
\newblock With applications to finite groups and orders, Reprint of the 1981
  original, A Wiley-Interscience Publication.

\bibitem{DWZ}
H.~Derksen, J.~Weyman and A.~Zelevinsky.
\newblock {Quivers with potentials and their representations. I. Mutations},
\newblock {\em Selecta Math. (N.S.)} 14 (2008), no. 1, 59--119.

\bibitem{egl19}
K. Erdmann, S. Gratz, and L. Lamberti.
\newblock Cluster tilting modules for mesh algebras.
\newblock arXiv:1904.01369, 2019.

\bibitem{eh08}
K. Erdmann and T. Holm.
\newblock Maximal {$n$}-orthogonal modules for selfinjective algebras.
\newblock {\em Proc. Amer. Math. Soc.}, 136(9):3069--3078, 2008.

\bibitem{gabriel}
P.~Gabriel.
\newblock Unzerlegbare {D}arstellungen. {I}.
\newblock {\em Manuscripta Math.} 6:71--103, 1972.

\bibitem{gabriel2}
P.~Gabriel, 
\newblock The universal cover of a representation-finite algebra, 
\newblock Representations of algebras (Puebla, 1980),  pp. 68--105, Lecture Notes in Math., 903, Springer, Berlin-New York, 1981.

\bibitem{gabriel-roiter}
P.~Gabriel, A.~V.~Ro{\u\i}ter.
\newblock Representations of finite-dimensional algebras.
\newblock {\em Springer, Berlin}, 1997.

\bibitem{gls06}
C. Gei{\ss}, B. Leclerc, and J. Schr\"{o}er.
\newblock Rigid modules over preprojective algebras.
\newblock {\em Invent. Math.}, 165(3):589--632, 2006.

\bibitem{gls07}
C. Gei{\ss}, B. Leclerc, and J. Schr\"{o}er.
\newblock Auslander algebras and initial seeds for cluster algebras.
\newblock {\em J. Lond. Math. Soc. (2)}, 75(3):718--740, 2007.

\bibitem{guo13} J.~Guo.
\newblock McKay quivers and absolute $n$-complete algebras.
\newblock {\em Sci. China Math.} 56 (2013), no. 8, 1607--1618.


\bibitem{happel88}
D.~Happel.
\newblock {Triangulated categories in the representation theory of
finite-dimensional algebras.}
volume 119 of {\em London Mathematical Society
Lecture Notes Series}.
\newblock Cambridge University Press, Cambridge, 1988.

\bibitem{hi11a}
M.~Herschend and O.~Iyama.
\newblock $n$-representation-finite algebras and twisted fractionally
  {C}alabi-{Y}au algebras.
\newblock {\em Bull. London Math. Soc.}, 43:449--466, 2011.

\bibitem{hi11b}
M.~Herschend and O.~Iyama.
\newblock Selfinjective quivers with potential and $2$-representation-finite
  algebras.
\newblock {\em Compositio Math.}, 147:1885--1920, 2011.

\bibitem{himo17}
M.~Herschend, O.~Iyama, H.~Minamoto, and S.~Oppermann.
\newblock Representation theory of {G}eigle--{L}enzing complete intersections.
\newblock arXiv:1409.0668v2, 2017, to appear in Mem. Amer. Math. Soc.

\bibitem{HIO} M.~Herschend, O.~Iyama, S.~Oppermann.
\newblock $n$-representation infinite algebras.
\newblock {\em Adv. Math.} 252 (2014), 292--342.


\bibitem{HZ} Z.~Huang, X.~Zhang.
\newblock Higher Auslander algebras admitting trivial maximal orthogonal subcategories.
\newblock {\em J. Algebra} 330 (2011), 375--387.

\bibitem{hw83}
D.~Hughes and Waschb\"usch.
\newblock Trivial extensions of tilted algebras.
\newblock {\em Proc. London Math. Soc.}, 46:347--364, 1983.

\bibitem{it84}
K.~Igusa, G.~Todorov.
\newblock A characterization of finite Auslander--Reiten quivers.
\newblock {\em J. Algebra} 89 (1984), no. 1, 148--177.

\bibitem{iyama05}
O.~Iyama.
\newblock The relationship between homological properties and representation theoretic
realization of Artin algebras.
\newblock {\em Trans. Amer. Math. Soc.} 357 (2005), no. 2, 709--734.  

\bibitem{iyama07}
O.~Iyama.
\newblock Higher-dimensional Auslander--Reiten theory on maximal orthogonal subcategories.
\newblock {\em Adv. Math.}, 210 (2007), no. 1, 22--50.

\bibitem{iyama07b}
O.~Iyama.
\newblock Auslander correspondence.
\newblock {\em Adv. Math.} 210 (2007), no. 1, 51--82.

\bibitem{iyama11}
O.~Iyama.
\newblock Cluster tilting for higher {A}uslander algebras.
\newblock {\em Adv. Math.}, 226:1--61, 2011.

\bibitem{io11}
O.~Iyama and S.~Oppermann.
\newblock $n$-representation-finite algebras and $n$-{APR} tilting.
\newblock {\em Trans. AMS}, 363(12):6575--6614, 2011.

\bibitem{io13}
O.~Iyama and S.~Oppermann.
\newblock Stable categories of higher preprojective algebras.
\newblock {\em Adv. Math.} 244 (2013), 23--68.

\bibitem{IW} O.~Iyama, M.~Wemyss.
\newblock Maximal modifications and Auslander--Reiten duality for non-isolated
singularities.
\newblock {\em Invent. Math.} 197 (2014), no. 3, 521--586.

\bibitem{iy08}
O.~Iyama, Y.~Yoshino.
\newblock Mutation in triangulated categories and rigid {C}ohen-{M}acaulay modules.
\newblock {\em Invent. Math.}, 172(1):117--168, 2008.


\bibitem{jasso15a}
G. Jasso.
\newblock {$\tau^2$}-stable tilting complexes over weighted projective lines.
\newblock {\em Adv. Math.}, 273:1--31, 2015.

\bibitem{jasso16}
G. Jasso.
\newblock {$n$}-abelian and {$n$}-exact categories.
\newblock {\em Math. Z.}, 283(3-4):703--759, 2016.

\bibitem{JK} G.~Jasso, J.~K\"ulshammer.
  \newblock Higher Nakayama algebras I: Construction.
  \newblock {\em Adv. Math.}, 351:1139--1200, 2019.

\bibitem{Jo}
P. J{\o}rgensen.
\newblock Torsion classes and t-structures in higher homological algebra.
\newblock {\em Int. Math. Res. Not. IMRN}, (13):3880--3905, 2016.

\bibitem{keller05}
B. Keller.
\newblock On triangulated orbit categories.
\newblock {\em Doc. Math.}, 10:551--581, 2005.


\bibitem{keller11}
B. Keller.
\newblock Deformed Calabi-Yau completions. With an appendix by Michel Van den Bergh.
\newblock {\em J. Reine Angew. Math.} 654 (2011), 125--180.

\bibitem{ky11}
B. Keller and D. Yang.
\newblock Derived equivalences from mutations of quivers with potential.
\newblock {\em Adv. Math.}, 226(3):2118--2168, 2011.

\bibitem{leszczynski94}
Z. Leszczy\'{n}ski.
\newblock On the representation type of tensor product algebras.
\newblock {\em Fund. Math.}, 144(2):143--161, 1994.

\bibitem{my01}
J. Miyachi and A. Yekutieli.
\newblock Derived {P}icard groups of finite-dimensional hereditary algebras.
\newblock {\em Compositio Math.}, 129(3):341--368, 2001.

\bibitem{Miz} Y.~Mizuno.
\newblock A Gabriel-type theorem for cluster tilting.
\newblock {\em Proc. Lond. Math. Soc.} (3) 108 (2014), no. 4, 836--868.

\bibitem{murphy10}
G.~J. Murphy.
\newblock Derived equivalence classification of {$m$}-cluster tilted algebras
  of type {$A_n$}.
\newblock {\em J. Algebra}, 323(4):920--965, 2010.

\bibitem{OT} S.~Oppermann, H.~Thomas.
\newblock Higher-dimensional cluster combinatorics and representation theory.
\newblock {\em  J. Eur. Math. Soc. (JEMS)} 14 (2012), no. 6, 1679--1737.

\bibitem{pasquali17} A.~Pasquali. 
\newblock Tensor products of higher almost split sequences.
\newblock {\em J. Pure Appl. Algebra} 221 (2017), no. 3, 645--665.

\bibitem{RV}
I.~Reiten, M.~Van den Bergh.
\newblock Noetherian hereditary abelian categories satisfying Serre duality.
\newblock {\em J. Amer. Math. Soc.} 15 (2002), no. 2, 295--366.

\bibitem{riedtmann80}
C.~Riedtmann.
\newblock Representation-finite self-injective algebras of class $A_{n}$.
\newblock In {\em Representation theory, {II} ({P}roc. {S}econd {I}nternat.
              {C}onf., {C}arleton {U}niv., {O}ttawa, {O}nt., 1979)},
\newblock {Lecture Notes in Math.} 832, pages 449--520. Springer, Berlin, 1980.

\bibitem{riedtmann83}
C.~Riedtmann.
\newblock Representation-finite self-injective algebras of class $D_{n}$.
\newblock {\em Compositio Math.} 49(2):231--282, 1983.

\bibitem{skowronski06}
A. Skowro\'nski.
\newblock Selfinjective algebras: finite and tame type.
\newblock In {\em Trends in representation theory of algebras and related
  topics}, volume 406 of {\em Contemp. Math.}, pages 169--238. Amer. Math.
  Soc., Providence, RI, 2006.
		
\bibitem{sy08}
A.~Skowro{\'n}ski and K.~Yamagata.
\newblock Selfinjective algebras of quasitilted type.
\newblock In {\em Trends in representation theory of algebras and related
topics}, EMS Ser. Congr. Rep., pages 639--708. Eur. Math. Soc., Z\"urich,
2008.

\bibitem{tachikawa80}
H.~Tachikawa.
\newblock Representations of trivial extensions of hereditary algebras.
\newblock In {\em Representation theory, {II} ({P}roc. {S}econd {I}nternat.
  {C}onf., {C}arleton {U}niv., {O}ttawa, {O}nt., 1979)}, Lecture Notes in
Math. 832, pages 579--599, Springer, Berlin, 1980


\bibitem{vaso17}
L.~Vaso.
\newblock {\em $n$-cluster tilting subcategories of representation-directed algebras}.
\newblock {\em J. Pure Appl. Algebra}, 223(5):2101--2122, 2019.

\bibitem{waschbusch}
J.~Waschb\"usch.
\newblock Symmetrische {A}lgebren vom endlichen {M}odultyp.
\newblock {\em J. Reine Angew. Math.} 321:78--98, 1981.

\bibitem{waschbusch2}
J.~Waschb\"usch.
\newblock On self-injective algebras of finite representation type.
\newblock In {\em Monograf\'\i as del Instituto de Matem\'aticas [Monographs of
              the Institute of Mathematics]} 14, pages i+59. Universidad Nacional
Aut\'onoma de M\'exico, M\'exico, 1983.

\bibitem{yamagata81}
K.~Yamagata.
\newblock Extensions over hereditary Artinian rings with self-dualities. I. 
\newblock {\em J. Algebra.} 73(2):386--433, 1981.

\bibitem{yamaura}
K.~Yamaura.
\newblock Realizing stable categories as derived categories.
\newblock {\em Adv. Math.} 248:784--819, 2013.

\end{thebibliography}
\end{document}